\documentclass[12pt]{amsart}
\usepackage{graphicx,color}
\usepackage{colordvi}
\usepackage{amssymb, mathrsfs, amsfonts, amsmath}
\usepackage{latexsym}

\newtheorem{theorem}{Theorem}[section]
\newtheorem{lemma}[theorem]{Lemma}
\newtheorem{corollary}[theorem]{Corollary}
\newtheorem{proposition}[theorem]{Proposition}
\newtheorem{definition}[theorem]{Definition}
\newtheorem{notation}[theorem]{Notation}
\newtheorem{example}[theorem]{Example}
\newtheorem{remark}[theorem]{Remark}

\newcommand{\R}{\mathbb R}
\newcommand{\Q}{\mathbb Q}

\newcommand{\F}{\mathbb F}

\newcommand{\M}{\mathcal M}

\newcommand{\K}{\mathcal{K}}

\begin{document}

\title[Outer Lipschitz Classification of Normal Pairs of H\"older Triangles]{Outer Lipschitz Classification of Normal Pairs of H\"older Triangles}

\author[]{Lev Birbrair*}
\address{Departamento de Matem\'atica, Universidade Federal do Cear\'a
(UFC), Campus do Pici, Bloco 914, Cep. 60455-760. Fortaleza-Ce,
Brasil}
\address{Department of Mathematics, Jagiellonian University, Prof.~Stanisława Łojasiewicza 6, 30-348, Kraków, Poland} \email{lev.birbrair@gmail.com}

\author[]{Andrei Gabrielov}
\address{Department of Mathematics, Purdue University,
West Lafayette, IN 47907, USA}\email{gabriea@purdue.edu}

\date{\today}


\begin{abstract}
 {A normal pair of H\"older triangles is the union of two normally embedded H\"older triangles
  satisfying some natural conditions on the tangency orders of their boundary arcs.
  It is a special case of a surface germ, a germ at the origin of a two-dimensional closed semialgebraic
  (or, more general, definable in a polynomially bounded o-minimal structure) subset of $\R^n$.
 Classification of normal pairs considered in this paper is a step towards
  outer Lipschitz classification of definable surface germs.
  In the paper \cite{BG} we introduced a combinatorial invariant of the outer Lipschitz equivalence class of normal pairs,
  called $\sigma\tau$-pizza, and conjectured that it is complete:
 two normal pairs of H\"older triangles with the same $\sigma\tau$-pizzas are outer Lipschitz equivalent.
 In this paper we prove that conjecture and define realizability conditions for the $\sigma\tau$-pizza invariant.
 Moreover, only one of the two pizzas in the $\sigma\tau$-pizza invariant,
 together with some admissible permutations related to $\sigma$ and $\tau$, is sufficient for the existence and
 uniqueness, up to outer Lipschitz equivalence, of a normal pair of H\"older triangles.}
\end{abstract}
\maketitle

\section{Introduction}\label{sec:intro}
All sets, functions and maps in this paper are germs at the origin of $\R^n$ definable in a polynomially bounded o-minimal structure over $\R$ with the field of exponents
$\F$.
The simplest (and most important in applications) examples of such structures are real semialgebraic and global subanalytic sets, with ${\F}={\Q}$ (see \cite{Dries}).

There are two natural metrics on a connected  set $X\subset\R^n$: the \emph{inner metric}, where the distance between two points of $X$
is the length of a shortest path in $X$ connecting these points,
and the \emph{outer metric}, where the distance between two points of $X$ is their distance in $\R^n$.
A set $X$ is called \emph{normally embedded} (see \cite{birbrair2000normal}) if its inner and outer metrics are equivalent.
There are three natural equivalence relations associated with these metrics. Two sets $X$ and $Y$ are \emph{inner} (resp., \emph{outer}) Lipschitz equivalent if there is an
inner (resp., outer) bi-Lipschitz homeomorphism $h:X\to Y$. The sets $X$ and $Y$ are \emph{ambient} Lipschitz equivalent if the homeomorphism $h:X\to Y$ can be extended to a
bi-Lipschitz homeomorphism $H$ of the ambient space. The ambient equivalence is stronger than the outer equivalence, and the outer equivalence is stronger then the inner
equivalence.  Finiteness theorems of Mostowski \cite{Mostowski} and Valette \cite{valette2005Lip} show that there are finitely many ambient Lipschitz equivalence classes in
any definable family.

Inner Lipschitz classification of surface germs was established by the first author in \cite{birbrair1999}.
The building block of the inner Lipschitz classification of surface germs is a \emph{$\beta$-H\"older triangle} (see Definition \ref{holder}).
A combinatorial model for the inner Lipschitz equivalence class of a surface germ $X$ is a \emph{Canonical H\"older Complex},
based on a decomposition of $X$ into H\"older triangles and isolated arcs (see \cite{birbrair1999}).

Outer Lipschitz geometry of surface germs is considerably more complicated.
A special case of a surface germ $(T, Graph(f))$, the union of a H\"older triangle $T$
and a graph of a Lipschitz function $f$ defined on $T$. Two such pairs $(T, Graph(f))$ and $(T, Graph(g))$ are outer Lipschitz equivalent when the functions $f$ and $g$ are contact Lipschitz equivalent \cite{birbrair2014lipschitz}.

This relates outer Lipschitz geometry of surface germs with the Lipschitz geometry of functions.
A complete invariant of the Lipschitz contact equivalence of Lipschitz functions, called \emph{minimal pizza}, was defined in \cite{birbrair2014lipschitz}.
Informally, a pizza for a Lipschitz function $f$ on a normally embedded H\"older triangle $T$ is a decomposition of $T$ into ``pizza slices,''
subtriangles  $T_i$ of $T$, such that the order of $f$ on each arc $\gamma\subset T_i$ depends linearly on the tangency order of $\gamma$
with a boundary arc of $T_i$. Thus a pizza can be encoded by a piecewise-linear combinatorial object, reminiscent of tropical geometry.
A pizza is minimal if the union of any two adjacent pizza slices is not a pizza slice.

A general pair of normally embedded H\"older triangles is considerably more complicated the a pair $(T, Graph(f))$.
Part of this complexity was described in our paper \cite{BG}, even more of it being uncovered in the current paper.

In this paper we consider \emph{normal pairs} of H\"older triangles defined in \cite{BG},
surface germs $X=T\cup T'\subset\R^n$, where $T$ and $T'$ are normally embedded H\"older triangles satisfying natural conditions (\ref{tord-tord}) on the tangency orders of their boundary arcs.
Let $f:T\to\R$ and $g:T'\to\R$ be Lipschitz functions defined as the distances in $\R^n$ from the points in each of these two triangles to the other triangle. The first question is whether $X$ is outer Lipschitz equivalent to the union of $T$ and the graph of the distance function $f$.
Some examples (see Figure \ref{fig:pi}) show that the answer may be negative. These examples show that the theory is much more rich, then one can expect.
Another natural question is whether the minimal pizzas of $f$ and $g$ are equivalent.
It was shown in the paper \cite{BG} that the answer to this question is also negative.
The minimal pizzas for $f$ and $g$ may even have different numbers of pizza slices (see Examples \ref{example:varpi} and \ref{example:varpi2} below).

In paper \cite{BG} we defined the \emph{$\sigma\tau$-pizza} invariant of a normal pair $(T,T')$ of H\"older triangles,
consisting of the minimal pizzas $\Lambda$ and $\Lambda'$ associated with the distance functions $f$ and $g$ on $T$ and $T'$,
characteristic permutation $\sigma$ and characteristic correspondence $\tau$, where
$\sigma$ and $\tau$ detect relations between certain elements of the pizzas $\Lambda$ and $\Lambda'$ as follows.
The spaces of arcs of $T$ and $T'$ contain outer Lipschitz invariant subsets, called \emph{maximum zones} (see Definition \ref{maxmin}).
The permutation $\sigma$ in \cite[Definition 4.5]{BG} (see Proposition \ref{sigma} and Definition \ref{characteristic} below) is encoding a canonical one-to-one correspondence between the maximum zones of $\Lambda$ and $\Lambda'$.
Two pizza slices $T_i$ and $T'_j$ of $\Lambda$ and $\Lambda'$ are called \emph{transverse} if the distance functions of the pair $(T_i,T'_j)$ are equivalent to the distances from the points of one of these triangles to a boundary arc of another one.
Otherwise $T_i$ and $T'_j$ are called \emph{coherent} (see Definition \ref{def:pizzaslicezone-transverse}).
There is a canonical one-to-one correspondence $\tau$ between coherent pizza slices of $\Lambda$ and $\Lambda'$ (see
\cite[Definition 4.8]{BG} and Definition \ref{def:tau} below). The theory, created in \cite{BG} can be interpreted as a study of the dynamics of the family of intersections of the surface with a small sphere, when the radius of the sphere tends to zero. The main tool of the theory is a serious use of non-archimedean geometry.

The main result of \cite{BG} states that $\sigma\tau$-pizzas of outer Lipschitz equivalent normal pairs of H\"older triangles are combinatorially equivalent.  In this paper we show that the converse is also true:
The pizzas $\Lambda$ and $\Lambda'$, together with the characteristic permutation $\sigma$
and characteristic correspondence $\tau$, constitute a \emph{complete invariant} of the outer
Lipschitz equivalence class of normal pairs of H\"older triangles.
Moreover, given any one of the two pizzas, and given the permutation $\sigma$ and correspondence $\tau$ (more precisely, given
$\sigma$ and a permutation $\varpi$ associated with $\sigma$ and $\tau$, see Definition \ref{def:varpi})
satisfying some explicit admissibility conditions,
a normal pair $(T,T')$ with the given pizza, $\sigma$ and $\tau$ exists and is unique up to outer Lipschitz equivalence.
The admissibility conditions are necessary: they are satisfied by the $\sigma\tau$-pizza invariant of any normal pair $(T,T')$.

Section \ref{sec:prelim} of this paper introduces the main tools of our study.
We consider the Valette link \cite{Valette-Link} of a germ $X$, i.e., the set $V(X)$ of arcs in $X$ parameterized by the distance to the origin.
It has a structure of a non-archimedean metric space,
where the metric is defined by the tangency order of arcs.

The main advantage of working in $V(X)$, instead of $X$ itself, is the possibility to define outer Lipschitz invariant sets of arcs
in $V(X)$, while the only outer Lipschitz invariant arcs in $X$ are Lipschitz singular arcs (see Definition \ref{singulararc}).
We describe the properties of sets and maps in the spaces of arcs corresponding to the geometric properties of surface germs and their bi-Lipschitz maps.
In particular, we give another proof of the theorem of Fernandes (see \cite{Alexandre}) that a map between two germs is outer bi-Lipschitz if, and only if, the corresponding
map between their Valette links is an isometry.
We introduce a notion of \emph{combinatorial normal embedding} of a H\"older triangle, in terms of its Valette link, and prove that combinatorial normal embedding is equivalent to normal embedding.

We remind the definition of a pizza from \cite{birbrair2014lipschitz}, and define an \emph{abstract pizza} (see Definition \ref{abstract-pizza})
as a combinatorial encoding of an equivalence class of a pizza associated with a non-negative Lipschitz function on a normally embedded H\"older triangle.
We show that a pizza, unique up to combinatorial equivalence, can be recovered from an abstract pizza (see Theorem \ref{pizza-realization}).
Thus a minimal abstract pizza represents a contact equivalence class of a non-negative Lipschitz function.

At the end of Section \ref{sec:prelim} we reformulate the notion of pizza in terms of \emph{pizza zones}, Lipschitz invariant subsets
of the Valette link of a normally embedded H\"older triangle $T$, where the boundary arcs of pizza slices of a minimal pizza associated with
a Lipschitz function on $T$ may be selected.

In Section \ref{sec:sigma} we define \emph{maximal exponent zones} (or simply \emph{maximum zones}),
the pizza zones of the minimal pizzas $\Lambda$ and $\Lambda'$ on $T$ and $T'$, associated with the distance functions $f$ and $g$,
where the orders of these functions attain local maxima, and introduce the \emph{characteristic permutation} $\sigma$ of a normal pair $(T,T')$
encoding one-to-one correspondence between the maximum zones of $\Lambda$ and $\Lambda'$.

In Section \ref{sec:tau} we define \emph{transverse} and \emph{coherent} pizza slices of the pizzas $\Lambda$ and $\Lambda'$
and introduce the correspondence $\tau$ between their coherent pizza slices.
This is a \emph{signed} correspondence: since pizza slices of $\Lambda$ and $\Lambda'$ have orientations induced by the orientations
of $T$ and $T'$, the action of $\tau$ may either preserve or reverse that orientation.
To understand relations between $\sigma$ and $\tau$, we investigate combinatorial and metric properties of the $\sigma\tau$-pizza
invariant defined in \cite{BG}.
Note that $\sigma$ and $\tau$ are of a rather different nature: the permutation
$\sigma$ is encoding a one-to-one correspondence between maximum zones
of pizzas $\Lambda$ and $\Lambda'$, which are some of their pizza zones,
while $\tau$ is a one-to-one correspondence between coherent pizza slices of $\Lambda$ and $\Lambda'$,
which cannot in general be extended to a one-to-one correspondence between their pizza zones.
We start with the permutation $\upsilon$ and the sign function $s$ on coherent pizza slices induced by $\tau$ (see Definition \ref{upsilon}).
The triple $(\sigma,\upsilon,s)$ defined for a pair $(T,T')$ satisfies \emph{allowability conditions} (see Definition \ref{def:allowable})
which can be formulated in terms of the pizza $\Lambda$ on $T$, based on the properties of sets of pizza slices of $\Lambda$
called \emph{caravans} (see Definition \ref{def:caravan}).
A triple $(\sigma,\upsilon,s)$ is \emph{allowable} when these conditions are satisfied.

To combine $\sigma$ with $\upsilon$, we consider the disjoint union $\mathcal K$ of the sets of maximum zones and
 coherent pizza slices of $\Lambda$, and the corresponding set $\mathcal K'$ for $\Lambda'$.
 These two sets have the same number of elements $K$.
 Ordering them according to orientation of $T$ and $T'$, we define the \emph{combined characteristic permutation} $\omega$
of the pair $(T,T')$ on a set of $K$ elements (see Definition \ref{omega})
 and show that, given a pizza $\Lambda$ on $T$ and an allowable triple $(\sigma,\upsilon,s)$,
 the permutation $\omega$ can be uniquely determined, even when $T'$ and $\Lambda'$ are not known.
This construction is important for the realization theorem (Theorem \ref{general-blocks}) in Section \ref{sec:blocks-general}.

In Section \ref{sec:invariant} we prove (Theorem \ref{complete})
that two normal pairs $(T,T')$ and $(S,S')$ of H\"older triangles are outer Lipschitz equivalent if, and only if, their $\sigma\tau$-pizza invariants are combinatorially equivalent.

In Section \ref{sec:blocks} we introduce a combinatorial notion of \emph{blocks} that allows us to establish relations between the metric (pizzas)
and combinatorial ($\sigma$ and $\tau$) parts of the $\sigma\tau$-pizza invariant, which are necessary and sufficient for the existence and uniqueness, up to outer Lipschitz equivalence, of a normal pair of H\"older triangles.
We do this first in the \emph{totally transverse} case, when there are no coherent pizza slices.

For a totally transverse pizza $\Lambda$ associated with a non-negative Lipschitz function $f$ on a H\"older triangle $T$,
a family $\lambda_0,\ldots,\lambda_{n-1}$ of $n$ arcs in $V(T)$,
ordered according to orientation of $T$,
which are either the boundary arcs of $T$ or belong to maximum zones of $\Lambda$,
such that each maximum zone contains exactly one of these arcs,
is called a \emph{supporting family} (see Definition \ref{supporting-family}).
The pizza $\Lambda$ is completely determined by the exponents $\beta_{ij}=tord(\lambda_i,\lambda_j)$ and $q_i=ord_{\lambda_i} f$ associated with a supporting family.

Let $\chi$ be a permutation of the set $[n]=\{0,\ldots,n-1\}$ of $n$ elements.
A \emph{segment} of $[n]$ is a non-empty set of consecutive indices $\{i,\ldots,k\}$.
A segment $B$ of $[n]$ is called a \emph{block} of $\chi$ if the set $\chi(B)$ is also a segment of $[n]$
(not necessarily in increasing order).
Each non-empty subset $J$ of $[n]$ is contained in a unique minimal block $B_{\chi}(J)$ of $\chi$.
A permutation $\chi$ of $[n]$
is called \emph{admissible} with respect to a Lipschitz function $f$ on $T$ (or with respect to a minimal pizza $\Lambda$ on $T$ associated with $f$) if it satisfies the \emph{block conditions}:
$\beta_{ij}\le\beta_{ik}$ for all $k\in B_{\chi}(\{i,j\})$ (condition (\ref{beta:blocks}) in Theorem \ref{beta-block}).

Given a totally transverse pizza $\Lambda$ with a supporting family of $n$ arcs on a H\"older triangle $T$,
associated with a non-negative function $f$ on $T$,
and an admissible with respect to $f$ permutation $\pi$ of $[n]$, there exists a unique, up to outer Lipschitz equivalence,
totally transverse normal pair $(T,T')$ realizing the pizza $\Lambda$ and the permutation $\pi$ of $[n]$
compatible with the characteristic permutation $\sigma$ of the pair $(T,T')$ on the family of maximum zones.

In Section \ref{sec:blocks-general} we establish the existence and uniqueness, up to
outer Lipschitz equivalence, of a general normal pair of H\"older triangles,
based on admissibility conditions, including the block conditions, for an analog $\varpi$ of the permutation $\pi$.
In this section we define \emph{pre-pizza} (see Definition \ref{def:pre-pizza}) obtained from a minimal pizza by removing \emph{non-essential} pizza zones where the action of
$\sigma$ and $\tau$
is not defined, and \emph{twin pre-pizza} (see Definition \ref{def:twin-pre-pizza}) obtained from a pre-pizza by adding ``twin arcs''
to ensure that $\tau$ is one-to-one and compatible with $\sigma$ on the expanded set of arcs.
In the \emph{totally transverse} case, when there are no coherent pizza slices, pre-pizza is determined by the maximum zones and twin pre-pizza
is the same as pre-pizza.
The minimal pizza, unique up to combinatorial equivalence, can be recovered from the corresponding pre-pizza,
and a pre-pizza can be recovered from the corresponding twin pre-pizza.
Using definitions of pre-pizza and twin pre-pizza, we define admissibility conditions for the permutation $\varpi$. The construction of $\varpi$ is based on the combined characteristic permutation $\omega$ defined in Section \ref{sec:tau}.
The admissibility conditions for $\varpi$ include the allowability conditions for the triple $(\sigma,\upsilon,s)$ and the block conditions  (\ref{dart:blocks}).
The main result of this section (and also the main realization result of the paper) is similar to the realization theorem (Theorem \ref{transverse-blocks}) for the totally transverse case in Section \ref{sec:blocks}.

Given a H\"older triangle $T$, a pizza $\Lambda$ and an admissible permutation $\varpi$, one can construct a normally embedded triangle $T'$, such that $(T,T')$ is a normal pair of H\"older triangles
with the pizza $\Lambda$ associated with the distance function $f(x)=dist(x,T')$ on $T$, and $\varpi$ is compatible with the characteristic permutation $\sigma$ and characteristic correspondence $\tau$ of the pair $(T,T')$.

Some remarks about the figures. We try to illustrate the ``dynamics'' of the link of a surface germ $X$,
dependence of the intersection of $X$ with a small sphere on the radius of the sphere.
Accordingly, in our figures $X$ is a curve, a pair of H\"older triangles is the union of two intervals, and
some arcs in $X$ are marked as points. We try to imitate the non-archimedean metric on $V(X)$ by representing
the arcs with higher tangency order by the nearby points in the plane, and mark some important zones as shaded disks.
We hope it will bring some intuition to the reader.

The second author would like to express his gratitude for the warm hospitality to the Jagiellonian University in Krakow, Poland,
which he visited while this paper was being written.

\section{Preliminaries}\label{sec:prelim}
As stated above, all sets, functions and maps in this paper are germs at the origin of $\R^n$ definable in a polynomially bounded o-minimal structure.
To simplify notations, we write $X$ for a germ $(X,0)\subset(\R^n,0)$.

\begin{definition}\label{metrics}\normalfont
A germ $X\subset\R^n$ inherits two metrics from the ambient space: the \emph{inner metric}, the distance $idist(x,y)$ between two points $x$ and $y$ of
$X$ being the length of the shortest path connecting them inside $X$, and the \emph{outer metric} with the distance $dist(x,y)=|x-y|$.
A germ $X$ is \emph{normally embedded} if these two metrics are equivalent.
For a point $x\in X$ and a subset $Y\subset X$, the \emph{inner distance} is $idist(x,Y)=\inf_{y\in Y} idist(x,y)$ and the \emph{outer distance} is $dist(x,Y)=\inf_{y\in Y}|x-y|$.\newline
A \emph{surface germ} is a closed germ $X$ at the origin such that $\dim_\R X=2$.
\end{definition}

\begin{remark}\label{inner-pancake}\normalfont
In general, the inner metric is not definable, but there is a definable ``pancake metric'' (see \cite{kurdyka-orro})
equivalent to the inner metric.
\end{remark}

\begin{definition}\label{arc}\normalfont
An \emph{arc} in $\R^n$ is (a germ at the origin of) a mapping $\gamma:[0,\epsilon)\rightarrow \R^n$ such that $\gamma(0)=\mathbf 0$.
Unless otherwise specified, we suppose that arcs are parameterized by the distance to the origin, i.e., $|\gamma(t)|=t$.
We usually identify an arc $\gamma$ with its image in $\R^n$.
The \emph{Valette link} of $X$ is the set $V(X)$ of all arcs $\gamma\subset X$.
\end{definition}

\begin{definition}\label{ordonarc}\normalfont
Let $f\not\equiv 0$ be a Lipschitz function defined on an arc $\gamma$, such that $f(\mathbf 0)=0$.
The \emph{order} of $f$ on $\gamma$ is the exponent $q=ord_\gamma f\in\F_{\ge 1}$ such that $f(\gamma(t))=c t^q+o(t^q)$ as $t\to 0$, where $c\ne 0$.
If $f\equiv 0$ on $\gamma$, then $ord_\gamma f=\infty$.
\end{definition}

\begin{definition}\label{tord}\normalfont
If $\gamma$ and $\gamma'$ are two arcs, then
the \emph{tangency order} between them is defined as $tord(\gamma,\gamma')=ord_{\gamma}|\gamma(t)-\gamma'(t)|$.
The tangency order between an arc $\gamma$ and a set of arcs $Z\subset V(X)$  is defined as
$tord(\gamma,Z)=\sup_{\lambda\in Z} tord(\gamma,\lambda)$.
The tangency order between two subsets $Z$ and $Z'$ of $V(X)$ is defined as $tord(Z,Z')=\sup_{\gamma\in Z} tord(\gamma,Z')$.
Similarly, $itord_X (\gamma,\gamma')$, $itord_X (\gamma,Z)$ and $itord_X (Z,Z')$ denote the tangency orders with respect to the inner metric.
The distance $\xi(\gamma,\gamma')=1/tord(\gamma,\gamma')$ between arcs in $X$ defines a \emph{non-archimedean metric} $\xi$ on $V(X)$.
\end{definition}

\begin{proposition}\label{important-proposition} \emph{(See \cite{Alexandre}.)}
 Let $X$ and $Y$ be two germs at the origin of $\R^n$. A homeomorphism $\Phi:X\to Y$ preserving the distance to the origin is bi-Lipschitz if,
 and only if, for any two arcs $\gamma_1, \gamma_2 \in V(X)$ one has
 \begin{equation}\label{twoarcs}
 tord(\gamma_1,\gamma_2)=tord(\Phi(\gamma_1),\Phi(\gamma_2)).
 \end{equation}
\end{proposition}

\begin{proof} We present a proof, slightly different from the proof given in \cite{Alexandre}.
It is similar to the proof of the main result of \cite{comRodrigo}.
If $\Phi$ is bi-Lipschitz then (\ref{twoarcs}) is obviously satisfied.
Let us show that $\Phi$ is a Lipschitz map when (\ref{twoarcs}) is satisfied.

Consider the set $W\subset\R^n\times\R^n\times\R$ defined as
\begin{equation}\label{W}
W=\{(x_1,x_2,z):x_1\in X,\,x_2\in X,\,0<z<1,\; |x_1-x_2|<z |\Phi(x_1)-\Phi(x_2)|\}
\end{equation}
Note that, since $\Phi$ preserves the distance to the origin, the set $W$ does not contain any points with $x_1=0$ or $x_2=0$.
If $\Phi$ is not a Lipschitz map, then the closure of the set $W_c=W\cap\{z=c\}$ contains the point $(0,0,c)$, for $0<c<1$.
By the arc selection lemma, the set $W\cup\{0,0,0\}$ contains an arc $\gamma(z)=(\gamma_1(z),\gamma_2(z),z)$
parameterized by $z\ge 0$, such that $\lim_{z\to 0} (\gamma_1(z),\gamma_2(z))=(0,0)$ and
$|\gamma_1(z)-\gamma_2(z)|<z |\Phi(\gamma_1(z))-\Phi(\gamma_2(z))|$.
Then the projection of $\gamma$ to $\R^n\times\R^n$ along the $z$-axis is an arc $\Gamma(z)=(\gamma_1(z),\gamma_2(z))\subset X\times X$ parameterized by $z$, such that
\begin{equation}\label{gamma}
|\gamma_1(z)-\gamma_2(z)|<z |\Phi(\gamma_1(z))-\Phi(\gamma_2(z))|,
\end{equation}
in contradiction with (\ref{twoarcs}).

The same arguments applied to $\Phi^{-1}$ show that $\Phi^{-1}$ is also a Lipschitz map.
\end{proof}

\begin{definition}\label{standard holder}
	\normalfont For $\beta \in \F$, $\beta \ge 1$, the \emph{standard $\beta$-H\"older triangle} is (a germ at the origin of) the set
	\begin{equation}\label{Formula:Standard Holder triangle}
	T_\beta = \{(u,v)\in \R^2 \mid u\ge 0, \; 0\le v \le u^\beta\}.
	\end{equation}
	The arcs $\{u\ge 0,\; v=0\}$ and $\{u\ge 0,\; v=u^\beta\}$ are the \emph{boundary arcs} of $T_\beta$.
\end{definition}

\begin{definition}\label{holder}\normalfont
A $\beta$-\emph{H\"older triangle} is (a germ at the origin of) a set $T \subset \R^n$ that is inner bi-Lipschitz homeomorphic
to the standard $\beta$-H\"older triangle (\ref{Formula:Standard Holder triangle}).
The number $\beta=\mu(T) \in \F$ is called the \emph{exponent} of $T$.
The arcs $\gamma_1$ and $\gamma_2$ of $T$ mapped to the boundary arcs of $T_\beta$ by an inner bi-Lipschitz homeomorphism are the
\emph{boundary arcs} of $T$ (notation $T=T(\gamma_1,\gamma_2)$).
All other arcs of $T$ are its \emph{interior arcs}. The set of interior arcs of $T$ is denoted by $I(T)$.
An arc $\gamma\subset T$ is \emph{generic} if $itord(\gamma,\gamma_1)=itord(\gamma,\gamma_2)$.
The set of generic arcs of $T$ is denoted by $G(T)$.
\end{definition}

\begin{definition}\label{def:transverse}\normalfont
Two normally embedded H\"older triangles $T$ and $T'$ are called \emph{transverse} if there is a boundary arc $\tilde\lambda$ of $T$ and a boundary arc $\tilde\lambda'$ of $T'$ such that
$tord(\tilde\lambda,\gamma')=tord(\gamma',T)$ for any arc $\gamma'$ of $T'$ and $tord(\tilde\lambda',\gamma)=tord(\gamma,T')$ for any arc $\gamma$ of $T$.
 \end{definition}

\begin{remark}\label{rem:transverse}\normalfont
Two transverse H\"older triangles $T$ and $T'$ in Definition \ref{def:transverse} either are disjoint or have a
common boundary arc $\lambda$. In the latter case, $\tilde\lambda=\tilde\lambda'=\lambda$.
\end{remark}

\begin{definition}\label{combinatorialLNE}\normalfont
Let $T=\bigcup T_i$ be a H\"older triangle decomposed into normally embedded subtriangles $T_i=T(\lambda_{i-1},\lambda_i)$,
such that $T_i \cap T_{i+1} = \lambda_i$.
We say that $T$ is \emph{combinatorially normally embedded} if any two triangles $T_i$ and $T_j$ are transverse (see Definition \ref{def:transverse}) and
\begin{equation}\label{check-LNE1}
tord(\lambda_i,\lambda_j)=\min(tord(\lambda_i,\lambda_k),tord(\lambda_k,\lambda_j))\;\;\text{for all}\;\; i,j,k\;\; \text{such that}\;\;i<k<j.
\end{equation}
\end{definition}

\begin{proposition}\label{combinatorialLNE-prop}
If a H\"older triangle $T$ is combinatorially normally embedded, then it is normally embedded.
\end{proposition}

\begin{proof}
We proceed by induction on the number of triangles $T_i$.
If $T$ consists of two normally embedded triangles $T_1=T(\lambda_0,\lambda_1)$ and $T_2=T(\lambda_1,\lambda_2)$ with the common boundary arc $\lambda_1$,
then transversality of $T_1$ and $T_2$, with $\tilde\lambda=\tilde\lambda'=\lambda_1$, means
$tord(\lambda_1, \gamma')=tord(\gamma',T_1)$ for any arc $\gamma'$ of $T_2$ and $tord(\lambda_1,\gamma)=tord(\gamma,T_2)$ for any arc $\gamma$ of $T_1$.
This implies $tord(\gamma,\gamma')\le tord(\gamma',T_1)=tord(\lambda_1,\gamma')$ and $tord(\gamma,\gamma')\le tord(\gamma, T_2)=tord(\lambda_2\gamma)$ for any arcs $\gamma\subset T_1$ and $\gamma'\subset T_2$. Thus
\begin{equation}\label{LNE2}
tord(\gamma,\gamma')\le\min(tord(\gamma,\lambda_1),tord(\lambda_1,\gamma'))=itord(\gamma,\gamma'),
\end{equation}
so $T_1\cup T_2$ is normally embedded.
Moreover, $tord(\lambda_0,\lambda_2)=\min(tord(\lambda_0,\lambda_1),tord(\lambda_1,\lambda_2)$, and condition (\ref{check-LNE1}) is automatically satisfied.

We consider next a triangle $T=T_1\cup T_2\cup T_3$, where $T_1=T(\lambda_0,\lambda_1),\;T_2=T(\lambda_1,\lambda_2)$ and $T_3=T(\lambda_2,\lambda_3)$,
where triangles $T_1,\;T_2$ and $T_3$ are normally embedded, any two of them are transversal, and condition (\ref{check-LNE1}) is satisfied.
The first step of induction implies that the sub-triangles $T_1\cup T_2$ and $T_2\cup T_3$ of $T$ are normally embedded.
To prove that $T$ is normally embedded, it is enough to show that $T_1$ is transverse to $T_2\cup T_3$ and to apply the first step of induction to
these two triangles.

If $\gamma$ is an arc in $T_1$, then $\beta=tord(\gamma,\lambda_1)=tord(\gamma,T_2)$, since $T_1\cup T_2$ is normally embedded.
To show that $\beta\ge tord(\gamma,T_3)$, let us assume that $\alpha=tord(\gamma,\gamma')>\beta$ for some arc $\gamma'\subset T_3)$.
Since triangles $T_1$ and $T_3$ are transverse, there are boundary arcs $\tilde\lambda$ of $T_1$ and $\tilde\lambda'$ of $T_3$
such that $\alpha\le tord(\tilde\lambda,\gamma')$ and $\alpha\le tord(\tilde\lambda',\gamma)$.

If $\tilde\lambda=\lambda_1$, then $tord(\gamma',\lambda_1)\ge\alpha$.
By the non-archimedean property of the tangency order, $tord(\gamma,\lambda_1)\ge\alpha$, in contradiction with the assumption $\alpha>\beta$.

If $\tilde\lambda'=\lambda_2$, then $tord(\gamma,\lambda_2)\ge\alpha$, in contradiction with $tord(\gamma,\lambda_2)\le tord(\gamma,\lambda_1)=\beta$.

In the remaining case $\tilde\lambda=\lambda_0$ and $\tilde\lambda'=\lambda_3$, we have $tord(\gamma',\lambda_0)\ge\alpha$
and $tord(\gamma,\lambda_3)\ge\alpha$. By the non-archimedean property of the tangency order, $tord(\lambda_0,\lambda_3)\ge\alpha$,
in contradiction with the inequality $tord(\lambda_0,\lambda_3)\le tord(\lambda_0,\lambda_1)\le\beta$ which follows from condition (\ref{check-LNE1}).

The case $T=T_1\cup\cdots\cup T_k$ for $k>3$ is similar. Consider the case $k>3$.

By the induction hypotheses  $T=T_1\cup\cdots\cup T_{k-1}$ is LNE and $T=T_2\cup\cdots\cup T_k$ and moreover $T_1$ and $T_k$ are transverse. Then the situation is reduced to the case of 3 triangles.

\end{proof}

\begin{definition}\label{singulararc}\normalfont Let $X$ be a surface germ.
An arc $\gamma\in V(X)$ is called \emph{Lipschitz non-singular} if it is topologically non-singular and
there exists a normally embedded H\"older triangle $T\subset X$ such that $\gamma\in I(T)$.
Otherwise, $\gamma$ is \emph{Lipschitz singular}.
A H\"older triangle $T$ is \emph{non-singular} if any arc $\gamma\in I(T)$ is Lipschitz non-singular.
\end{definition}

\begin{definition}\label{def:Qf}\normalfont
For a Lipschitz function $f$ on a H\"older triangle $T$, let
\begin{equation}\label{eqn:Qf}
Q_f(T)=\bigcup_{\gamma\in V(T)}ord_\gamma f.
\end{equation}
\end{definition}
\begin{remark}\label{remark-prosto} \normalfont
It was shown in \cite[ Lemma 3.3]{birbrair2014lipschitz} that $Q_f(T)$ is either a point or a closed interval in $\mathbb{F}\cup \{\infty\}$.
\end{remark}
\begin{definition}\label{def:elementary}\normalfont
A H\"older triangle $T$ is \emph{elementary} with respect to a Lipschitz function $f$ on $T$ if, for any $q\in Q_f(T)$ and any two arcs
$\gamma$ and $\gamma'$ in $T$ such that $ord_{\gamma}f=ord_{\gamma'}f=q$, the order of $f$ is $q$ on any arc in the H\"older triangle
$T(\gamma,\gamma')\subset T$.
\end{definition}

\begin{definition}\label{def:width function}\normalfont Let $T$ be a normally embedded H\"older triangle and $f$ a Lipschitz function on $T$.
For an arc $\gamma \subset T$, the \emph{width} $\mu_T(\gamma,f)$ of $\gamma$ with respect to $f$ is the infimum of exponents of H\"older triangles
$T'\subset T$ containing $\gamma$ such that $Q_f(T')$ is a point.
For $q\in Q_f(T)$ let $\mu_{T,f}(q)$ be the set of exponents $\mu_T(\gamma,f)$, where $\gamma$ is any arc in $T$ such that $ord_{\gamma}f=q$.
For each $q\in Q_f(T)$, the set $\mu_{T,f}(q)$ is finite (see \cite{birbrair2014lipschitz}).
This defines a multivalued \emph{width function} $\mu_{T,f}: Q_f(T)\to \F\cup \{\infty\}$.
If $T$ is an elementary H\"older triangle with respect to $f$, then the function $\mu_{T,f}$ is single valued.
When $f$ is fixed, we write $\mu_T(\gamma)$ and $\mu_T$ instead of $\mu_T(\gamma,f)$ and $\mu_{T,f}$.

The \emph{depth} $\nu_T(\gamma,f)$ of an arc $\gamma$ with respect to $f$
is the infimum of exponents of H\"older triangles $T'\subset T$ such that $\gamma\in G(T')$ and $Q_f(T')$ is a point.
By definition, $\nu_T(\gamma,f)=\infty$ when there are no such triangles $T'$.
\end{definition}

\begin{definition}\label{def:contact-equiv}\normalfont
 Two function germs  $f,g: (T,0) \to (\R,0)$ on a normally embedded H\"older triangle $T$ are
\emph{Lipschitz contact equivalent} if there exist two germs of outer
bi-Lipschitz homeomorphisms $h:(T,0) \to (T,0)$ and  $\Omega: (T\times\R,0) \to (T \times \R,0)$ such that
$\Omega(T \times \{0\}) = T\times \{0\}$ and the following diagram is commutative:
\[
\begin{array}{lllll}
(T,0) & \stackrel{(id,\, f)}{\longrightarrow}
 & (T \times \R,0) & \stackrel{\pi}{\longrightarrow} & (T,0) \\
\,\,\,h \, \downarrow &  & \,\,\,\, \Omega \, \downarrow  &  & \,\,\, h \, \downarrow \\
(T,0)& \stackrel{(id, \,g)}{\longrightarrow} & (T \times \R,0) & \stackrel{\pi}{\longrightarrow}& (T,0) \\
\end{array}
\]

\end{definition}

\begin{definition}\label{def:pizza-slice}\normalfont
Let $T=T(\gamma_1,\gamma_2)$ be a normally embedded $\beta$-H\"older triangle and $f$ a Lipschitz function on $T$.
We say that $T$ is a \emph{pizza slice} associated with
$f$ if it is elementary with respect to $f$, either $Q_f(T)$ is not a point and $\mu_{T,f}(q)=aq+b$ is an affine function on $Q_f(T)$, or
$Q_f(T)=\{q\}$ is a point and $\mu_{T,f}(q)=\beta$ is a single exponent.
If $T$ is a pizza slice and $Q_f(T)$ is not a point, then $a\ne 0$, and one of the boundary arcs of $T$ (either $\tilde\gamma=\gamma_1$ or $\tilde\gamma=\gamma_2$) such that
$\mu_T(\tilde\gamma,f)=\max_{q\in Q_f(T)}\mu_{T,f}(q)$ is called the \emph{supporting arc} of $T$ with respect to $f$.
If $Q_f(T)$ is a point then supporting arc is not defined.
\end{definition}

\begin{proposition}\label{prop:width function properties}
Let $T$ be a normally embedded $\beta$-H\"older triangle which is a pizza slice associated with a non-negative Lipschitz function $f$.
Then $\beta\le\mu_{T,f}(q)\le \max(q,\beta)$ for all $q \in Q$.
If $\tilde\gamma$ is the supporting arc of $T$ with respect to $f$, then
$\mu_T(\gamma,f)=tord(\tilde\gamma,\gamma)$ for all arcs $\gamma\subset T$ such that $\mu_T(\gamma,f)<\mu_T(\tilde\gamma,f)$.
In particular, $\min_{q\in Q}\mu_{T,f}(q)=\beta$.
\end{proposition}

\begin{proof}
Since $f$ is a Lipschitz function, we have $\beta\le\mu_{T,f}(q)\le \max(q,\beta)$.
The equality  $\mu_T(\gamma,f)=tord(\tilde\gamma,\gamma)$ follows from the non-archimedean property of the tangency order
(see also \cite[Lemma 3.3]{birbrair2014lipschitz}).
\end{proof}

\begin{definition}\label{pizza-decomp}\normalfont (See \cite[Definition 2.13]{birbrair2014lipschitz}.)
Let $f$ be a non-negative Lipschitz function on a normally embedded $\beta$-H\"older triangle $T=T(\gamma_1,\gamma_2)$
oriented from $\gamma_1$ to $\gamma_2$.
A \emph{pizza} on $T$ associated with $f$ is a decomposition $\Lambda = \{T_\ell\}_{\ell=1}^p$ of $T$ into H\"older triangles $T_\ell=T(\lambda_{\ell-1},\lambda_\ell)$
ordered according to the orientation of $T$, such that
$\lambda_0=\gamma_1$ and $\lambda_p=\gamma_2$ are the boundary arcs of $T,\;
T_\ell\cap T_{\ell+1}=\lambda_\ell$ for $0<\ell< p$, and
each triangle $T_\ell$ is a pizza slice associated with $f$.
The pizza $\Lambda$ comes with the following \emph{toppings}:\newline
$\bullet$ Exponents $q_\ell=ord_{\lambda_\ell}f$ for $0\le\ell\le p$;\newline
$\bullet$ Exponents $\beta_\ell=\mu(T_\ell)$ for $1\le\ell\le p$;\newline
$\bullet$ Intervals of exponents $Q_\ell=[q_{\ell-1},q_\ell]$ for $1\le\ell\le p$;\newline
$\bullet$ Affine functions $\mu_\ell(q)$ on $Q_\ell$ for $1\le\ell\le p$, such that $\mu_\ell(q)\le q$ for all $q\in Q_\ell$ and $\min_{q\in Q_\ell}\mu_\ell(q)=\beta_\ell$,
or $\mu_\ell(q_\ell)=\beta_\ell$ is a single exponent when $Q_\ell=\{q_\ell\}$ is a point;\newline
$\bullet$ Exponents $\nu_\ell=\nu_T(\lambda_\ell,f)$, where $\nu_0=\nu_p=\infty,\;
\nu_\ell=\max(\mu_\ell(q_\ell),\mu_{\ell+1}(q_\ell))$ for $0<\ell<p$.\newline
The pizza $\Lambda$ is \emph{minimal} if $T_{\ell-1}\cup T_\ell$ is not a pizza slice for any $\ell>1$.
\end{definition}

\begin{definition}\label{abstract-pizza} \normalfont (See \cite[Definition 2.12]{birbrair2014lipschitz}.)
An \emph{abstract pizza} ${\mathcal P}$ on a set $\{0,\ldots,p\}$ is a finite sequence $\{q_\ell,\nu_\ell\}_{\ell=0}^p$ of exponents $q_\ell$ and $\nu_\ell$ in $\F_{\ge 1}\cup\{\infty\}$, and a finite
collection $\{\beta_\ell,Q_\ell,\mu_\ell\}_{\ell=1}^p$, where $\beta_\ell\in\F_{\ge 1}$,
$Q_\ell=[q_{\ell-1},q_\ell]\subset \F_{\ge 1}\cup\{\infty\}$ is either a point or a closed interval,
$\mu_\ell:Q_\ell\to\F\cup\{\infty\}$ is an affine function, non-constant when $Q_\ell$ is not a point, such that $\mu_\ell(q)\le q$ for all $q\in Q_\ell$ and $\min_{q\in
Q_\ell}\mu_\ell(q)=\beta_\ell$. If $Q_\ell=\{q_\ell\}$ is a point, then $\mu_\ell(q_\ell)=\beta_\ell$ is a single exponent.
Exponents $\nu_\ell$ are defined as follows:
$\nu_0=\nu_p=\infty$ and $\nu_\ell=\max(\mu_\ell(q_\ell),\mu_{\ell+1}(q_\ell))$ for $0<\ell< p$.\newline
An abstract pizza is \emph{reducible} if $p>1$ and one of the following conditions is satisfied:\newline
(a) $Q_\ell=Q_{\ell+1}=\{q_\ell\}$ for some $\ell<p$.\newline
(b) $Q_\ell=\{q_\ell\}\ne Q_{\ell+1}$ and $\beta_\ell\ge\mu_{\ell+1}(q_\ell)$ for some $\ell<p$.\newline
(c) $Q_\ell\ne\{q_\ell\}=Q_{\ell+1}$ and $\beta_{\ell+1}\ge\mu_\ell(q_\ell)$ for some $\ell<p$.\newline
(d) $Q_\ell\ne\{q_\ell\}\ne Q_{\ell+1}$ and $Q_\ell\cap Q_{\ell+1}=\{q_\ell\}$ for some $\ell<p$, and the function $\mu(q)$ on $Q_\ell\cup Q_{\ell+1}$, such that
$\mu(q)=\mu_\ell(q)$ for $q\in Q_\ell$ and
$\mu(q)=\mu_{\ell+1}(q)$ for $q\in Q_{\ell+1}$, is affine.\newline
An abstract pizza is \emph{minimal} if it is not reducible.
\end{definition}

\begin{remark} \normalfont We do not need the signs $\pm$ that are part of the definition in
\cite{birbrair2014lipschitz}, since only non-negative functions are considered in this paper.
\end{remark}

\begin{lemma}\label{lem:abstract-pizza}
Let ${\mathcal P}=\{\{q_\ell\}_{\ell=0}^p,\;\{\beta_\ell,Q_\ell,\mu_\ell\}_{\ell=1}^p,\;\{\nu_\ell\}_{\ell=0}^p\}$, where
$p>1$, be an abstract pizza. If $q_{\ell-1}=q_\ell$ and $\min(\nu_{\ell-1},\nu_\ell)=\beta_\ell$ for some $\ell>0$,
then $\mathcal P$ is reducible.
\end{lemma}

\begin{proof}
According to Definition \ref{abstract-pizza}, an abstract pizza is reducible if one of the conditions (a)-(d) is satisfied.
If $q_{\ell-1}=q_\ell$ and $\min(\nu_{\ell-1},\nu_\ell)=\beta_\ell$, then one of conditions (a)-(c) is satisfied, thus the abstract pizza is reducible.
\end{proof}

\begin{definition}\label{geometric-abstract} \normalfont Given a pizza $\Lambda = \{T_\ell\}_{\ell=1}^p$, where $T_\ell=T(\lambda_{\ell-1},\lambda_\ell)$,  associated with a
non-negative Lipschitz function $f$ on an oriented $\beta$-H\"older triangle $T$, the \emph{corresponding abstract pizza} ${\mathcal P}$ is defined by setting
$q_\ell=ord_{\lambda_\ell} f$, $\beta_\ell=\mu(T_\ell)=tord(\lambda_{\ell-1},\lambda_\ell)$, $Q_\ell=Q_f(T_\ell)$,
$\mu_\ell(q)=\mu_{T_\ell,f}(q)$, $\nu_\ell=\nu_T(\lambda_\ell,f)$. The abstract pizza ${\mathcal P}$ is minimal if, and only if, the pizza $\Lambda$ is minimal.
\end{definition}

Conversely, any abstract pizza as in Definition \ref{abstract-pizza} is associated with the pizza of a non-negative Lipschitz function on a standard $\beta$-H\"older triangle
$T_\beta$, where $\beta=\min_\ell\beta_\ell$ (see Theorem \ref{pizza-realization} below).
To construct such a function, we start with the definition of a standard pizza slice.

\begin{definition}\label{pizza_slice_standard}\normalfont
Given exponents $\beta$ and $\underline q$ in $\F_{\ge 1}$, exponent $\tilde q\ne\underline q$ in $\F_{\ge 1}\cup\{\infty\}$, and a non-constant affine function
$\mu(q)=a(q-\alpha)$ from $Q=[\underline q,\tilde q]$ to $\F\cup\{\infty\}$,
such that $\mu(q)\le q$ for all $q\in Q$ and $\min_{q\in Q}\mu(q)=\mu(\underline q)=\beta\ge 1$, let
\begin{equation}\label{phi}
\phi=\phi_{\beta,\underline q,\tilde q,\mu}(u,v)=u^\alpha v^{1/a}
\end{equation}
be a function on the $\beta$-H\"older triangle $T_{\beta,\kappa}=\{u\ge 0,\,u^\kappa\le v\le u^\beta\}$.
Here $\kappa=\max_{q\in Q}\mu(q)=\mu(\tilde q)>\beta$ and $u^\kappa\equiv 0$ if $\kappa=\infty$.
Let $h_{\beta,\kappa}$ be a bi-Lipschitz mapping
$(u,v)\mapsto\left(u,u^\kappa+v(1-u^{\kappa-\beta})\right)$ from the standard $\beta$-H\"older triangle $T_\beta$ to $T_{\beta,\kappa}$,
and let
\begin{equation}\label{psi}
\psi=\psi_{\beta,\underline q,\tilde q,\mu}(u,v)=\phi_{\beta,\underline q,\tilde q,\mu}(h_{\beta,\kappa}(u,v)).
\end{equation}
The pair $(T_\beta, \psi_{\beta,\underline q,\tilde q,\mu})$ is called the \emph{standard pizza slice}
associated with $\beta,\,\underline q,\,\tilde q$ and $\mu$.
When $\underline q=\tilde q$, thus $Q=\{\tilde q\}$ is a point and $\mu(\tilde q)=\beta$ is a single exponent,
the standard pizza slice is defined as $(T_\beta, \psi)$, where
\begin{equation}\label{psi-point}
 \psi=\psi_{\beta,\underline q,\tilde q,\mu}(u,v)=u^{\tilde q}.
\end{equation}
\end{definition}

\begin{lemma}\label{slice-geometric} The function $\psi=\psi_{\beta,\underline q,\tilde q,\mu}(u,v)$ in (\ref{psi}) is Lipschitz, and $T_\beta$ is a pizza slice associated
with $\psi$, with the affine width function $\mu(q)=a(q-\alpha)$ defined on $Q_\psi(T_\beta)=[\underline q,\tilde q]$.
\end{lemma}

\begin{proof}
Since $h_{\beta,\kappa}$ is (the germ at the origin of) a bi-Lipschitz homeomorphism $T_\beta\to T_{\beta,\kappa}$,
it is enough to show that $\phi=u^\alpha v^{1/a}$ in (\ref{phi}) is a Lipschitz function on $T_{\beta,\kappa}$.
Moreover, as the family of arcs $\gamma_\mu=\{v=u^\mu\}\subset T_{\beta,\kappa}$, for $\mu\in\F,\;\kappa\ge\mu\ge\beta$, is dense in $T_{\beta,\kappa}$, it is enough to show
that $d\phi(u,u^\mu)/du=(\alpha+\frac\mu{a}) u^{\alpha+\frac\mu{a}-1}$, and $\partial\phi/\partial v=\frac1a u^\alpha v^{\frac1a-1}$ restricted to $\gamma_\mu$, are uniformly bounded in $\mu$ in a neighborhood of the origin.
Since $\mu(q)=a(q-\alpha)$, we have $q(\mu)=\alpha+\frac\mu{a}$. In particular $T_{\beta,\kappa}$ is elementary because $h_{\beta,\kappa}$ is a bi-Lipschitz homeomorphism and $\phi=u^\alpha v^{1/a}$ is monotone.
We consider three cases, depending on the value of $a$.

{\bf Case} $0<a\le 1$.
Note that, when $\kappa=\infty$, we have $T_{\beta,\kappa}=T_\beta$ and $\psi=\phi=u^\alpha v^{1/a}$.
Since $q$ may be arbitrary large in that case, condition $1\le\beta\le\mu(q)\le q$ for all $q\in Q$
implies that this is possible only when $0<a\le 1$.
As the minimal value of $q(\mu)$ is $q(\beta)=\alpha+\frac\beta{a}$, we have
$\alpha\ge q\left(1-\frac1a\right)\ge\left(\alpha+\frac\beta{a}\right)\left(1-\frac1a\right)$.
Thus $\alpha\ge\beta\left(1-\frac1a\right)$ and $\alpha+\frac\mu{a}-1\ge\beta+\frac{\mu-\beta}a-1\ge0$.
Thus implies that $d\phi(u,u^\mu)/du=(\alpha+\frac\mu{a}) u^{\alpha+\frac\mu{a}-1}$ is either monotone increasing
function of $u$ for $u>0$, or a constant $\alpha+\frac\mu{a}=1$ when $\mu=\beta=1$ and $\alpha=1-\frac1a$.
Since this function is bounded as a function of $\mu$ for a fixed positive $u<1$, it is uniformly bounded on
a neighborhood of the origin in $T_\beta$.
This proves also that $d\phi(u,u^\mu)/du$ is bounded on $T_{\beta,\kappa}\subset T_\beta$ when $\kappa\ne\infty$.
Similarly, $\partial\phi/\partial v=\frac1a u^\alpha v^{\frac1a-1}$ restricted to $\gamma_\mu$ is
$\frac1a u^{\alpha+\frac\mu{a}-\mu}\le u^{(\mu-\beta)\left(\frac1a-1\right)}$ when $0\le u\le 1$, as
$\alpha+\frac\mu{a}-\mu\ge (\mu-\beta)\left(\frac1a-1\right)$.

{\bf Case} $a>1$. In that case $\kappa\ne\infty$, thus $\mu\le\kappa$ is bounded.
Condition $\mu(q)\le q$ is equivalent to $\alpha\ge q(1-\frac1a)$ for all $q\in Q$.
Since the maximal value of $q\in Q$ is $q(\kappa)=\alpha+\frac\kappa{a}$, this implies $\alpha\ge\left(\alpha+\frac\kappa{a}\right)\left(1-\frac1a\right)$, thus
$\alpha\ge\kappa\left(1-\frac1a\right)$.
We have $d\phi(u,u^\mu)/du=\left(\alpha+\frac\mu{a}\right) u^{\alpha+\frac\mu{a}-1}$,
which is bounded since $\alpha+\frac\mu{a}=q(\mu)\le q(\kappa)=\alpha+\frac\kappa{a}$ and $\alpha+\frac\mu{a}-1=q(\mu)-1\ge0$.
Similarly, $\partial\phi/\partial v=\frac1a u^\alpha v^{\frac1a-1}$ restricted to $\gamma_\mu$ is
$\frac1a u^{\alpha+\frac\mu{a}-\mu}$ which is bounded when $0\le u\le 1$, as
$\alpha+\frac\mu{a}-\mu\ge (\kappa-\mu)\left(1-\frac1a\right)\ge 0$.

{\bf Case} $a<0$. In that case $\kappa\ne\infty$, thus $\mu\le\kappa$ is bounded.
Condition $\mu(q)\le q$ is equivalent to $\alpha\le q(1-\frac1a)$ for all $q\in Q$.
Since the minimal value of $q\in Q$ is $q(\kappa)=\alpha+\frac\kappa{a}$, this implies $\alpha\le\left(\alpha+\frac\kappa{a}\right)\left(1-\frac1a\right)$, thus
$\alpha\ge\kappa\left(1-\frac1a\right)$.
We have $d\phi(u,u^\mu)/du=\left(\alpha+\frac\mu{a}\right) u^{\alpha+\frac\mu{a}-1}$,
which is bounded, as $\alpha+\frac\mu{a}=q(\mu)\le q(\beta)=\alpha+\frac\beta{a}$ and $\alpha+\frac\mu{a}-1=q(\mu)-1\ge 0$.
Similarly, $\partial\phi/\partial v=\frac1a u^\alpha v^{\frac1a-1}$ restricted to $\gamma_\mu$ is
$\frac1a u^{\alpha+\frac\mu{a}-\mu}$ which is bounded when $0\le u\le 1$, as
$\alpha+\frac\mu{a}-\mu\ge (\kappa-\mu)\left(1-\frac1a\right)\ge 0$.
\end{proof}

\begin{theorem}\label{pizza-realization} Given an abstract pizza
${\mathcal P}=\left\{\{q_\ell\}_{\ell=0}^p,\;\{\beta_\ell,Q_\ell,\mu_\ell\}_{\ell=1}^p\right\}$,
there is a non-negative  Lipschitz function $f:T_{\beta}\to\R$,
where $\beta=\min_\ell(\beta_\ell)$, such that $\mathcal P$ is the abstract pizza corresponding to the pizza
$\Lambda = \{T_\ell\}_{\ell=1}^p$ on $T_\beta$ associated with $f$.
\end{theorem}

\begin{proof}
The standard $\beta$-H\"older triangle $T_\beta$ can be decomposed
into $\beta_\ell$-H\"older triangles $T_\ell=T(\lambda_{\ell-1},\lambda_\ell)$, where $1\le\ell\le p$,
by the arcs $\lambda_\ell=\{u\ge 0,\;v=v_\ell(u)\}$, where $0\le\ell\le p$, $v_0(u)\equiv 0$ and $v_p(u)=u^\beta$.
For each $\ell=1,\ldots,p$, let $\mu_\ell(q)=a_\ell q+b_\ell$ and $v_\ell(u)=u^{\alpha_\ell}$, where $\alpha_\ell=min{\beta_i}$ for $i=1..\ell$.

If $a_\ell=0$ then $Q_\ell=\{q_\ell\}$ is a point and $f|_{T_\ell}=u^{q_\ell}$.

If $a_\ell>0$ and $q_{\ell-1}>q_\ell$, or $a_\ell<0$ and $q_{\ell-1}<q_\ell$, thus $\mu_\ell(q_\ell)=\beta$ is the minimal value of $\mu_\ell(q)$ and $\mu_\ell(q_{\ell-1})$ is
its maximal value,
we define a bi-Lipschitz homeomorphism $H^+_\ell:T_{\ell}\to T_{\beta_\ell}$, and we define $f(u,v)|_{T_\ell}$ as follows.
\begin{equation}\label{hell+}
H^+_\ell(u,v)=\big(u, u^{\beta_\ell}(v-v_{\ell-1}(u))/(v_\ell(u)-v_{\ell-1}(u))\big),
\end{equation}
such that $H^+_\ell(\lambda_{\ell-1})=\{v\equiv 0\}$ and $H^+_\ell(\lambda_\ell)=\{v=u^{\beta_\ell}\}$, and set
\begin{equation}\label{fell+}
f(u,v)|_{T_\ell}=\psi_\ell(H^+_\ell(u,v))
\end{equation}
where $\psi_\ell=\psi_{\beta_\ell,q_\ell,q_{\ell-1},\mu_\ell}$ is the standard pizza slice (\ref{psi}).
In this case, $\lambda_{\ell-1}$ is the supporting arc $\tilde\lambda_\ell$ of $T_\ell$ with respect to $f$,
$f|_{\lambda_{\ell-1}}=u^{q_{\ell-1}}$ and $f|_{\lambda_\ell}=u^{q_\ell}$.

If $a_\ell>0$ and $q_{\ell-1}<q_\ell$, or $a_\ell<0$ and $q_{\ell-1}>q_\ell$, thus $\mu_\ell(q_{\ell-1})=\beta$ is the minimal value of $\mu_\ell(q)$ and $\mu_\ell(q_\ell)$
is its maximal value,
we define a bi-Lipschitz homeomorphism $H^-_\ell:T_{\ell}\to T_{\beta_\ell}$,
\begin{equation}\label{hell-}
H^-_\ell(u,v)=\big(u, u^{\beta_{\ell}}(v-v_\ell(u))/(v_{\ell-1}(u)-v_\ell(u))\big),
\end{equation}
such that $H^-_\ell(\lambda_\ell)=\{v\equiv 0\}$ and $H^-_\ell(\lambda_{\ell-1})=\{v=u^{\beta_\ell}\}$, and set
\begin{equation}\label{fell-}
f(u,v)|_{T_\ell}=\psi_\ell(H^-_\ell(u,v))
\end{equation}
where $\psi_\ell=\psi_{\beta_\ell,q_{\ell-1},q_\ell,\mu_\ell}$ is the standard pizza slice (\ref{psi}).
In this case, $\lambda_\ell$ is the supporting arc $\tilde\lambda_\ell$ of $T_\ell$ with respect to $f$, $f|_{\lambda_{\ell-1}}=u^{q_{\ell-1}}$ and
$f|_{\lambda_\ell}=u^{q_\ell}$.

Since the standard pizza slice (\ref{psi}) and (\ref{psi-point}) is a Lipschitz function, and $f|_{\lambda_\ell}=u^{q_\ell}$
for $0\le\ell\le p$, the function $f$ is a Lipschitz function on $T_\beta$, such that $\mathcal P$ is the abstract pizza corresponding to the pizza associated with $f$.
\end{proof}

\begin{definition}\label{equivalent_pizza}\normalfont
Two pizzas are called \emph{combinatorially equivalent} if the corresponding abstract pizzas are the same. In other words, 
two pizzas $\Lambda=\{T_\ell=T(\lambda_{\ell-1},\lambda_\ell)\}_{\ell=1}^p$ and $\Lambda'=\{T'_\ell=T(\lambda'_{\ell-1},\lambda'_\ell)\}_{\ell=1}^p$
on a H\"older triangle $T$, associated with Lipschitz functions $f$ and $g$ on $T$, respectively,
are combinatorially equivalent if the exponents $\beta_\ell=\beta'_\ell$ of $T_\ell$ and $T'_\ell$ are the same,
the exponents $q_\ell=ord_{\lambda_\ell} f=q'_\ell=ord_{\lambda'_\ell} g$ are the same,
and the width functions $\mu_\ell(q)\equiv\mu'_\ell(q)$ on $Q_\ell=Q'_\ell$ are the same for $\Lambda$ and $\Lambda'$.
For each $\ell$ such that $Q_\ell=Q'_\ell$ is not a point, equalities $q_{\ell-1}=q'_{\ell-1},\;q_\ell=q'_\ell$
and $\mu_\ell(q)\equiv\mu'_\ell(q)$ imply that supporting arcs $\tilde\lambda_\ell$ and $\tilde\lambda'_\ell$
for $T_\ell$ and $T'_\ell$ are consistent: either $\tilde\lambda_\ell=\lambda_{\ell-1}$ and $\tilde\lambda'_\ell=\lambda'_{\ell-1}$
or $\tilde\lambda_\ell=\lambda_\ell$ and $\tilde\lambda'_\ell=\lambda'_\ell$.
 
\end{definition}

\begin{proposition}\label{contact_equiv} \emph{(See \cite[Theorem 4.9]{birbrair2014lipschitz}.)}
Two non-negative Lipschitz functions $f$ and $g$ on a normally embedded H\"older triangle $T$ are contact Lipschitz equivalent
if, and only if, minimal pizzas on $T$ associated with $f$ and $g$ are combinatorially equivalent.
\end{proposition}

\begin{definition}\label{zone}\normalfont (See \cite[Definitions 2.36 and 2.40]{GS}.) Let $X$ be a surface germ.
A non-empty set of arcs $Z\subset V(X)$ is called a \emph{zone} if, for any two arcs $\gamma_1\ne\gamma_2$ in $Z$, there exists a
non-singular H\"older triangle $T=T(\gamma_1,\gamma_2)$ such that $V(T)\subset Z$.
A \emph{singular zone} is a zone $Z=\{\gamma\}$ consisting of a single arc $\gamma$.
A zone $Z$ is \emph{normally embedded} if, for any two arcs $\gamma_1\ne\gamma_2$ in $Z$,
there exists a normally embedded H\"older triangle $T=T(\gamma_1,\gamma_2)$ such that $V(T)\subset Z$.
The \emph{order} of a zone $Z$ is defined as $\mu(Z)=\inf_{\gamma,\gamma'\in Z} tord(\gamma,\gamma')$.
If $Z$ is a singular zone, then $\mu(Z)=\infty$. A zone $Z$ is called a $\beta$-zone when $\mu(Z)=\beta$.
\end{definition}

\begin{definition}\label{closed}\normalfont
(See \cite[Definition 2.43]{GS}.)
A $\beta$-zone $Z$ is \emph{closed} if there are two arcs $\gamma$ and $\gamma'$ in $Z$ such that $tord(\gamma,\gamma')=\beta$.
Otherwise, $Z$ is an \emph{open} zone.
A zone $Z\subset V(X)$ is \emph{perfect} if, for any two arcs $\gamma$ and $\gamma'$ in $Z$, there exists a H\"older triangle $T\subset X$ such that
$V(T)\subset Z$ and both $\gamma$ and $\gamma'$ are generic arcs of $T$.
By definition, a singular zone is closed perfect.
\end{definition}

\begin{remark}\label{perfect-zone-remark} \normalfont (see \cite[Lemma 2.29]{BG})
A zone $Z\subset V(X)$ is perfect if, and only if, for any two arcs $\{\gamma_1,\gamma_2\} \in Z$ there exists an inner bi-Lipschitz map $\Psi:X \to X$, such that
$\Psi(\gamma_1)=\gamma_2$ and $\Psi$ is identity on $V(X)\setminus Z$.
If $Z$ and $Z'$ are disjoint closed perfect zones, then $itord(Z,Z')=itord(\gamma,\gamma')$ for any arcs $\gamma\in Z$ and $\gamma'\in Z'$.
\end{remark}

\begin{definition}\label{q-zone}\normalfont (See \cite[Definition 2.22]{BG}.)
Let $f:T\to\R$ be a Lipschitz function on a normally embedded H\"older triangle $T$.
A zone $Z\subset V(T)$ is a $q$-\emph{order zone} for $f$ if $ord_\gamma f=q$ for any arc $\gamma\in Z$.
A $q$-order zone for $f$ is \emph{maximal} if it is not a proper subset of any other $q$-order zone for $f$.
The \emph{width zone} $W_T(\gamma,f)$ of an arc $\gamma\subset T$ with respect to $f$ is the maximal $q$-order zone
for $f$ containing $\gamma$, where $q=ord_\gamma f$.
The order of $W_T(\gamma,f)$ is $\mu_T(\gamma,f)$.
The \emph{depth zone} $D_T(\gamma,f)$ of an arc $\gamma\subset T$ with respect to $f$ is
the union of zones $G(T')$ for all triangles $T'\subset T$ such that $\gamma\in G(T')$ and $Q_f(T')$ is a point.
The order of $D_T(\gamma,f)$ is $\nu_T(\gamma,f)$.
By definition, $D_T(\gamma,f)=\{\gamma\}$ and $\nu_T(\gamma,f)=\infty$ when there are no such triangles $T'$.
\end{definition}

\begin{lemma}\label{width zone is closed} \emph{(See \cite[Lemma 2.23]{BG}.)}
If $f$ is a Lipschitz function on a normally embedded H\"older triangle $T$, then the width zone $W_T(\gamma,f)$ is closed
for any arc $\gamma\subset T$.
\end{lemma}

\begin{definition}\label{def:zone-slice}\normalfont
(See \cite[Definition 2.24]{BG}.) Let $T$ be a normally embedded H\"older triangle and $f$ a Lipschitz function on $T$.
If $Z\subset V(T)$ is a zone, we define $Q_f(Z)$
as the set of all exponents $ord_\gamma f$ for $\gamma\in Z$. The zone $Z$ is \emph{elementary} with respect to $f$ if the set of arcs $\gamma\in Z$
such that $ord_\gamma f=q$ is a zone for each $q\in Q_f(Z)$.

For $\gamma\in Z$ and $q=ord_\gamma f$, the \emph{width} $\mu_Z(\gamma,f)$ of $\gamma$ with respect to $f$ is the infimum of exponents of H\"older triangles $T'$ containing
$\gamma$ such that $V(T')\subset Z$ and $Q_f(T')$ is a point.
The \emph{width zone} $W_Z(\gamma,f)$ of $\gamma$ with respect to $f$ is the maximal subzone of $Z$ containing $\gamma$ such that $q=ord_\lambda f$ for all arcs
$\lambda\in W_Z(\gamma,f)$. The order of $W_Z(\gamma,f)$ is $\mu_Z(\gamma,f)$.
For $q\in Q_f(Z)$ let $\mu_{Z,f}(q)$ be the set of exponents $\mu_Z(\gamma,f)$, where $\gamma\in Z$ is any arc such that $ord_{\gamma}f=q$.
It follows from Lemma 3.3 from (\cite{birbrair2014lipschitz}) that, for each $q\in Q_f(Z)$, the set $\mu_{Z,f}(q)$ is finite.
This defines a multivalued \emph{width function} $\mu_{Z,f}: Q_f(Z)\to \F\cup \{\infty\}$.
If $Z$ is an elementary zone with respect to $f$ then the function $\mu_{Z,f}$ is single valued.

We say that $Z$ is a \emph{pizza slice zone} associated with
$f$ if it is elementary with respect to $f$ and unless $Q_f(Z)$ is a point $\mu_{Z,f}(q)=aq+b$ is an affine function on
$Q_f(Z)$.
\end{definition}

\begin{lemma}\label{depth} \emph{(See \cite[Lemma 2.25]{BG}.)}
Let $f$ be a Lipschitz function on a normally embedded H\"older triangle $T$.
Let $\gamma$ be an interior arc of $T$, so that $T=T'\cup T''$ and $T'\cap T''=\{\gamma\}$.
Then either $\mu_{T'}(\gamma,f)=\mu_{T''}(\gamma,f)$ and $\nu_T(\gamma,f)=\mu_T(\gamma,f)= \mu_{T'}(\gamma,f)=\mu_{T''}(\gamma,f)$, or
$\nu_T(\gamma,f)=\max(\mu_{T'}(\gamma,f),\mu_{T''}(\gamma,f))>\mu_T(\gamma,f)$.
In both cases, $D_T(\gamma,f)$ is a closed perfect zone.
\end{lemma}

\begin{lemma}\label{MP} \emph{(See \cite[Proposition 2.30]{BG}.)}
Let $f$ be a non-negative Lipschitz function on a normally embedded H\"older triangle $T=T(\gamma_1,\gamma_2)$, oriented from $\gamma_1$ to $\gamma_2$.
There exists a unique finite family $\{D_\ell\}_{\ell=0}^p$ of disjoint zones $D_\ell\subset V(T)$ with the following properties:\newline
\emph{\bf 1.} The singular zones $D_0=\{\gamma_1\}$ and $D_p=\{\gamma_2\}$ are the boundary arcs of $T$.\newline
\emph{\bf 2.} For $0<\ell<p, D_\ell$ is a closed perfect $\nu_\ell$-zone, where $\nu_\ell=\nu_T(\gamma,f)$ for any arc $\gamma\in D_\ell$.
In particular, $D_\ell$ is a $q_\ell$-order zone for $f$, where $q_\ell=ord_\gamma f$.
Moreover, $D_\ell$ is a maximal perfect $q_\ell$-order zone for $f$: if $Z\supseteq D_\ell$ is a perfect $q_\ell$-order zone for $f$,
then $Z=D_\ell$.\newline
\emph{\bf 3.} Any choice of arcs $\lambda_\ell\in D_\ell$ defines a minimal pizza  $\Lambda = \{T_\ell\}_{\ell=1}^p$, where $T_\ell=T(\lambda_{\ell-1},\lambda_\ell)$ on $T$
associated with $f$.\newline
\emph{\bf 4.} Any minimal pizza on $T$ associated with $f$ can be obtained as a decomposition $\{T_\ell\}_{\ell=1}^p$ of $T$ defined by some choice of arcs $\lambda_\ell\in
D_\ell$ for $\ell=0,\ldots,p$.
\end{lemma}

\begin{definition}\label{def:MP}\normalfont The zones $D_\ell$ in Lemma \ref{MP} are called \emph{pizza zones} of a minimal pizza on $T$
associated with $f$.
\end{definition}

\begin{corollary}\label{zone-slice} Let $\{D_\ell\}_{\ell=0}^p$ be the family of pizza zones of a minimal pizza
 $\Lambda = \{T_\ell\}_{\ell=1}^p$ on $T$ associated with $f$, where $T_\ell=T(\lambda_{\ell-1},\lambda_\ell)$, as in Lemma \ref{MP}.
For each $\ell=1,\ldots,p$, the set $Y_\ell=D_{\ell-1}\cup D_\ell\cup V(T_\ell)$ is a pizza slice zone associated with $f$,
independent of the choice of arcs $\lambda_\ell\in D_\ell$.
Moreover, $Y_\ell$ is a maximal pizza slice zone: if $Y\supseteq Y_\ell$ is a pizza slice zone associated with $f$, then $Y=Y_\ell$.
\end{corollary}

\begin{definition}\label{zone-order}\normalfont
If $T=T(\gamma_1,\gamma_2)$ is a H\"older triangle oriented from $\gamma_1$ to $\gamma_2$,
then its Valette link $V(T)$ is totally ordered: for two arcs $\lambda_1\ne\lambda_2$ in $V(T)$,
we define $\lambda_1\prec\lambda_2$ when orientation of $T$ induces orientation from $\lambda_1$ to $\lambda_2$ on $T(\lambda_1,\lambda_2)$.
This order of $V(T)$ defines a partial order on the set of zones in $V(T)$: for two zones $Z_1\ne Z_2$ in $V(T)$, we define
$Z_1\prec Z_2$ when either $Z_2\setminus Z_1$ is a zone and $\lambda_1\prec\lambda_2$ for any arcs $\lambda_1\in Z_1$ and $\lambda_2\in Z_2\setminus Z_1$, or when
$Z_1\setminus Z_2$ is a zone and $\lambda_1\prec\lambda_2$ for any arcs $\lambda_1\in Z_1\setminus Z_2$ and $\lambda_2\in Z_2$.
If $\Lambda$ is a pizza on $T$, then this partial order defines a total order on the set of all pizza zones $D_\ell$ and $Y_\ell$ of $\Lambda$,
consistent with the total order $\lambda_{\ell-1}\prec T_\ell\prec\lambda_\ell$ on the disjoint union of the sets of arcs and
pizza slices of $\Lambda$ induced by orientation of $T$.
If $Z$ is a zone in $V(T)$ and $T_\ell$ is a pizza slice of $\Lambda$, we write $Z\prec T_\ell$ or $T_\ell\prec Z$ when
$Z\prec V(T_\ell)$ or $V(T_\ell)\prec Z$.
\end{definition}

\section{Maximal exponent zones and the permutation $\sigma$}\label{sec:sigma}

We consider normally embedded H\"older triangles $T$ and $T'$. Let $f:T\to \R$ and $g:T'\to \R$ be the distance functions and let $\Lambda$ and $\tilde\Lambda$ be the corresponding pizzas.

\begin{definition}\label{def:tord-tord}\normalfont
 A pair of arcs $(\gamma,\gamma')$, where $\gamma\subset T$ and
$\gamma'\subset T'$, is called \emph{normal} when $tord(\gamma,T')=tord(\gamma,\gamma')=tord(\gamma',T)$.
A pair $(T,T')$ of normally embedded H\"older triangles $T=T(\gamma_1,\gamma_2)$ and $T'=T(\gamma'_1,\gamma'_2)$, oriented from
$\gamma_1$ to $\gamma_2$ and from $\gamma'_1$ to $\gamma'_2$ respectively, is called a \emph{normal pair}
if $(\gamma_1,\gamma'_1)$ and $(\gamma_2,\gamma'_2)$ are normal pairs, i.e., the following condition is satisfied:
 \begin{equation}\label{tord-tord}
tord(\gamma_1,T')=tord(\gamma_1,\gamma'_1)=tord(\gamma'_1,T),\;\; tord(\gamma_2,T')=tord(\gamma_2,\gamma'_2)=tord(\gamma'_2,T).
\end{equation}
\end{definition}

\begin{definition}\label{maxmin}\normalfont (See \cite[Definition 4.2]{BG}.)
Let $T(\gamma_1,\gamma_2)$ be a normally embedded H\"older triangle, oriented from $\gamma_1$ to $\gamma_2$. Let $\Lambda=\{T_\ell\}_{\ell=1}^p$ a minimal pizza on $T$, associated with a non-negative Lipschitz function $f$ defined on $T$.
Let $\{D_\ell\}_{\ell=0}^p$ be the pizza zones of $\Lambda$, ordered according to the orientation of $T$,
and let $q_\ell=tord(D_\ell,T')=ord_{\lambda_\ell}f$.
A pizza zone $D_\ell$ is called a \emph{maximal exponent zone} for $f$ (or simply a \emph{maximum zone})
and $\lambda_\ell\in D_\ell$ is called a \emph{maximum arc} of $\Lambda$ if one of the following inequalities holds:
\begin{equation}\label{eq:max}
\begin{split}
 \ell=0,&\quad\beta_1<q_0\ge q_1,\\
 \ell=p,&\quad\beta_p<q_p\ge q_{p-1},\\
  0<\ell<p,&\quad \max(\beta_{\ell},\beta_{\ell+1})<q_\ell\ge\max(q_{\ell-1},q_{\ell+1}).
 \end{split}
 \end{equation}
If a zone $D_\ell$ is not a maximum zone, it is called a \emph{minimal exponent zone} for $f$ (or simply a \emph{minimum zone})
and an arc $\lambda_\ell$ of $\Lambda$ is called a \emph{minimum arc} if either $\ell=0$ and $q_0\le q_1$, or $0<\ell<p$ and $q_\ell\le\min(q_{\ell-1},q_{\ell+1})$, or
$\ell=p$ and $q_p\le q_{p-1}$.
In particular, each boundary arc of $T$ is either a maximum zone or a minimum zone.\newline
If $\mathcal P$ is a minimal abstract pizza corresponding to $\Lambda$ (see Definition \ref{geometric-abstract})
then $\ell$ is called a \emph{maximum index} of $\mathcal P$ if conditions (\ref{eq:max}) are satisfied.\newline
If $(T,T')$ is a normal pair of H\"older triangles with distance functions
$f(x)=dist(x,T')$ and $g(x')=dist(x',T)$, we use the following notations:
$\Lambda'=\{T'_{\ell'}\}_{\ell'=1}^{p'}$, where $T'_{\ell'}=T(\lambda'_{\ell'-1},\lambda'_{\ell'})$,  is the  minimal pizza  on $T'$ associated with $g$.
 Maximum and minimum zones $D'_{\ell'}\subset V(T')$ and arcs $\lambda'_{\ell'}$ of $\Lambda'$
are defined replacing $T$ with $T'$ and $f$ with $g$.

\end{definition}

\begin{proposition}\label{sigma} \emph{(See \cite[Proposition 4.4]{BG}.)}
Let $(T,T')$, where $T=T(\gamma_1,\gamma_2)$ and $T'=T(\gamma'_1,\gamma'_2)$, be a normal pair of H\"older triangles.
Let $\M=\{M_i\}_{i=1}^m$ and $\M'=\{M'_{j}\}_{j=1}^{m'}$ be the sets of maximum zones in $V(T)$ and $V(T')$
for minimal pizzas $\Lambda$ and $\Lambda'$ associated with the distance functions $f(x)=dist(x,T')$ and $g(x')=dist(x',T)$,
 ordered according to the orientations of $T$ and $T'$.
Let $\bar q_i=tord(M_i,T')$ and $\bar q'_{j}=tord(M'_{j},T)$.
Then $m'=m$, and there is a one-to-one correspondence $j=\sigma(i)$ between the sets $\M=\{M_i\}_{i=1}^m$ and $\M'=\{M'_j\}_{j=1}^m$,
such that $\mu(M'_j)=\mu(M_i)$ and $tord(M_i,M'_j)=\bar q_i=\bar q'_j$ for $j=\sigma(i)$.
If $\{\gamma_1\}=M_1$ is a maximum zone of $\Lambda$, then $\{\gamma'_1\}=M'_1$ is a maximum zone of $\Lambda'$ and $\sigma(1)=1$.
If $\{\gamma_2\}=M_m$ is a maximum zone of $\Lambda$, then $\{\gamma'_2\}=M'_m$ is a maximum zone of $\Lambda'$ and $\sigma(m)=m$.
\end{proposition}

\begin{definition}\label{characteristic}\normalfont (See \cite[Definition 4.5]{BG}.)
The permutation $\sigma$ of the set $[m]=\{1,\ldots,m\}$ in Proposition \ref{sigma} is called
the \emph{characteristic permutation} of the normal pair $(T,T')$ of H\"older triangles.
\end{definition}

\begin{remark}\label{non-arch}\normalfont Since the H\"older triangles $T$ and $T'$ are normally embedded and the zones $M_i$ and $M'_i$ are closed perfect,
Remark \ref{perfect-zone-remark} implies that $\xi(M_i,M_j) = 1/tord(M_i, M_j)$ and $\xi'(M'_i,M'_j) = 1/tord(M'_i, M'_j)$
(see Definition \ref{tord}) define non-archimedean metrics $\xi$ and $\xi'$ on the sets $\M$ and $\M'$ of the maximum zones of $\Lambda$ and $\Lambda'$.
The characteristic permutation $\sigma$ defines an isometry $\sigma:\M \to\M'$ with respect to these metrics.
\end{remark}

\section{Transverse and coherent pizza slices, the correspondence $\tau$}\label{sec:tau}

As in the previous section we consider normally embedded H\"older triangles $T$ and $T'$. Let $f:T\to \R$ and $g:T'\to \R$ be the distance functions and let $\Lambda$ and $\tilde\Lambda$ be the corresponding minimal pizzas . Let  $\Lambda=\{T_\ell\}_{\ell=1}^p$, where $T_\ell=(\lambda_{\ell-1},\lambda_\ell)$, be the minimal pizza on $T$ associated with $f$.

\begin{definition}\label{def:pizzaslicezone-transverse}\normalfont (See \cite[Definition 4.6]{BG}.)
Let $\{D_\ell\}_{\ell=0}^p$ be the pizza zones (see Lemma \ref{MP}) such that $\lambda_\ell\in D_\ell$.
We have $\mu(D_\ell)=\nu(\lambda_\ell)$ and $q_\ell=ord_{\lambda_\ell} f$.
Let $Y_\ell=D_{\ell-1}\cup V(T_\ell)\cup D_\ell$ be the maximal pizza slice zones associated with $f$ (see Corollary \ref{zone-slice}).
An interior pizza zone $D_\ell$ is called \emph{transverse} if $q_\ell=\mu(D_\ell)$ and \emph{coherent} otherwise.
A pizza slice $T_\ell$, and a pizza slice zone $Y_\ell$, are called \emph{transverse}
if either $\ell=p=1$ and $q_0=q_1\le\beta$ or $Q_\ell=[q_{\ell-1},q_\ell]$ is not a point and the width function
 $\mu_\ell$ on $Q_\ell$ (see Definition \ref{def:width function}) satisfies $\mu_\ell(q)\equiv q$.
 Otherwise, $T_\ell$ and $Y_\ell$ are called \emph{coherent}.
 A boundary arc $\check\gamma$ of $T$ is called \emph{transverse} if $\mu(\check\gamma)=ord_f(\check\gamma)$.
The function $f$ is called \emph{totally transverse} if all pizza slices $T_\ell$ of a minimal pizza associated with $f$ are transverse.
Alternatively, $f$ is totally transverse if either $ord_\gamma f<\beta$ or $ord_\gamma f=\mu_T(\gamma,f)$ for all arcs $\gamma\in V(T)$.\newline
If $(T,T')$ is a normal pair of H\"older triangles, $\Lambda=\{T_\ell\}_{\ell=1}^p$ is the minimal pizza on $T$ associated with $f(x)=dist(x,T')$ and
$\Lambda'=\{T'_{\ell'}\}_{\ell'=1}^{p'}$ is a minimal pizza on $T'$ associated with $g(x')=dist(x',T)$, then transverse pizza zones $D'_{\ell'}$ of $\Lambda'$,
transverse and coherent pizza slices $T'_{\ell'}$, maximal pizza slice zones $Y'_{\ell'}$, and affine functions $\mu'_{\ell'}(q)$ on $Q'_{\ell'}$, are defined similarly,
replacing $T$ with $T'$ and $f$ with $g$.
\end{definition}

\begin{definition}\label{def:totally-transverse} \normalfont
A normal pair $(T,T')$ of H\"older triangles, where $T=T(\gamma_1,\gamma_2)$ and $T'=T(\gamma'_1,\gamma'_2)$,
with the distance functions $f(x)=dist(x,T')$ and $g(x')=dist(x',T)$ on $T$ and $T'$,
respectively, is called \emph{totally transverse} if the distance function $f$ on $T$ is totally transverse.
\end{definition}

\begin{remark}\label{rem:totally-transverse} \normalfont
It follows from \cite[Proposition 4.7]{BG} that a pair $(T,T')$ is totally transverse if, and only if,
the distance function $g$ on $T'$ is totally transverse.
\end{remark}

Let $T=T(\gamma_1,\gamma_2)$ be a normally embedded $\beta$-H\"older triangle and $f$ a totally transverse non-negative Lipschitz function on $T$.
Let $\M=\{M_i\}_{i=1}^m$, be the set of maximum zones in $V(T)$ of a minimal pizza $\Lambda$ on $T$ associated with $f$.
We may assume that at least one maximum zone exists, as $m=0$ appears only in a trivial case when $ord_\gamma f\le\beta$ for all $\gamma\subset T$.
Let $n=m-1$ when both boundary arcs of $T$ are maximum zones, $n=m$ when only one of the boundary arcs of $T$ is a maximum zone, $n=m+1$ otherwise.
Selecting $n+1$ arcs $\theta_0,\ldots,\theta_n$ so that $\theta_0=\gamma_1$, $\theta_n=\gamma_2$, and there is exactly
one arc $\theta_j$ in each maximum zone $M_i$, we define a decomposition of $T$ into $n$ H\"older triangles $\bar T_j=T(\theta_{j-1},\theta_j)$.
Here $i=j+1$ if $\{\gamma_1\}$ is a maximum zone and $i=j$ otherwise.
Let $\bar q_j=ord_{\theta_j} f$ and $\bar\beta_j=\mu(\bar T_j)$. Then
\begin{equation}\label{barq}\begin{split}
\bar q_0&>\bar\beta_1\;\text{when}\;\{\gamma_1\}\;\text{is a maximum zone, otherwise}\;\bar q_0=\bar\beta_1,\\
\bar q_n&>\bar\beta_n\;\text{when}\;\{\gamma_2\}\;\text{is a maximum zone, otherwise}\;\bar q_n=\bar\beta_n,\\
\bar q_j&>\max(\bar\beta_j,\bar\beta_{j+1})\;\;\text{for}\;\; 0<j<n.
\end{split}
\end{equation}

\begin{lemma}\label{determination} Let $T=T(\gamma_1,\gamma_2)$  be a normally embedded $\beta$-H\"older triangle oriented from $\gamma_1$ to $\gamma_2$.
Let $\{\bar\beta_j\}_{j=1}^n$ be exponents such that $\min_j(\bar\beta_j)=\beta$, and $\{\bar q_j\}_{j=0}^n$ be exponents, such that  the inequalities (\ref{barq}) hold. Then here exists a totally
transverse non-negative Lipschitz function $f$ on $T$, unique up to contact Lipschitz equivalence, such that exponents $\bar\beta_j$
and $\bar q_j$ are determined by the maximum zones of a minimal pizza on $T$ associated with $f$.
\end{lemma}

\begin{proof}

Let $p=2n$ if $\bar q_0>\bar\beta_1$ and $\bar q_n>\bar\beta_n$, $p=2n-2$ if $\bar q_0=\bar\beta_1$ and $\bar q_n=\bar\beta_n$, $p=2n-1$ otherwise.
Then one can define a pizza $\{T_\ell\}_{\ell=1}^p$ on $T$ as follows.
Consider decomposition $\{\bar T_j\}_{j=1}^n$ of $T$ into $\bar\beta_j$-H\"older triangles $\bar T_j=T(\bar\theta_{j-1},\bar\theta_j)$.
Decompose each H\"older triangle $\bar T_j$, for $1<j<n$, into two $\bar\beta_j$-H\"older triangles $T^-_j=T(\bar\theta_{j-1},\lambda_j)$ and
$T^+_j=T(\lambda_j,\bar\theta_j)$ by a generic arc $\lambda_j\subset\bar T_j$. Decompose similarly $\bar T_1$ if $\bar q_0>\bar\beta_1$
and $\bar T_n$ if $\bar q_n>\bar\beta_n$. This defines a decomposition of $T$ into $p$ H\"older triangles.
If exponents $q_j=\bar\beta_j$ are assigned to $\lambda_j$, and width
functions $\mu^\pm_j(q)\equiv q$ on the intervals $Q^\pm_j$ are assigned to H\"older triangles $T^\pm_j$,
where $Q^-_j=[\bar q_{j-1},q_j]$ and $Q^+_j=[\bar q_j,q_j]$, then this decomposition
becomes a minimal pizza for a totally transverse function $f$ on $T$, such that the arcs $\bar\theta_j$ either are the boundary arcs of $T$
or belong to maximum zones of the pizza, and the arcs $\lambda_j$ belong to minimum zones of the pizza which are not the boundary arcs of $T$.
According to Theorem \ref{pizza-realization}, such a function $f$ exists.  It is unique up to contact Lipschitz equivalence according to the results of \cite{birbrair2014lipschitz}.
\end{proof}

\begin{definition}\label{def:pizzaslice-oriented}\normalfont
Let $(T,T')$, where $T=T(\gamma_1,\gamma_2)$ and $T'=T(\gamma'_1,\gamma'_2)$, be two normally embedded $\beta$-H\"older triangles satisfying (\ref{tord-tord}), such that $T$
is a pizza slice associated with $f(x)=dist(x,T')$ and $tord(T,T')=tord(\gamma_1,\gamma'_1)>\beta$.
The pair $(T,T')$ is called \emph{positively oriented} if either $T$ is oriented from $\gamma_1$ to $\gamma_2$ and $T'$ from $\gamma'_1$ to $\gamma'_2$, or if $T$ is oriented
from $\gamma_2$ to $\gamma_1$ and $T'$ from $\gamma'_2$ to $\gamma'_1$.
Otherwise, the pair $(T,T')$ is called \emph{negatively oriented}.
\end{definition}

\begin{proposition}\label{pizzaslice-oriented}\emph{(See \cite[Proposition 4.7]{BG}.)}
Let $(T,T')$ be a normal pair of H\"older triangles, and
let $T_\ell$, $D_\ell$, $Q_\ell$, $\mu_\ell$, $T'_{\ell'}$,
$D'_{\ell'}$, $Q'_{\ell'}$, $\mu'_{\ell'}$ be as in Definition \ref{def:pizzaslicezone-transverse}.
Then, for each index $\ell$ such that the pizza slice $T_\ell$ is coherent, there is a unique index $\ell'=\tau(\ell)$ such that
$Q'_{\ell'}=Q_\ell$, $\mu'_{\ell'}(q)\equiv\mu_\ell(q)$ and one of the following two conditions holds:
\begin{multline}\label{tord-plus}
tord(D_\ell,T') = tord(D_\ell,D'_{\ell'}) = tord(D'_{\ell'},T), \\
tord(D_{\ell-1},T') = tord(D_{\ell-1},D'_{\ell'-1}) = tord(D'_{\ell'-1},T);
\end{multline}
\begin{multline}\label{tord-minus}
tord(D_\ell,T') = tord(D_\ell,D'_{\ell'-1}) = tord(D'_{\ell'-1},T), \\
tord(D_{\ell-1},T') = tord(D_{\ell-1},D'_{\ell'}) = tord(D'_{\ell'},T).
\end{multline}
\end{proposition}

\begin{definition}\label{def:tau}\normalfont (See \cite[Definition 4.8]{BG}.)
Let $(T,T')$ be a normal pair of H\"older triangles, and
let $\Lambda$ and $\Lambda'$ be minimal pizzas on $T$ and $T'$ associated with the distance functions $f(x)=dist(x,T')$ and $g(x')=dist(x',T)$ respectively.
According to Proposition \ref{pizzaslice-oriented}, the sets $\mathcal L$ and $\mathcal L'$ of coherent pizza slices
of $\Lambda$ and $\Lambda'$ have the same number of elements, and there is a canonical one-to-one \emph{characteristic correspondence} $\tau: \ell'=\tau(\ell)$ between these
two sets.
A pair $(T_\ell,T'_{\ell'})$ of coherent pizza slices, where $\ell'=\tau(\ell)$ is called \emph{positively oriented} if (\ref{tord-plus}) holds and \emph{negatively oriented}
if (\ref{tord-minus}) holds.
Alternatively, we say that $\tau$ is \emph{positive} (resp., \emph{negative}) on $T_\ell$.
\end{definition}

\begin{definition}\label{upsilon}\normalfont
Let $L$ be the number of coherent pizza slices of $\Lambda$, same as the number of coherent pizza slices of $\Lambda'$.
Since pizza slices of $\Lambda$ and $\Lambda'$ are ordered according to orientations of $T$ and $T'$, respectively, the characteristic correspondence $\tau$ induces a
permutation $\upsilon$ of the set $[L]=\{1,\ldots,L\}$, such that $j=\upsilon(i)$ when $T_\ell$ is the $i$-th coherent pizza slice of $\Lambda$
and $T'_{\tau(\ell)}$ is the $j$-th coherent pizza slice of $\Lambda'$.
Since $\tau$ is a signed correspondence, it defines a sign function
$s:[L]\to\{+,-\}$, where $s(l)=``+"$ (resp., $s(l)=``-"$) if $\tau$ is positive (resp., negative) on the $l$-th coherent pizza slice $T_\ell$ of $\Lambda$.
\end{definition}

\begin{proposition}\label{sigma=tau} \emph{(See \cite[Proposition 4.10]{BG}.)}
Let $(T,T')$ be a normal pair of H\"older triangles.
Let $\Lambda$ and $\Lambda'$ be minimal pizzas on $T$ and $T'$ associated with the distance functions $f(x)=dist(x,T')$ and $g(x')=dist(x',T)$.
Let $T_\ell$ and $T'_{\ell'}$ respectively, where $\ell'=\tau(\ell)$, be coherent pizza slices
such that a pizza zone $D$ of $\Lambda$, containing one of the boundary arcs of $T_\ell$ (say $D=D_\ell$, the case $D=D_{\ell-1}$ being similar) is a maximum zone $M_i$ of $\Lambda$.
Then the pizza zone $D'$ of $\Lambda'$ containing one of the boundary arcs of $T'_{\ell'}$ (either $D'=D'_{\ell'}$ when $\tau$ is positive on $T_\ell$ or $D'=D'_{\ell'-1}$ when $\tau$ is
negative) is a maximum zone $M'_{j}$ of $\Lambda'$, where $j=\sigma(i)$.
\end{proposition}

\begin{definition}\label{def:adjacent}\normalfont
Let $T_\ell$ be a coherent pizza slice of $\Lambda$ such that $q_\ell\ge q_{\ell-1}$.
A maximum zone $M_i$ of $\Lambda$ is called \emph{right-adjacent} to $T_\ell$ if $tord(\lambda_\ell,M_i)\ge q_\ell$.
If there are no maximum zones of $\Lambda$ right-adjacent to $T_\ell$, then $T_\ell$ is called \emph{tied on the right}.
Similarly, if $T_\ell$ is a coherent pizza slice of $\Lambda$ such that $q_{\ell-1}\ge q_\ell$, then a
maximum zone $M_i$ of $\Lambda$ is called \emph{left-adjacent} to $T_\ell$ when $tord(\lambda_{\ell-1},M_i)\ge q_{\ell-1}$.
If  there are no maximum zones left-adjacent to $T_\ell$, then $T_\ell$ is called \emph{tied on the left}.
Note that $T_\ell$ may be tied on both sides, or have both right- and left-adjacent maximum zones, when $q_{\ell-1}=q_\ell$.
\end{definition}

\begin{remark}\label{maximum-near}\normalfont
Let $T_\ell=T(\lambda_{\ell-1},\lambda_\ell)$ be a coherent pizza slice of $\Lambda$ such that $q_{\ell}\ge q_{\ell-1}$,
and let $M_i$ be a maximum zone of $\Lambda$ right-adjacent to $T_\ell$.
Since the pizza zone $D_\ell$ of $\Lambda$ has exponent $\nu_\ell\le q_\ell$, either $\lambda_\ell\in M_i$ or $tord(\lambda_\ell,M_i)=q_\ell$.
Similarly, if $q_{\ell-1}\ge q_\ell$ and $M_i$ is a maximum zone of $\Lambda$ left-adjacent to $T_\ell$, then
either $\lambda_{\ell-1}\in M_i$ or $tord(\lambda_{\ell-1},M_i)=q_{\ell-1}$.
\end{remark}

\begin{remark}\label{rem:tied}\normalfont
If a coherent pizza slice $T_\ell$ of $\Lambda$ is tied on the right, it cannot be the last coherent pizza slice of $\Lambda$:
if $\lambda_\ell=\gamma_2$ is a boundary arc of $T$, then it is a maximum zone adjacent to $T_\ell$,
and if $T_{\ell+1}$ is a transverse pizza slice of $\Lambda$ with $q_{\ell+1}>q_\ell$, then either $\lambda_{\ell+1}$
belongs to a maximum zone adjacent to $T_\ell$ or $T_{\ell+2}$ is a coherent pizza slice of $\Lambda$,
since two consecutive transverse pizza slices of a minimal pizza have either a maximal or a minimal common pizza zone.
Thus, if $T_\ell$ is tied on the right, then either $T_{\ell+1}$ is a coherent pizza slice of $\Lambda$
with $q_{\ell+1}>q_\ell$, or $T_{\ell+1}$ is a transverse pizza slice of $\Lambda$ with $q_{\ell+1}>q_\ell$ and
$T_{\ell+2}$ is a coherent pizza slice of $\Lambda$ with $q_{\ell+2}>q_{\ell+1}$.

If a coherent pizza slice $T_\ell$ of $\Lambda$ is tied on the left, then either $T_{\ell-1}$ is a coherent pizza slice
of $\Lambda$ with $q_{\ell-2}>q_{\ell-1}$, or $T_{\ell-1}$ is a transverse pizza slice of $\Lambda$ with $q_{\ell-2}>q_{\ell-1}$ and
$T_{\ell-2}$ is a coherent pizza slice of $\Lambda$ with $q_{\ell-3}>q_{\ell-2}$.
\end{remark}

\begin{definition}\label{def:tied}\normalfont
If $T_\ell$ is a coherent pizza slice of $\Lambda$ tied on the right, then two coherent pizza slices in Remark \ref{rem:tied}, either $T_\ell$ and $T_{\ell+1}$ or $T_\ell$ and $T_{\ell+2}$, are called \emph{right-tied}.

Similarly, if $T_\ell$ is a coherent pizza slice of $\Lambda$  tied on the left, then two coherent pizza slices in Remark \ref{rem:tied}, either $T_\ell$ and $T_{\ell-1}$ or $T_\ell$ and $T_{\ell-2}$, are called \emph{left-tied}.
\end{definition}

 \begin{definition}\label{def:caravan}\normalfont
A sequence $C=\{T_\ell,\ldots,T_{\ell+k}\}$ (resp., $C=\{T_\ell,\ldots,T_{\ell-k}\}$) of consecutive pizza slices of $\Lambda$, where $T_\ell$ and $T_{\ell+k}$ (resp., $T_\ell$ and $T_{\ell-k}$) are coherent pizza slices,
is called a \emph{rightward caravan} (resp., a \emph{leftward caravan})
if each coherent pizza slice in $C$ is tied on the right (resp., tied on the left), except the last pizza slice $T_{\ell+k}$ (resp., $T_{\ell-k}$)
of $C$ being not tied.
The non-empty set $\mathcal A(C)$ of maximum zones of $\Lambda$ right-adjacent to $T_{\ell+k}$ (resp., left-adjacent to $T_{\ell-k}$) is called the \emph{adjacent set} of a caravan $C$.\newline
If $M_i$ is a maximum zone of $\Lambda$ such that $M_i\prec T_\ell$ (resp., $M_i\prec T_{\ell-k}$), we say that $M_i\prec C$.
If $M_i$ is a maximum zone of $\Lambda$ such that $T_{\ell+k}\prec M_i$ (resp., $T_\ell\prec M_i$), we say that $C\prec M_i$.
In particular $C\prec M_i$ (resp., $M_i\prec C$) for any maximum zone $M_i\in\mathcal A(C)$.\newline
A caravan $C'$ of pizza slices of $\Lambda'$ and its adjacent set $\mathcal A(C')$ are defined similarly,
replacing the pizza slices and maximum zones
of $\Lambda$ with the pizza slices and maximum zones of $\Lambda'$.
\end{definition}

\begin{remark}\label{rem:caravan}\normalfont
It follows from Definitions \ref{def:tied} and \ref{def:caravan} that each coherent pizza slice of $\Lambda$
belongs to at least one caravan.
If two rightwards (resp., leftwards) caravans are not disjoint, then one of them is a subset of another, and
their adjacent sets of maximum zones are the same.
A rightward (resp., leftward) caravan may consist of a single coherent pizza slice $T_\ell$ of $\Lambda$ when $T_\ell$ is not tied on the right (resp., on the left).\newline
If $C$ is a caravan of pizza slices of $\Lambda$, then the set $V(C)=\bigcup_{k:T_k\in C} V(T_k)$ is a zone in $V(T)$.
Partial order on zones in $V(T)$ (see Definition \ref{zone-order}) induces total order on the set of all (rightward and leftward) caravans of pizza slices of $\Lambda$, such
that $C_1\prec C_2$ when $V(C_1)\prec V(C_2)$.
Similarly, partial order on zones in $V(T')$ induces total order on the set of all caravans of pizza slices of $\Lambda'$.\newline
If $C=\{T_\ell,\ldots,T_{\ell+k}\}$ is a rightward caravan of pizza slices of $\Lambda$ and $M_i$ is a maximum zone of $\Lambda$ not right-adjacent to $C$, then either $M_i\prec T_\ell$ or
$M_j\prec M_i$ for any maximum zone $M_j\in\mathcal A(C)$.
Similarly, if $C=\{T_\ell,\ldots,T_{\ell-k}\}$ is a leftward caravan of pizza slices of $\Lambda$ and $M_i$ is a maximum zone of $\Lambda$ not left-adjacent to $C$, then either $T_\ell\prec
M_i$ or $M_i\prec M_j$ for any maximum zone $M_j\in\mathcal A(C)$.
\end{remark}

\begin{proposition}\label{prop:tied}
Let $(T,T')$ be a normal pair of H\"older triangles, and let $\Lambda$ and $\Lambda'$ be minimal pizzas on $T$ and $T'$, respectively,
associated with the distance functions $f(x)=dist(x,T')$ and $g(x')=dist(x',T)$.
Let $\tau$ be the characteristic correspondence between coherent pizza slices of $\Lambda$ and $\Lambda'$.
If two coherent pizza slices of $\Lambda$ (either $T_\ell$ and $T_{\ell+1}$ or $T_{\ell}$ and $T_{\ell+2}$)
are tied, then $\tau$ is either positive on both of these pizza slices, with either $\tau(\ell+1)=\tau(\ell)+1$ or $\tau(\ell+2)=\tau(\ell)+2$,
 or negative on both of these pizza slices, with either $\tau(\ell+1)=\tau(\ell)-1$ or $\tau(\ell+2)=\tau(\ell)-2$.
 Moreover, coherent pizza slices of $\Lambda'$ assigned by $\tau$ to tied pizza slices of $\Lambda$ are also tied.
 If two tied pizza slices of $\Lambda$ belong to a caravan $C$ and the tied pizza slices of $\Lambda'$
 assigned to them by $\tau$ belong to a caravan $C'=\tau(C)$, then the set of maximum zones
 of $\Lambda'$ adjacent to $C'$ is $\mathcal A(C')=\{M'_{\sigma(i)}: M_i\in\mathcal A(C)\}$.
\end{proposition}

\begin{proof}
Consider first the case of two consecutive pizza slices $T_\ell$ and $T_{\ell+1}$.
We may assume that $T_\ell$ is tied on the right, $\tau$ is positive on $T_\ell$,
and $T_{\ell+1}$ is a coherent pizza slice with $q_{\ell+1}>q_\ell$, the other cases being similar.
Let $\ell'=\tau(\ell),\;\ell''=\tau(\ell+1)$, and let $T'_{\ell'}=T(\lambda'_{\ell'-1},\lambda'_{\ell'})$ and $T'_{\ell''}=T(\lambda'_{\ell''-1},\lambda'_{\ell''})$
be coherent pizza slices of $\Lambda'$ corresponding to $T_\ell$ and $T_{\ell+1}$, respectively.

From Definition \ref{pizza-decomp} we have the following inequalities:
$\beta_\ell\le\min(q_{\ell-1},q_\ell)\le q_\ell$ (the last inequality holds since $T_\ell$ is tied on the right).
Since $\Lambda$ is a minimal pizza, $\beta_\ell<q_\ell$ when $q_{\ell-1}=q_\ell$. In any case $\beta_\ell<q_\ell$.
Also, \cite[Proposition 3.9]{BG} implies that $\mu'_{\ell'}(q)\equiv\mu_\ell(q),\;\mu'_{\ell''}(q)\equiv\mu_{\ell+1}(q)$,
$q_\ell=tord(\lambda_\ell,T')=tord(\lambda'_{\ell'},T)$ and $q_{\ell+1}=tord(\lambda_{\ell+1},T')=tord(\lambda'_{\ell''-1},T)$.
In particular, $\beta'_{\ell'}=\beta_\ell<q_\ell$ and $\beta'_{\ell''}=\beta_{\ell+1}$.
From the non-archimedean property of the tangency order, we have $tord(\lambda'_{\ell'},\lambda'_{\ell''})\ge q_\ell$.
Since $\lambda_\ell$ is a common arc of $T_\ell$ and $T_{\ell+1}$, we have
$tord(\lambda_\ell,\lambda'_{\ell'})=tord(\lambda_\ell,\lambda'_{\ell''})=q_\ell$.

Let us show first that $\tau$ is positive on $T_{\ell+1}$.
If $\tau$ is negative on $T_{\ell+1}$, then $tord(\lambda_\ell,\lambda'_{\ell''})=q_\ell>\beta_\ell$
implies $T'_{\ell'}\prec T'_{\ell''}$, since $T'$ is normally embedded.
Thus $\lambda'_{\ell''-1}\subset T(\lambda'_{\ell'},\lambda'_{\ell''})$ and $tord(\lambda'_{\ell'},\lambda'_{\ell''})\ge q_\ell$.
Since $\beta'_{\ell''}=\beta_{\ell+1}\le q_\ell$, this implies $tord(\lambda'_{\ell'},\lambda'_{\ell''})=\beta'_{\ell''}=\beta_{\ell+1}=q_\ell$.
As $q_{\ell+1}>q_\ell$, there is a maximum zone $M'_j\subset V(T(\lambda'_{\ell'},\lambda'_{\ell''}))$ of $\Lambda'$,
such that $tord(M'_j,\lambda_\ell)\ge q_\ell$. If $j=\sigma(i)$, then $tord(M_i,\lambda_\ell)\ge q_\ell$, a contradiction.

We are going to show next that $\ell''=\ell'+1$.
Since $\tau$ is positive on both $T_\ell$ and $T_{\ell+1}$ and $\beta_\ell<q_\ell$, the same argument as above shows that $T'_{\ell'}\prec T'_{\ell''}$.
Also, since $tord(\lambda'_{\ell'},T)=tord(\lambda'_{\ell''-1},T)=q_\ell$ and $tord(\lambda'_{\ell'},\lambda'_{\ell''-1})\ge q_\ell$,
we have $tord(\gamma',T)\ge q_\ell$ for any arc $\gamma'$ in $\tilde T=T(\lambda'_{\ell'},\lambda'_{\ell''-1})$, as $T'$ is normally embedded.
If $\tilde T$ contains an arc $\gamma'$ such that $tord(\gamma',T)>q_\ell$, then there is a
maximum zone $M'_j\subset V(\tilde T)$ of $\Lambda'$, corresponding to a maximum zone $M_i$ of $\Lambda$, where $j=\sigma(i)$, such that $tord(M_i,\lambda_\ell)\ge q_\ell$, in contradiction with $T_\ell$ being tied on the right.

The statement about the sets $\mathcal A(C)$ and $\mathcal A(C')$ of adjacent maximum zones follows from the isometry
 of the characteristic permutation $\sigma$ (see Remark \ref{non-arch}).

If $T_\ell$ and $T_{\ell+2}$ are two coherent pizza slices of $\Lambda$ such that $T_\ell$ is tied on the right,
$T_{\ell+1}$ is a transverse pizza slice, and $\tau$ is positive on $T_\ell$, then
the same arguments as above show that $\tau$ is positive on $T_{\ell+2}$,
the pizza slices $T'_{\ell'}$ and $T'_{\ell''}$ of $\Lambda'$ assigned by $\tau$ to $T_\ell$ and $T_{\ell+2}$
satisfy $T'_{\ell'}\prec T'_{\ell''}$, and $T'_{\ell'}$ is tied on the right.
Since $q_{\ell+1}>q_\ell$, the pizza slices $T'_{\ell'}$ and $T'_{\ell''}$
 are not consecutive pizza slices of $\Lambda'$, thus $\ell''>\ell'+1$.
 Also, $T'_{\ell'+1}$ is not a coherent pizza slice of $\Lambda'$,
 otherwise the same arguments as above applied to tied pizza slices $T'_{\ell'}$ and $T'_{\ell'+1}$
 would imply that $T_{\ell+1}$ is a coherent pizza slice, a contradiction.
 Thus $T'_{\ell'+1}$ is a transverse pizza slice of $\Lambda'$ and $T_{\ell'+2}$ is a coherent pizza slice of $\Lambda'$.
 Since $T'_{\ell''}$ and $T'_{\ell'+2}$ belong to the same rightward caravan $C'$ as $T'_{\ell'}$,
 and the order of coherent pizza slices in the caravans $C$ and $C'$ is the same as the order of the exponents $q$ and $q'$
 of their boundary arcs, we have $\ell''=\ell'+2$. This completes the proof of Proposition \ref{prop:tied}.
\end{proof}

\begin{corollary}\label{cor:tied}
Let $(T,T')$ be a normal pair of H\"older triangles, with the minimal pizzas $\Lambda$ and $\Lambda'$ on $T$ and $T'$ associated
with the distance functions $f(x)=dist(x,T')$ and $g(x')=dist(x',T)$, and let $C$ be a caravan of pizza slices of $\Lambda$.
Then the following properties hold:\newline
{\bf (A)} The characteristic correspondence $\tau$ has the same sign, either positive or negative, on all coherent pizza slices of $C$.
We say that $\tau$ is positive (resp., negative) on $C$.\newline
{\bf (B)} There is a caravan $C'=\tau(C)$ of pizza slices of $\Lambda'$ such that
$\tau$ defines a one-to-one correspondence between coherent pizza slices of $C$ and $C'$,
preserving their order when $\tau$ is positive on $C$, and reversing their order when $\tau$ is negative on $C$.
If $C$ is a rightward (resp., leftward) caravan of pizza slices of $\Lambda$ and $\tau$ is positive on $C$, then $C'=\tau(C)$ is a rightward (resp., leftward) caravan of pizza slices of $\Lambda'$.
Similarly, if $C$ is a rightward (resp., leftward) caravan of pizza slices of $\Lambda$ and $\tau$ is negative on $C$, then $C'=\tau(C)$ is a leftward (resp., rightward) caravan of pizza slices of $\Lambda'$.\newline
{\bf (C)} If $\tau$ is positive on a rightward (resp., leftward) caravan $C$ of pizza slices of $\Lambda$, then a maximum zone $M_i$ of $\Lambda$ is right-adjacent (resp., left-adjacent) to
$C$ if, and only if, the maximum zone $M'_{\sigma(i)}$ of $\Lambda'$ is right-adjacent (resp., left-adjacent) to $C'=\tau(C)$.
Similarly, if $\tau$ is negative on $C$, then a maximum zone $M_i$ of $\Lambda$ is right-adjacent (resp., left-adjacent) to $C$ if, and only if, the maximum zone
$M'_{\sigma(i)}$ of $\Lambda'$ is left-adjacent (resp., right-adjacent) to $C'=\tau(C)$.
Thus the adjacent set $\mathcal A(C)$ of a caravan $C$ of pizza slices of $\Lambda$ is mapped by $\sigma$ to the adjacent set $\mathcal A(C')$ of the caravan $C'=\tau(C)$ of pizza slices of $\Lambda'$.
\end{corollary}

\begin{remark}\label{rem:coherent-zone}\normalfont
If $D_\ell$ is a coherent pizza zone (see Definition \ref{def:pizzaslicezone-transverse})
common to consecutive coherent pizza slices $T_\ell$ and $T_{\ell+1}$ of $\Lambda$,
then \cite[Proposition 3.9]{BG} implies that $T'_{\tau(\ell)}$ and $T'_{\tau(\ell+1)}$ are consecutive coherent pizza slices
of $\Lambda'$ with a common pizza zone $D'_{\ell'}$ of $\Lambda'$ corresponding to $D_\ell$.
Thus $\tau$ must have the same sign on these two pizza slices, with either $\ell'=\tau(\ell)$ and $\ell'+1=\tau(\ell+1)$
when $\tau$ is positive, or $\ell'=\tau(\ell+1)$ and $\ell'+1=\tau(\ell)$ when $\tau$ is negative.
Similarly, if $D_\ell=M_i$ is a maximum zone of $\Lambda$ common to two coherent pizza slices $T_\ell$ and $T_{\ell+1}$,
then Proposition \ref{sigma=tau} implies that $T'_{\tau(\ell)}$ and $T'_{\tau(\ell+1)}$ are consecutive coherent pizza slices
of $\Lambda'$ with a common pizza zone $D'_{\ell'}=M'_{\sigma(i)}$ corresponding to $D_\ell$, thus $\tau$ must have
the same sign on these two pizza slices.
\end{remark}

\begin{definition}\label{def:order}\normalfont
If $C$ is a rightward caravan of pizza slices of $\Lambda$ and $\tau$ is positive on $C$, or if
$C$ is a leftward caravan of pizza slices of $\Lambda$ and $\tau$ is negative on $C$,
let $j_+(C)=\min_{i:M_i\in\mathcal A(C)}\sigma(i)$ and $j_-(C)=j_+(C)-1$.
If $C$ is a rightward caravan of pizza slices of $\Lambda$ and $\tau$ is negative on $C$, or if
$C$ is a leftward caravan of pizza slices of $\Lambda$ and $\tau$ is positive on $C$,
let $j_-(C)=\max_{i:M_i\in\mathcal A(C)}\sigma(i)$ and $j_+(C)=j_-(C)+1$.
\end{definition}

\begin{proposition}\label{prop:order}
If $C$ is a caravan of pizza slices of $\Lambda$ and $C'=\tau(C)$ is the corresponding caravan of pizza slices of $\Lambda'$,
then the number of maximum zones $M'_j$ of $\Lambda'$ such that $M'_j\prec C'$ is $j_-(C)$.
If $C_1$ and $C_2$ are two caravans of pizza slices of $\Lambda$ such that $\tau(C_1)\prec\tau(C_2)$, then $j_-(C_1)\le j_-(C_2)$.
\end{proposition}

\begin{proof}
We may assume that $C$ is a rightward caravan and $\tau$ is positive on $C$, the other cases being similar.
Then $C'=\tau(C)$ is a rightward caravan of pizza slices of $\Lambda'$, according to item (B) of Corollary \ref{cor:tied},
and $C'\prec M'_j$ for all maximum zones $M'_j\in \mathcal A(C')$ of $\Lambda'$. Since $\mathcal A(C)$ is
mapped to $\mathcal A(C')$ by $\sigma$, according to item (C) of Corollary \ref{cor:tied}, the maximum
zone $M'_{j_+}$ has the minimal index $j=j_+$ among all maximum zones $M'_j\in\mathcal A(C')$ of $\Lambda'$.
Since all maximum zones $M'_j$ of $\Lambda'$ which are not in $\mathcal A(C')$ satisfy either $M'_j\prec C'$
or $M'_{j_+}\prec M'_j$, according to Remark \ref{rem:caravan}, there are exactly $j_-=j_+-1$ maximum zones $M'_j$ of $\Lambda'$ such that $M'_j\prec C'$.
\end{proof}

\begin{definition}\label{omega}\normalfont
Let $\K=\M\bigcup\mathcal L$ be the disjoint union of the set $\M$ of maximum zones of $\Lambda$
and the set $\mathcal L$ of coherent pizza slices of $\Lambda$,
and let $\K'=\M'\bigcup\mathcal L'$ be the corresponding set of maximum zones and coherent pizza slices of $\Lambda'$.
Notice that the sets $\mathcal K$ and $\mathcal K'$ are totally ordered according to orientations of $T$ and $T'$, respectively
(see Definition \ref{zone-order}).
Since the sets $\K$ and $\K'$ have the same number of elements $K=m+L$,
the combination of permutations $\sigma$ and $\upsilon$ of the sets $[m]$ and $[L]$, respectively, defines a permutation
$\omega$ of the set $[K]=\{1,\ldots,K\}$, such that $k'=\omega(k)$ either when the
$k$-th element of $\K$ corresponds to the $i$-th maximum zone $M_i$ of $\Lambda$ and the $k'$-th element of $K$ corresponds to
the $\sigma(i)$-th maximum zone $M'_{\sigma(i)}$ of $\Lambda'$,
or when the $k$-th element of $\mathcal K$ corresponds to the $l$-th coherent pizza slice $T_\ell$ of $\Lambda$ and the
$k'$-th element of $K$ corresponds to the $\upsilon(l)$-th coherent pizza slice $T'_{\tau(\ell)}$ of $\Lambda'$.
The permutation $\omega$ of $[K]$ is called the \emph{combined characteristic permutation} of the pair $(T,T')$.
\end{definition}

\begin{remark}\label{omega-remarkk}\normalfont
Let $(T,T')$ be a normal pair of H\"older triangles with the minimal pizzas $\Lambda$ and $\Lambda'$
on $T$ and $T'$ associated with the distance functions $f(x)=dist(x,T')$ and $g(x')=dist(x',T)$, respectively.
Let $(T,T'')$ be another normal pair of H\"older triangles with a minimal
pizza on $T$ associated with the distance function $dist(x,T'')$ combinatorially equivalent to $\Lambda$
and a minimal pizza $\Lambda''$ on $T''$ associated with the distance function $dist(x'',T)$.
Since the sets $\M'$ and $\M''$ of maximum zones of $\Lambda'$ and $\Lambda''$ have the same number of elements $m$,
and the sets $\mathcal L'$ and $\mathcal L''$ of coherent pizza slices of $\Lambda'$ and $\Lambda''$ have the same number of elements $L$,
the disjoint union $\K''$ of the sets $\M''$ and $\mathcal L''$ has the same number of elements $K=m+L$ as the sets
$\K$ and $\K'$ in Definition \ref{omega}.
\end{remark}

\begin{theorem}\label{omega-omega}
For two normal pairs $(T,T')$ and $(T,T'')$ of H\"older triangles in Remark \ref{omega-remarkk}, the following statement holds:
If the characteristic permutation $\sigma$ of the set $[m]$ for the pair $(T,T')$ is the same as the characteristic permutation of $[m]$
for the pair $(T,T'')$,
the permutation $\upsilon$ of the set $[L]$ induced by the characteristic correspondence $\tau$ for the pair $(T,T')$ is the same as the permutation of $[L]$ induced by the characteristic correspondence for the pair $(T,T'')$, and the sign function $s:[L]\to\{+,-\}$ of the pair $(T,T')$ is the same as the sign function of the pair $(T,T'')$, then the combined characteristic permutation $\omega$ of the set $[K]$ for the pair $(T,T')$ is the same as the
combined characteristic permutation of $[K]$ for the pair $(T,T'')$.
\end{theorem}

\begin{proof}
The permutation $\sigma$
defines the same order on the sets $\M'$ and $\M''$ of maximum zones of $\Lambda'$ and $\Lambda''$, respectively,
consistent with orientations of $T'$ and $T''$, since $\sigma(i)<\sigma(j)$ implies that
$M'_{\sigma(i)}\prec M'_{\sigma(j)}$ and $M''_{\sigma(i)}\prec M''_{\sigma(j)}$.
Similarly, the permutation $\upsilon$ defines the same order on the sets $\mathcal L'$ and $\mathcal L''$ of coherent pizza slices
of $\Lambda'$ and $\Lambda''$, respectively, consistent with orientations of $T'$ and $T''$.
To prove that the permutation $\omega$ is the same for $(T,T')$ and $(T,T'')$, it is enough to show that a maximum zone $M'_j$
of $\Lambda'$ precedes the $k$-th coherent pizza slice of $\Lambda'$ if, and only if, a maximum zone $M''_j$
of $\Lambda''$ precedes the $k$-th coherent pizza slice of $\Lambda''$.
But this follows from Proposition \ref{prop:order}: if $k=\upsilon(l)$ and
$C$ is the caravan of pizza slices of $\Lambda$ containing the $l$-th coherent pizza slice of $\Lambda$,
then the $k$-th coherent pizza slice of $\Lambda'$ belongs to the caravan $C'$ of pizza slices of $\Lambda'$
and the $k$-th coherent pizza slice of $\Lambda''$ belongs to the caravan $C''$ of pizza slices of $\Lambda'$,
such that the number $j_-(C)$ of maximum zones of $\Lambda'$ preceding $C'$ is the same as
the number of maximum zones of $\Lambda''$ preceding $C''$, since $j_-(C)$ is determined by the pizza $\Lambda$, permutation $\sigma$ and
the sign function $s$.
\end{proof}

\begin{remark}\label{omega-remark}\normalfont If conditions of Theorem \ref{omega-omega} are satisfied, we say that $\omega$ is
\emph{determined} by $\Lambda,\;\sigma,\;\upsilon$ and $s$, meaning that $\omega$ is the same for all normal pairs of H\"older triangles for which
$\Lambda,\;\sigma,\;\upsilon$ and $s$ are the same.
In what follows, we are going to use this terminology for other invariants of normal pairs of H\"older triangles.
\end{remark}

\begin{definition}\label{def:allowable}\normalfont
Let $\Lambda$ be a minimal pizza on a normally embedded H\"older triangle $T$ associated with a non-negative Lipschitz function $f$ on $T$.
Let $m$ and $L$ be the numbers of maximum zones and coherent pizza slices of $\Lambda$, respectively.
Let $\sigma$ and $\upsilon$ be permutations of the sets $[m]=\{1,\ldots,m\}$ and $[L]=\{1,\ldots,L\}$, respectively,
and let $s:[L]\to\{+,-\}$ be a sign function on $[L]$. For any caravan $C$ of pizza slices of $\Lambda$, let $\mathcal A(C)$ be the adjacent set
of $C$ (see Definition \ref{def:caravan}), and let $j_-(C)$ be the index defined by $\sigma,\;\upsilon$ and $s$ in Definition \ref{def:order}.
A triple $(\sigma,\upsilon,s)$ is called \emph{allowable} if it satisfies the following conditions:\newline
$\mathbf{(A1)}$. \emph{If the $k$-th and $l$-th coherent pizza slices of $\Lambda$
belong to the same caravan $C$, then either $s(k)=s(l)=+$ and $\upsilon(l)-\upsilon(k)=l-k$, or $s(k)=s(l)=-$ and $\upsilon(l)-\upsilon(k)=k-l$.}\newline
$\mathbf{(A2)}$ \emph{If $C_1$ and $C_2$ are two caravans of pizza slices of $\Lambda$ such that $\tau(C_1)\prec\tau(C_2)$, then $j_-(C_1)\le j_-(C_2)$.}\newline
$\mathbf{(A3)}$ \emph{If $D_\ell$ is a coherent pizza zone or a maximum zone of $\Lambda$ adjacent to
the $k$-th and $(k+1)$-st coherent pizza slices of $\Lambda$, then $s(k)=s(k+1)$.}\newline
Note that these conditions are satisfied when $\sigma,\;\upsilon$ and $s$ are defined for a distance function $f(x)=dist(x,T')$ on $T$,
when $(T,T')$ is a normal pair of H\"older triangles (see Proposition \ref{prop:order} and Remark \ref{rem:coherent-zone}).
In particular, these conditions are necessary to define the combined
characteristic permutation $\omega$ of the pair $(T,T')$ (see Definition \ref{omega}).
\end{definition}

\begin{proposition}\label{prop:allowable}
Let $(\sigma,\upsilon,s)$ be an allowable triple, as in Definition \ref{def:allowable}, for a minimal pizza $\Lambda$
associated with a non-negative Lipschitz function $f$ on a normally embedded H\"older triangle $T$.
Let $\mathcal K$ be the disjoint union of the set $\M$ of $m$ maximum zones of $\Lambda$ and the set $\mathcal L$ of $L$ coherent pizza slices of $\Lambda$, ordered
according to orientation of $T$.
Then there exists a unique permutation $\omega$ of the set $[K]=\{1,\ldots,K\}$, where $K=m+L$, compatible with the permutations $\sigma$ and $\upsilon$
on the subsets of $\mathcal K$ corresponding to $\M$ and $\mathcal L$, respectively (see Remark \ref{rem:allowable} below)
and for any $k\in[K]$ such that the $k$-th element of $\mathcal K$ corresponds to a coherent pizza slice
of $\Lambda$ belonging to a caravan $C$, the number of indices $l\in[K]$ corresponding to maximum zones of $\Lambda$,
 such that $\omega(l)<\omega(k)$, is equal to $j_-(C)$.\newline
If the triple $(\sigma,\upsilon,s)$ is defined for a minimal pizza $\Lambda$ on $T$ associated with the distance function $f(x)=dist(x,T')$,
where $(T,T')$ is a normal pair of H\"older triangles,
then $\omega$ is the combined characteristic permutation of the pair $(T,T')$ (see Definition \ref{omega}).
\end{proposition}

\begin{remark}\label{rem:allowable}\normalfont
In Proposition \ref{prop:allowable}, ``compatible'' means the following:\newline
For any indices $k$ and $l$ of $[K]$ corresponding to maximum zones $M_i$ and $M_j$ of $\Lambda$ (resp., to the $i$-th and $j$-th
coherent pizza slices of $\Lambda$), we have $\omega(k)<\omega(l)$ if, and only if, $\sigma(i)<\sigma(j)$ (resp., $\upsilon(i)<\upsilon(j))$.
\end{remark}

\begin{proof}[Proof of Proposition \ref{prop:allowable}]
Let $K_m$ and $K_L$ be the subsets of indices $k\in [K]$ corresponding to the maximum zones and coherent pizza slices of $\Lambda$, respectively.
Since $\omega$ is compatible with $\sigma$ and $\upsilon$,
the order of indices $\omega(k)$ in the subsets $\omega(K_m)$ and $\omega(K_L)$ of $[K]$ is determined by the permutations $\sigma$ and $\upsilon$, respectively.
In particular, for each $k\in K_L$, the set $K_L(k)$ of indices $l\in K_L$, such that $\omega(l)<\omega(k)$, is known.
Thus the permutation $\omega$ would be completely defined if, for each $k\in K_L$, the set $K_m(k)$ of indices $i\in K_m$,
such that $\omega(i)<\omega(k)$, is known.
If $k\in K_L$ corresponds to a $j$-th coherent pizza slice of $\Lambda$ belonging to a caravan $C$ of pizza slices of $\Lambda$,
such that either $C$ is a rightward caravan and $s(j)=+$ or $C$ is a leftward caravan and $s(j)=-$, then, according to Definition \ref{def:order},
the set $K_m(k)$ consists of all indices of $K_m$ corresponding to the maximum zones $M_l$ of $\Lambda$ such that
$\sigma(l)<\min_{i\in\mathcal A(C)}\sigma(i)$.
Condition $\mathbf{(A1)}$ of Definition \ref{def:allowable} implies that the value $s(j)$ is the same for all coherent pizza slices of $\Lambda$
belonging to the same caravan $C$.
Similarly, if $k\in K_L$ corresponds to a $j$-th coherent pizza slice of $\Lambda$ belonging to a caravan $C$ of pizza slices of $\Lambda$,
such that either $C$ is a rightward caravan and $s(j)=-$ or $C$ is a leftward caravan and $s(j)=+$, then
the set $K_m(k)$ consists of all indices of $K_m$ corresponding to the maximum zones $M_l$ of $\Lambda$ such that
$\sigma(l)\le\max_{i\in\mathcal A(C)}\sigma(i)$.
Condition $\mathbf{(A2)}$ of Definition \ref{def:allowable} implies that $K_m(j)\subseteq K_m(j')$ when $\upsilon(j)<\upsilon(j')$.
Thus $\omega$ is determined by $\Lambda,\;\sigma,\;\upsilon$ and $s$.\newline
If $(T,T')$ is a normal pair of H\"older triangles, $\Lambda$ is a minimal pizza associated with the distance function $f(x)=dist(x,T')$ on $T$,
and $\sigma,\;\upsilon$ and $s$ are defined by the characteristic permutation $\sigma$ and characteristic correspondence $\tau$ of the pair $(T,T')$,
then the permutation $\omega$ is the combined characteristic permutation of the pair $(T,T')$, due to Definition \ref{omega} and Theorem \ref{omega-omega}.
\end{proof}

\begin{remark}\label{rem:tau}\normalfont
The characteristic correspondence $\tau$ can be extended to a correspondence between the pizza zones of $\Lambda$ and $\Lambda'$ adjacent
to coherent pizza slices as follows:
If $T_\ell$ is a coherent pizza slice of $\Lambda$ and $T'_{\ell'}=\tau(T_\ell)$, we can define
$\tau(D_{\ell-1})=D'_{\ell'-1}$ and $\tau(D_{\ell})=D'_{\ell'}$ (resp., $\tau(D_{\ell-1})=D'_{\ell'}$ and
$\tau(D_{\ell})=D'_{\ell'-1}$) if $\tau$ is positive (resp., negative) on $T_{\ell}$.\newline
Proposition \ref{pizzaslice-oriented} implies that this correspondence preserves pizza toppings: $\mu(\tau(D_{\ell-1}))=\mu(D_{\ell-1})$ and
$\mu(\tau(D_{\ell}))=\mu(D_{\ell}),\;tord(D_{\ell-1},\tau(D_{\ell-1}))=q_{\ell-1}$ and $tord(D_{\ell},\tau(D_{\ell}))=q_{\ell}$.\newline
If $\tau$ is positive on $T_\ell$, then $q'_{\ell'-1}=q_{\ell-1},\;q'_{\ell'}=q_\ell,\;\mu'_{\ell'}(q'_{\ell'-1})=\mu_\ell(q_{\ell-1})$ and
$\mu'_{\ell'}(q'_{\ell'})=\mu_\ell(q_\ell)$.\newline
If $\tau$ is negative on $T_\ell$, then $q'_{\ell'-1}=q_\ell,\;q'_{\ell'}=q_{\ell-1},\;\mu'_{\ell'}(q'_{\ell'-1})=\mu_\ell(q_\ell)$ and
$\mu'_{\ell'}(q'_{\ell'})=\mu_\ell(q_{\ell-1})$.\newline
This correspondence between the pizza zones of $\Lambda$ and $\Lambda'$ is not necessarily one-to-one:
a pizza zone of $\Lambda$ common to two coherent pizza slices may be ``split,'' assigned by $\tau$ to two different
pizza zones of $\Lambda'$,
and two pizza zones of $\Lambda$ may be assigned to the same ``split'' pizza zone of $\Lambda'$.
A boundary arc of $T$ adjacent to a coherent pizza slice of $\Lambda$ may be assigned by $\tau$ to an interior pizza zone of $\Lambda'$,
and an interior pizza zone of $\Lambda$ may be assigned to a boundary arc of $T'$ adjacent to a coherent pizza slice of $\Lambda'$.
However, the correspondence between the pizza zones of $\Lambda$ and $\Lambda'$ defined by $\tau$ is one-to-one
on coherent pizza zones and on the maximum zones (see Remark \ref{rem:coherent-zone}),
and on the pizza zones common to tied coherent triangles, due to Proposition \ref{prop:tied}.\newline
To recover one-to-one correspondence between \emph{essential} pizza zones of $\Lambda$ and $\Lambda'$ (see Definition \ref{def:primary})
compatible with $\sigma$ on the maximum zones and assigning the boundary arcs of $T'$ to the boundary arcs of $T$,
 we are going to define in Section \ref{sec:blocks-general} \emph{pre-pizzas} $\tilde\Lambda$ and $\tilde\Lambda'$ (see Definition \ref{def:pre-pizza})
 removing \emph{non-essential} arcs from pizzas $\Lambda$ and $\Lambda'$,  and
 \emph{twin pre-pizzas} $\check\Lambda$ and $\check\Lambda'$ (see Definition \ref{def:twin-pre-pizza}) expanding pre-pizzas $\tilde\Lambda$ and $\tilde\Lambda'$
 by \emph{twin arcs}.
\end{remark}

\begin{definition}\label{compatible}\normalfont
Let $(T,T')$ be a normal pair of H\"older triangles, and let
$\Lambda=\{T_\ell\}_{\ell=1}^p$ and $\Lambda'=\{T'_{\ell'}\}_{\ell=1}^{p'}$, where $T_\ell=T(\lambda_{\ell-1},\lambda_\ell)$ and
 $T'_{\ell'}=T(\lambda'_{{\ell'}-1},\lambda'_{\ell'})$,
 be minimal pizzas on $T$ and $T'$ associated with the distance functions $f$ and $g$.
 Then $\Lambda$ and $\Lambda'$ are called \emph{compatible} if, for any coherent pizza zones $D_{\ell}$ and $D'_{\ell'}=\tau(D_{\ell})$
  (see Remark \ref{rem:tau}) the  pair of arcs $(\lambda_\ell,\lambda'_{\ell'})$, where $\lambda_{\ell} \in D_\ell$ and
  $\lambda'_{\ell'} \in D'_{\ell'}$, is normal:
 \begin{equation}\label{eq:compatible}
 tord(\lambda_{\ell},\lambda'_{\ell'})= q_{\ell}=q'_{\ell'}.
 \end{equation}
It follows from \cite[Proposition 3.9]{BG} that, for any normal pair $(T,T')$ of H\"older triangles,
there exist compatible minimal pizzas $\Lambda$ and $\Lambda'$ associated with the distance functions $f$ and $g$.
Since the sets of coherent pizza zones in pizzas, pre-pizzas and twin pre-pizzas are the same,
pairs of pre-pizzas $\tilde\Lambda$ and $\tilde\Lambda'$, and pairs of twin pre-pizzas $\check\Lambda$ and $\check\Lambda'$,
corresponding to compatible pairs of pizzas $\Lambda$ and $\Lambda'$, are also compatible.\newline
In what follows, all pairs of pizzas, pre-pizzas and twin pre-pizzas
associated with the distance functions on normal pairs of H\"older triangles are assumed to be compatible.
\end{definition}

\begin{remark}\label{rem:compatible}\normalfont
If $D_\ell=M_i$ is a transverse maximum pizza zone of $\Lambda$ and $D_{\ell'}=M'_{\sigma(i)}$, or if $D_\ell$ is a transverse pizza zone of $\Lambda$
adjacent to a coherent pizza slice and $D_{\ell'}=\tau(D_\ell)$ (see Remark \ref{rem:tau}) then $(\lambda_\ell,\lambda'_{\ell'})$ is a normal pair of arcs (see Definition \ref{def:tord-tord}) for any $\lambda_\ell\in D_\ell$ and $\lambda'_{\ell'}\in D_{\ell'}$.
\end{remark}

\section{The $\sigma\tau$-pizza invariant}\label{sec:invariant}

Let $(T,T')$ be a normal pair of $\beta$-H\"older triangles $T=T(\gamma_1,\gamma_2)$ and $T'=T(\gamma'_1,\gamma'_2)$.
Let $\{D_\ell\}_{\ell=0}^p$ and $\{D'_{\ell'}\}_{\ell'=0}^{p'}$ be the sets of pizza zones
for the minimal pizzas $\Lambda_T=\{T_\ell\}_{\ell=1}^p$ and $\Lambda'_{T'}=\{T'_{\ell'}\}_{\ell'=1}^{p'}$ associated with the distance functions $f(x)=dist(x,T')$ and
$g(x')=dist(x',T)$, where
$T_\ell=T(\lambda_{\ell-1},\lambda_\ell),\;T'_{\ell'}=T(\lambda'_{\ell'-1},\lambda'_{\ell'})$, and the arcs $\lambda_\ell\in D_\ell$ and $\lambda'_{\ell'}\in D_{\ell'}$ are
selected so that the pizzas $\Lambda_T$ and $\Lambda'_{T'}$ are compatible (see Definition \ref{compatible}).
Let $\{M_i\}_{i=1}^m$ and $\{M'_j\}_{j=1}^m$ be the maximum zones in $V(T)$ and $V(T')$ for the functions $f$ and $g$.
Let $j=\sigma(i)$ be the characteristic permutation $\sigma$ of the pair $(T,T')$,
and let $\ell'=\tau(\ell)$ be the characteristic correspondence between the sets of coherent pizza slices of $\Lambda_T$ and $\Lambda'_{T'}$ (see Definition \ref{def:tau}).

\begin{definition}\label{def:sigmatau}\normalfont
The $\sigma\tau$\emph{-pizza} on $T\cup T'$ (see \cite[Definition 4.12]{BG})
consists of the minimal pizzas $\Lambda_T$ and $\Lambda'_{T'}$, characteristic permutation $\sigma$
and characteristic correspondence $\tau$.
\end{definition}

\begin{definition}\label{combinatorial} \normalfont
Two $\sigma\tau$-pizzas $(\Lambda_T,\Lambda'_{T'},\sigma_T,\tau_T)$ on $T\cup T'$ and $(\Lambda_S,\Lambda'_{S'},\sigma_S,\tau_S)$ on $S\cup S'$ are \emph{combinatorially
equivalent} if pizzas
$\Lambda_T$ and $\Lambda_S$ are combinatorially equivalent (see Definition \ref{equivalent_pizza}), pizzas $\Lambda'_{T'}$ and $\Lambda'_{S'}$ are combinatorially equivalent,
$\sigma_T=\sigma_S$ and $\tau_T=\tau_S$.
\end{definition}

It was shown in \cite[Theorem 4.13]{BG} that the $\sigma\tau$-pizza is an outer Lipschitz invariant of a normal pair of H\"older triangles.
The main result of this section is the following theorem (see \cite[Conjecture 4.14]{BG}).

\begin{theorem}\label{complete}
Let $(T,T')$ and $(S,S')$ be two normal pairs of H\"older triangles.
If the $\sigma\tau$-pizza on $T\cup T'$ is combinatorially equivalent to the $\sigma\tau$-pizza on $S\cup S'$,
then there is an orientation-preserving outer bi-Lipschitz homeomorphism $H:T\cup T'\to S\cup S'$ such that $H(T)=S$ and $H(T')=S'$.
\end{theorem}

We prove it first (see Theorem \ref{pre-complete} below) for totally transverse pairs of H\"older triangles.

\begin{lemma}\label{totally-pizza} Let $(T,T')$ be  a totally transverse pair of
$\beta$-H\"older triangles (see Definition \ref{def:pizzaslicezone-transverse}) with
the pizza $\Lambda_T$ on $T$ associated with the distance function $f(x)=dist(x,T')$.
Let $q_\ell=ord_{\lambda_\ell} f$, for $0\le\ell\le p$.
Then, unless $tord(T,T')\le\beta$, one of the boundary arcs  of each H\"older triangle $T_\ell$ belongs to a maximum zone $M=M_{i(\ell)}$ of $\Lambda_T$. \newline
If $\gamma\subset T_\ell$ is an arc such that $\gamma\notin M$, then $ord_\gamma f=tord(\gamma,M)$.
\end{lemma}

\begin{proof} This follows immediately from Definition \ref{def:pizzaslicezone-transverse} and Proposition \ref{prop:width function properties}.
\end{proof}

\begin{remark}\label{rem:totally}\normalfont For a totally transverse pair $(T,T')$, the $\sigma\tau$-pizza does not contain the correspondence $\tau$, as the sets of
coherent pizza slices of $\Lambda_T$ and $\Lambda'_{T'}$ are empty.
Also, Lemma \ref{totally-pizza} implies that the pizza $\Lambda_T$
is completely determined by exponents $\beta_\ell$ of the pizza slices $T_\ell$ and exponents $q_\ell=ord_{\lambda_\ell} f$.
Similarly, the pizza $\Lambda'_{T'}$
is completely determined by exponents $\beta'_{\ell'}$ of the pizza slices $T'_{\ell'}$ and exponents $q'_{\ell'}=ord_{\lambda'_{\ell'}} g$.
\end{remark}

\begin{lemma}\label{count-order1}
Let $(T,T')$ be a totally transverse pair of $\beta$-H\"older triangles with the pizzas $\Lambda_T$ and $\Lambda'_{T'}$
associated with the distance functions $f(x)=dist(x,T')$ and $g(x')=dist(x',T)$ on $T$ and $T'$, respectively.
Let $\gamma'$ be an arc of a pizza slice $T'_{\ell'}$ of $\Lambda'_{T'}$, and let $q=tord(\gamma',T)$.
Let $Z_{\gamma'}$ be the set of all arcs $\gamma\subset T$ such that $tord(\gamma,\gamma')=q$.
If $q>\beta$ then there is a unique maximum zone $M_i$ of $\Lambda_T$ such that the maximum zone $M'_{\sigma(i)}$ of $\Lambda'_{T'}$
contains a boundary arc of $T'_{\ell'},\; tord(\gamma',M_i)=q$,
and $Z_{\gamma'}$ is the set of all arcs $\gamma\in T$ such that $tord(\gamma,M_i)\ge q$.
Moreover, $tord(\gamma,M_i)=tord(\gamma,\gamma')$ for any arc $\gamma \in V(T)\setminus Z_{\gamma'}$.
\end{lemma}

\begin{proof} Since $(T,T')$ is a totally transverse pair, we have $\mu'_{\ell'}(q)\equiv q$.
It follows from Proposition \ref{prop:width function properties} that, unless $q\le\beta$,
we have $q=tord(\gamma',M')$, where $M'$ is the unique maximal zone of $\Lambda'_{T'}$ containing a boundary arc of $T'_{\ell'}$.
If $M'=M'_{\sigma(i)}$, where $M_i$ is a maximum zone of $\Lambda_T$, then $tord(M_i,M')\ge q$, thus $tord(\gamma',M_i)=q$
by the non-archimedean property of the tangency order, since $q=tord(\gamma',T)$.
Also, $tord(\gamma,M_i)\ge q$ for any arc $\gamma\in Z_{\gamma'}$ by the non-archimedean property of the tangency order.
If $\gamma \in V(T)\setminus Z_{\gamma'}$, then $tord(\gamma,\gamma')<q$, thus $tord(\gamma,M_i)=tord(\gamma,\gamma')$
by the non-archimedean property of the tangency order.
\end{proof}

\begin{theorem}\label{pre-complete}
Let $(T,T')$ and $(S,S')$ be two totally transverse pairs of H\"older triangles.
If the $\sigma\tau$-pizzas of the pairs $(T,T')$ and $(S,S')$ are combinatorially equivalent, then there is an orientation-preserving outer bi-Lipschitz homeomorphism
$H:T\cup T'\to S\cup S'$ such that $H(T)=S$ and $H(T')=S'$.
\end{theorem}

\begin{proof}
Let $\Lambda_T=\{T_\ell\}_{\ell=1}^p$ and $\Lambda'_{T'}=\{T'_{\ell'}\}_{\ell'=1}^{p'}$ be minimal pizzas on $T$ and $T'$ associated with the distance functions
$f(x)=dist(x,T')$ and $g(x')=dist(x',T)$,
where $T_\ell=T(\lambda_{\ell-1},\lambda_\ell)$ and $T'_{\ell'}=T(\lambda'_{\ell'-1},\lambda'_{\ell'})$.
Let $\Lambda_S=\{S_\ell\}_{\ell=1}^p$ and $\Lambda'_{S'}=\{S'_{\ell'}\}_{\ell'=1}^{p'}$,
where $S_\ell=T(\theta_{\ell-1},\theta_\ell)$ and $S'_{\ell'}=T(\theta'_{\ell'-1},\theta'_{\ell'})$, be minimal pizzas on $S$ and $S'$ associated with the distance functions
$\phi(y)=dist(y,S')$ and $\psi(y')=dist(y',S)$, respectively.

Since the pizzas $\Lambda_T$ and $\Lambda_S$ are equivalent, the H\"older triangles $T_\ell$ and $S_\ell$ have the same exponent $\beta_\ell$ for each $\ell$.
Thus, for $1\le \ell\le p$, there exists a bi-Lipschitz homeomorphism $H_\ell:T_\ell\to S_\ell$ preserving the distance to the origin, such that
$H_\ell(\lambda_{\ell-1})=\theta_{\ell-1}$ and $H_\ell(\lambda_\ell)=\theta_\ell$.
Similarly, for $1\le\ell'\le p'$, there exists a bi-Lipschitz homeomorphism $H'_{\ell'}:T'_{\ell'}\to S'_{\ell'}$
preserving the distance to the origin, such that $H'_{\ell'}(\lambda'_{\ell'-1})=\theta'_{\ell'-1}$ and
$H'_{\ell'}(\lambda'_{\ell'})=\theta'_{\ell'}$.
Let us define a map $H:T\cup T'\to S\cup S'$ so that its
restriction to each H\"older triangle $T_\ell$ coincides with $H_\ell$, and its restriction to each H\"older triangle $T'_{\ell'}$ coincides with $H'_{\ell'}$. Since $T$,
$T'$, $S$ and $S'$ are normally embedded, $H$ defines outer bi-Lipschitz homeomorphisms
$T\to S$ and $T'\to S'$. In particular, for any two arcs $\gamma_1$ and $\gamma_2$ in $T$ we have $tord(H(\gamma_1),H(\gamma_2))=tord(\gamma_1,\gamma_2)$, and for any two
arcs $\gamma'_1$ and $\gamma'_2$ in $T'$ we have $tord(H(\gamma'_1),H(\gamma'_2))=tord(\gamma'_1,\gamma'_2)$.

Since $H(\lambda_\ell)=\theta_\ell$ and the pizzas on $T$ and $S$ are equivalent,
we have $tord(\lambda_\ell,T')=ord_{\lambda_\ell} f=ord_{\theta_\ell}\phi=tord(H(\lambda_\ell),S')$ for $0\le\ell\le p$.
Similarly, $tord(\lambda'_{\ell'},T)=tord(H(\lambda'_{\ell'}),S)$ for $0\le\ell'\le p'$.
This implies, in particular, that $H$ maps the maximum zones $M_i$ for $f$ to the maximum zones $N_i$ for $\phi$, and the maximum zones $M'_j$ for $g$ to the maximum zones
$N'_j$ for $\psi$.
Since $\sigma_S=\sigma_T$, we have $\bar q_i=tord(M_i,M'_j))=tord(N_i,N'_j)$, where $j=\sigma_T(i)=\sigma_S(i)$.

Let $\gamma$ and $\gamma'$ be any two arcs in $T$ and $T'$, respectively.
We are going to show that Lemma \ref{count-order1} implies $tord(H(\gamma),H(\gamma'))=tord(\gamma,\gamma')$.
Let $\gamma'\subset T'_{\ell'}$, a pizza slice of $\lambda'_T$, let $q=tord(\gamma',T)$,
and let $Z_{\gamma'}$ be the set of all arcs in $T$ having tangency order $q$ with $\gamma'$.
According to Lemma \ref{count-order1}, $Z_{\gamma'}$ is the set of all arcs $\gamma\subset T$ such that $tord(\gamma,M_i)\ge q$.
Thus $tord(H(\gamma),H(\gamma'))=tord(H(\gamma),N_i)=q$ for all arcs $\gamma\in Z_{\gamma'}$.
If $\gamma'\in M'_{\sigma(i)}$, then $Z_{\gamma'}=M_i$, thus $H(Z_{\gamma'})=\check M_i$ is the set of arcs in $S$ having tangency order $q$ with $H(\gamma')$.
If $\gamma'\notin M'_{\sigma(i)}$, then Lemma \ref{count-order1} applied to $(S,S')$ implies that $tord(\gamma,\gamma')=tord(H(\gamma),H(\gamma'))$ for any two arcs
$H(\gamma)\subset S$ and $H(\gamma')\subset S'$ such that $tord(H(\gamma),H(\gamma'))=q$.
Thus $H(Z_{\gamma'})$ coincides with the set $Z_{H(\gamma')}$ of arcs in $S$ having tangency order $q$ with $H(\gamma')$.
If $\gamma\notin Z_{\gamma'}$, then $H(\gamma)\notin Z_{H(\gamma')}$, and Lemma \ref{count-order1} implies that
$tord(H(\gamma),H(\gamma'))=tord(H(\gamma),N_i)=tord(\gamma,M_i)=tord(\gamma,\gamma')$.
This proves that $tord(H(\gamma),H(\gamma'))=tord(\gamma,\gamma')$ for all arcs $\gamma\subset T$ and $\gamma'\subset T'$.

Since $H$ preserves the tangency orders of any two arcs in $T\cup T'$, it is outer bi-Lipschitz by Proposition \ref{important-proposition}.
This completes the proof of Theorem \ref{pre-complete}.
\end{proof}

\begin{proof}[Proof of Theorem \ref{complete}] The proof is done in two steps. In {\bf Step 1} a homeomorphism $H:T\cup T'\to S\cup S'$
is defined for general pairs of H\"older triangles $(T,T')$ and $(S,S')$ with combinatorially equivalent $\sigma\tau$-pizzas.
In {\bf Step 2} we prove that the homeomorphism $H$ defined in Step 1 is outer bi-Lipschitz.

{\bf Step 1.}
Let $(T,T')$ and $(S,S')$ be two normal pairs of $\beta$-H\"older triangles with combinatorially equivalent $\sigma\tau$-pizzas.\newline
Let $\Lambda_T=\{T_\ell\}_{\ell=1}^p$ and $\Lambda'_{T'}=\{T'_{\ell'}\}_{\ell'=1}^{p'}$,
where $T_\ell=T(\lambda_{\ell-1},\lambda_\ell)$ and $T'_{\ell'}=T(\lambda'_{\ell'-1},\lambda'_{\ell'})$,
be compatible minimal pizzas on $T$ and $T'$
(see Definition \ref{compatible}) associated with the distance functions $f(x)=dist(x,T')$ and $g(x')=dist(x',T)$, respectively.
Let $\{D_\ell\}_{\ell=0}^p$ and $\{D'_{\ell'}\}_{\ell'=0}^{p'}$ be pizza zones of $\Lambda_T$ and $\Lambda'_{T'}$,
respectively, such that $\lambda_\ell\in D_\ell$ and $\lambda'_{\ell'}\in D'_{\ell'}$.\newline
Let $\Lambda_S=\{S_\ell\}_{\ell=1}^p$ and $\Lambda'_{S'}=\{S'_{\ell'}\}_{\ell'=1}^{p'}$,
where $S_\ell=T(\theta_{\ell-1},\theta_\ell)$ and $S'_{\ell'}=T(\theta'_{\ell'-1},\theta'_{\ell'})$,
be compatible minimal pizzas on $S$ and $S'$ associated with the distance functions
$\phi(y)=dist(y,S')$ and $\psi(y')=dist(y',S)$, respectively.
Let $\{\Delta_\ell\}_{\ell=0}^p$ and $\{\Delta'_{\ell'}\}_{\ell'=0}^{p'}$ be pizza zones of $\Lambda_S$ and $\Lambda'_{S'}$, respectively, such that $\theta_\ell\in\Delta_\ell$ and $\theta'_{\ell'}\in \Delta'_{\ell'}$.

Since the pizzas $\Lambda_T$ and $\Lambda_S$ are equivalent, we have $q_\ell=ord_{\lambda_\ell} f = ord_{\theta_\ell} \phi$ for $0\le\ell\le p$.
Since the pizzas $\Lambda'_{T'}$ and $\Lambda'_{S'}$ are equivalent, we have $q'_{\ell'}=ord_{\lambda'_{\ell'}} g = ord_{\theta'_{\ell'}} \psi$
 for $0\le\ell'\le p'$.

Let $D_\ell$ be a coherent interior pizza zone of $\Lambda_T$, i.e., $\mu(D_\ell)=\nu(\lambda_\ell)<q_\ell$.
Then $D_\ell$ is a maximal perfect $q_\ell$-order zone for $f$ (see Lemma \ref{MP}), and \cite[Proposition 3.9]{BG} implies that
there is a unique pizza zone $D'_{\ell'}=\tau(D_\ell)$ of $\Lambda'_{T'}$, a maximal perfect $q_\ell$-order zone for $g$,
such that $\mu(D'_{\ell'})=\mu(D_\ell)$ and $tord(D_\ell,D'_{\ell'})=q_\ell$.
Note that $\tau(D_\ell)$ is well defined, since $D_\ell$ is a coherent zone (see Definition \ref{compatible}).

Since $D_\ell$ is a coherent zone, it is a common pizza zone for two coherent pizza slices
$T_\ell$ and $T_{\ell+1}$ of $\Lambda_T$, thus $\tau(D_\ell)=D'_{\ell'}$ is a common pizza zone for
the pizza slices $T'_{\tau(\ell)}$ and $T'_{\tau(\ell+1)}$.
In particular, $\tau$ has the same sign (either positive or negative) on $T_{\ell}$ and $T_{\ell+1}$,
with $\ell'=\tau(\ell)$ and $\ell'+1=\tau(\ell+1)$ when $\tau$ is positive,
$\ell'=\tau(\ell+1)$ and $\ell'+1=\tau(\ell)$ when $\tau$ is negative.
Conversely, for any coherent pizza zone $D'_{\ell'}$ of $\Lambda'_{T'}$, there is a coherent pizza zone $D_\ell$ of $\Lambda_T$ such that $D'_{\ell'}=\tau(D_\ell)$.

Since pizzas $\Lambda_T$ and $\Lambda'_{T'}$ are compatible,
any coherent pair of pizza slices $(T_\ell,T'_{\tau(\ell)})$ satisfies (\ref{tord-tord}).
In particular, according to \cite[Proposition 3.2 and Theorem 3.20]{BG}, there is an outer bi-Lipschitz homeomorphism $h_\ell:\Gamma_\ell\to T'_{\tau(\ell)}$,
 where $\Gamma_\ell$ is the graph of $f|_{T_\ell}$, such that
$(id,h_\ell):T_\ell\cup\Gamma_\ell\to T_\ell\cup T'_{\tau(\ell)}$ is an outer bi-Lipschitz homeomorphism.
Similarly, since pizzas $\Lambda_S$ and $\Lambda'_{S'}$ are compatible, for any coherent pair $(S_\ell,S'_{\tau(\ell)})$
of pizza slices, there is an outer bi-Lipschitz homeomorphism $\eta_\ell:\Phi_\ell\to S'_{\tau(\ell)}$, where $\Phi_\ell$ is the graph of $\phi|_{S_\ell}$, such that $(id,\eta_\ell):S_\ell\cup\Phi_\ell\to S_\ell\cup S'_{\tau(\ell)}$ is an outer bi-Lipschitz homeomorphism.

To construct the homeomorphism $H:T\cup T'\to S\cup S'$,
we define it first on the arcs $\lambda_\ell$ and $\lambda'_{\ell'}$ of pizzas $\Lambda_T$ and $\Lambda'_{T'}$, so
that $H(\lambda_\ell)=\theta_\ell$ for $\ell=0,\ldots,p$, and $H(\lambda'_{\ell'})=\theta'_{\ell'}$ for $\ell'=0,\ldots,p'$.
Next, we define $H$ on coherent pairs of pizza slices of $\Lambda_T$ and $\Lambda'_{T'}$, so that
$H$ maps each coherent pair of pizza slices $(T_\ell,T'_{\tau(\ell)})$ to the coherent pair of pizza slices $(S_\ell,S'_{\tau(\ell)})$,
consistent with the action of $H$ on the boundary arcs of $T_\ell$ and $T'_{\tau(\ell)}$, and so that restrictions
$H|_{T_\ell}:T_\ell\to S_\ell$ and $H|_{T'_{\tau(\ell)}}: T'_{\tau(\ell)}\to S'_{\tau(\ell)}$ are bi-Lipschitz homeomorphisms.
For a given homeomorphism $H|_{T_\ell}$, we may choose a homeomorphism $H|_{T'_{\tau(\ell)}}$ so that the following diagram is commutative.
\begin{equation}\label{diagram}
\begin{array}{lllll}\medskip
\;\;\;\;S_\ell & \stackrel{\;\;\;(id,\, \phi)\;\;\;}{\longrightarrow}
 & \Phi_\ell & \stackrel{\eta_\ell}{\longrightarrow} & \;\;S'_{\tau(\ell)} \\
{\scriptstyle H|_{T_\ell}} \, \uparrow &  & \;\;\;\;\; \,  &  & \,\,\, \, \uparrow \, {\scriptstyle H|_{T'_{\tau(\ell)}}}\\
\;\;\;\;T_\ell& \stackrel{\;\;\;(id, \,f)\;\;\;}{\longrightarrow} & \Gamma_\ell & \stackrel{h_\ell}{\longrightarrow}& \;\;T'_{\tau(\ell)} \\
\end{array}
\end{equation}
Finally, we define $H$ on transverse pizza slices $T_\ell$ and $T'_{\ell'}$ of $\Lambda$ and $\Lambda'$ as
any bi-Lipschitz homeomorphisms $T_\ell\to S_\ell$ and $T'_{\ell'}\to S'_{\ell'}$
consistent with the action of $H$ on the boundary arcs of $T_\ell$ and $T'_{\ell'}$.

Note that, for a coherent pair $(T,T')$ of H\"older triangles,
such that $T$ is a pizza slice associated with a Lipschitz function $f$
and $T'$ is a graph of $f$, a mapping $H:T\cup T'\to T\cup T'$ that is
bi-Lipschitz on each of the two triangles may be not outer bi-Lipschitz on their union,
even when $H^* f$ is Lipschitz contact equivalent to $f$.

For example, let $T=\{x\ge 0,\,y\ge 0\}$ be the first quadrant in the $xy$-plane and $T'$ the graph of
a function $z=f(x,y)=y^2$ on $T$. Let $H(x,y,0)=(x,y,0)$ and $H(x,y,y^2)=(x,2y,4y^2)$.
If $\gamma=\{x\ge 0,\,y=x^2,\,z\equiv 0\}$ and $\gamma'=\{x\ge 0,\,y=x^2,\,z=x^4\}$
are arcs in $T$ and $T'$, then $tord(\gamma,\gamma')=4$ but $tord(H(\gamma),H(\gamma'))=2$.

{\bf Step 2.} Proof that the homeomorphism $H:T\cup T'\to S\cup S'$ is outer bi-Lipschitz.\newline
Since the H\"older triangles $T,\,T',\,S$ and $S'$ are normally embedded and restrictions of $H$ to pizza slices $T_\ell$ and $T'_{\ell'}$
are bi-Lipschitz homeomorphisms $T_\ell\to S_\ell$ and $T'_{\ell'}\to S'_{\ell'}$ consistent with the action of $H$ on their boundary arcs,
the homeomorphisms $H|_T:T\to T'$ and $H|_S:S\to S'$ are obviously bi-Lipschitz.
According to Proposition \ref{important-proposition}, it is enough to prove that $tord(H(\gamma),H(\gamma'))=tord(\gamma,\gamma')$
for any arcs $\gamma\in V(T)$ and $\gamma'\in V(T')$.

Let us show first that $tord(H(\gamma'),S)=tord(\gamma',T)$ for any arc $\gamma'\subset T'$.
Let $q=tord(\gamma',T)$. If $\gamma'$ belongs to a coherent pizza slice $T'_{\ell'}$ of $\Lambda'_{T'}$, where $\ell'=\tau(\ell)$,
the statement follows from (\ref{diagram}),
as $tord(\gamma',T)=tord(\gamma',T_\ell$) (see \cite[Proposition 4.7]{BG}) and $H$ is compatible with the graph structure of the pair $(T_\ell,T'_{\ell'})$.
If $\gamma'$ belongs to a transverse pizza slice $T'_{\ell'}$ of $\Lambda'_{T'}$, then either $\lambda'_{\ell'-1}$ or $\lambda'_{\ell'}$ is the supporting arc $\tilde\lambda'_{\ell'}$ of $T'_{\ell'}$ (see Definition \ref{def:pizza-slice}).
We may assume that $\tilde\lambda'_{\ell'}=\lambda'_{\ell'}$, the case of $\tilde\lambda'_{\ell'}=\lambda'_{\ell'-1}$ being similar.
Then, since $T'_{\ell'}$ is transverse, we have either $tord(\gamma',\lambda'_{\ell'})>tord(\lambda'_{\ell'},T)=q$, or
$q=tord(\gamma',\lambda'_{\ell'})<tord(\lambda'_{\ell'},T)$.
Since $H|_{T'_{\ell'}}:T'_{\ell'}\to S'_{\ell'}$ is a bi-Lipschitz map such that $H(\lambda'_{\ell'})=\theta'_{\ell'}$ is a supporting arc for
the transverse pizza slice $S'_{\ell'}$ of $\Lambda'_{S'}$,
we have either $tord(H(\gamma'),\theta'_{\ell'})>tord(\theta'_{\ell'}$,
in which case $tord(H(\gamma'),S)=tord(\theta'_{\ell'},S)=tord(\lambda'_{\ell'},T)=q$,
or $tord(H(\gamma',S)=tord(H(\gamma'),\theta'_{\ell'})=tord(\gamma',\lambda'_{\ell'}=q$.\newline
In any case, we have $tord(H(\gamma'),S)=tord(\gamma',T)=q$.

To show that $tord(\gamma',\gamma)=tord(H(\gamma'),H(\gamma))$ for any arcs $\gamma\subset T$ and $\gamma'\subset T'$,
note that, since pizzas $\Lambda_T$ and $\Lambda_S$ are combinatorially equivalent,
we have the same toppings $Q_\ell$ and $\mu_\ell(q)$ for $T_\ell$ and $S_\ell=H(T_\ell)$, for each $\ell=1,\ldots,p$.
Similarly, we have the same toppings $Q'_{\ell'}$ and $\mu'_{\ell'}(q)$ for $T'_{\ell'}$ and $S'_{\ell'}=H(T'_{\ell'})$, for each $\ell'=1,\ldots,p'$.
In particular, the homeomorphism $H$ maps coherent pizza slices of $\Lambda_T$ and $\Lambda'_{T'}$
 to coherent pizza slices of $\Lambda_S$ and $\Lambda'_{S'}$, and transverse pizza slices of $\Lambda_T$ and $\Lambda'_{T'}$
 to transverse pizza slices of $\Lambda_S$ and $\Lambda'_{S'}$, respectively.

Let $q=tord(\gamma',T)$, where $\gamma'$ is an arc in $T'$. We are going to choose an arc $\lambda$ in $T$, such that $tord(\gamma',\lambda)=tord(H(\gamma'),H(\lambda))=q$, as follows.
If $\gamma'\subset T'_{\tau(\ell)}$ belongs to a coherent pizza slice of $\Lambda'_{T'}$, then $\lambda\subset T_\ell$ can be chosen so that
$\gamma'$ is its image under composition of the maps $(id,f)$ and $h_\ell$ in the bottom row of the diagram (\ref{diagram}).
Otherwise, if $\gamma'$ belongs to a transverse pizza slice $T'_{\ell'}$ of $\Lambda'_{T'}$,
then either $tord(\gamma',\tilde\lambda'_{\ell'})>q$, where $\tilde\lambda'_{\ell'}$ is the supporting arc of $T'_{\ell'}$,
in which case $tord(\tilde\lambda',T)=q$, or $q=tord(\gamma',\tilde\lambda'_{\ell'})$.
The arc $\tilde\lambda'_{\ell'}$ is either a boundary arc of a coherent pizza slice of $\Lambda'_{T'}$ or belongs to a maximum zone $M'_j$ of $\Lambda'_{T'}$.
If $tord(\gamma',\tilde\lambda'_{\ell'})>q$, then the arc $\lambda$ for $\gamma'$ is the same as for $\tilde\lambda_{\ell'}$
(either as for an arc in a coherent pizza slice of $\Lambda'_{T'}$ or in a maximum zone $M_i$ of $\Lambda_T$ for $j=\sigma(i)$).
If $q=tord(\gamma',\tilde\lambda'_{\ell'})$, then $q'=tord(\tilde\lambda'_{\ell'},T)\ge q$.
In any case, there is an arc $\lambda\subset T$, which is
either a boundary arc of a coherent pizza slice of $\Lambda_T$ or belongs to a maximum zone $M_i$ of $\Lambda_T$, where $j=\sigma(i)$,
such that $tord(\lambda,\tilde\lambda')=tord(H(\lambda),H(\tilde\lambda'))=q'$. Thus $tord(\gamma',\lambda)=\min(q,q')=q$.
Since $H|_{T'}$ is a bi-Lipschitz homeomorphism, we have $tord(H(\gamma'),H(\tilde\lambda'_{\ell'}))=tord(\gamma',\tilde\lambda'_{\ell'})$,
thus $tord(H(\gamma'),H(\lambda))=\min(tord(H(\gamma'),H(\tilde\lambda')),tord(H(\lambda),H(\lambda')))=\min(q,q')=q$.

If $\gamma$ is any arc in $T$, then either $tord(\gamma,\lambda)\ge q$ and $tord(\gamma',\gamma)=tord(H(\gamma'),H(\gamma))=q$ or $tord(\gamma,\lambda)<q$ and
$tord(\gamma',\gamma)=tord(H(\gamma'),H(\gamma))<q$ by the non-archimedean property of the tangency order.
Thus $H$ preserves tangency orders of any two arcs in $T\cup T'$.
Proposition \ref{important-proposition} implies that $H$ is an outer bi-Lipschitz homeomorphism.

This completes the proof of Theorem \ref{complete}
\end{proof}

\section{Realization Theorem for Totally Transverse Pairs: Blocks}\label{sec:blocks}
In this section we formulate the necessary and sufficient conditions for the existence
 of a totally transverse normal pair $(T,T')$ of H\"older triangles with the given
$\sigma\tau$-pizza invariant.
Throughout this section we use the following notations:\newline
$\bullet\;(T,T')$, where $T=T(\gamma_1,\gamma_2)\; T'=T(\gamma'_1,\gamma'_2)$, is a totally transverse normal pair of
$\beta$-H\"older triangles, oriented from $\gamma_1$ to $\gamma_2$ and from $\gamma'_1$ to $\gamma'_2$, respectively.\newline
$\bullet\;\Lambda$ and $\Lambda'$ are minimal pizzas on $T$ and $T'$ associated with the distance functions $f(x)=dist(x,T')$ and $g(x')=dist(x',T)$,
respectively.\newline
$\bullet\;\M=\{M_i\}_{i=1}^m$ and $\M'=\{M'_i\}_{i=1}^m$ are the sets of maximum zones of $\Lambda$ and $\Lambda'$,
ordered according to the orientations of $T$ and $T'$.
We assume that $m>0$, thus at least one maximum zone exists.\newline
$\bullet\;$ The permutation $\sigma$ of the set $[m]=\{1,\ldots,m\}$ is the characteristic permutation between the sets $\M$ and $\M'$ of maximum zones
of $\Lambda$ and $\Lambda'$ (see Definition \ref{characteristic}).\newline
Since there are no coherent pizza slices in $\Lambda$ and $\Lambda'$, the correspondence $\tau$ is trivial.

The permutation $\sigma$ must be consistent with the metric data defined by the pizzas $\Lambda$ and $\Lambda'$ on $T$ and $T'$.
We are going to show that consistency conditions can be defined for only one of the two pizzas, say $\Lambda$,
so that the second pizza $\Lambda'$ exists and is unique up to combinatorial equivalence.

Let $\bar q_i=\bar q'_{\sigma(i)}=tord(M_i, M'_{\sigma(i)})$ for $i=1,\ldots,m$, and let $\bar\beta_i=tord(M_i,M_{i+1})$ for $i=1,\ldots,m-1$.
If $M_1\ne\{\gamma_1\}$, let $\bar\beta_0=tord(\gamma_1,M_1)$, otherwise let $\bar\beta_0=\beta$.
If $M_m\ne\{\gamma_2\}$, let $\bar\beta_m=tord(M_m,\gamma_2)$, otherwise $\bar\beta_m=\beta$.\newline
The following statement follows from Definition \ref{maxmin}.

\begin{lemma}\label{beta}
For $i=1,\ldots, m$ the following inequality hold:
\begin{equation}\label{q:blocks}
\bar q_i>\max(\bar\beta_{i-1},\bar\beta_i).
\end{equation}
\end{lemma}

\begin{proof}
If $M_i=D_\ell$ is an interior maximum zone of $\Lambda$,
then $M_i$ is a $q_\ell$-order zone for $f$, thus $\bar q_i=q_\ell$, and
either $q_\ell>q_{\ell+1}\ge tord(D_\ell,D_{\ell+1})\ge\bar\beta_i$
or $q_\ell=q_{\ell+1}>tord(D_\ell,D_{\ell+1})\ge\bar\beta_i$.
In both cases $\bar q_i>\bar\beta_i$.
Similarly, $\bar q_i>\bar\beta_{i-1}$.
If $M_1=\{\gamma_1\}$, then $\bar q_1>\beta$. If $M_m=\{\gamma_2\}$, then $\bar q_m>\beta$.
\end{proof}

We are going to define an extension of the permutation $\sigma$ between the sets $\M$ and $M'$
to a permutation $\pi$ between the sets $\M\cup\{\gamma_1,\gamma_2\}$ and $\M'\cup\{\gamma'_1,\gamma'_2\}$.
Let $\check\lambda_0,\ldots,\check\lambda_{n-1}$ be a set of $n$ arcs in $T$,
ordered according to the orientation of $T$, such that $\check\lambda_0=\gamma_1,\;\check\lambda_{n-1}=\gamma_2$,
and each maximum zone $M_i$ of $\Lambda$ contains exactly one of these arcs.
Here $n=m$ if $M_1=\{\gamma_1\}$ and $M_m=\{\gamma_2\}$, $n=m+2$ if $M_1\ne\{\gamma_1\}$ and $M_m\ne\{\gamma_2\}$, $n=m+1$ otherwise.
Note that, if $\check\lambda_j\in M_i$, then $i=j$ when $M_1\ne\{\gamma_1\}$, otherwise $i=j+1$.
Let $\check q_j=ord_{\check\lambda_j} f$, in particular, $\check q_j=\bar q_i$ when $\check\lambda_j\in M_i$.

\begin{definition}\label{def:pi}\normalfont
Let $n$ be the number of elements in the set $\M\cup\{\gamma_1,\gamma_2\}$.
An \emph{extended permutation} $\pi$ for $(T,T')$ is a permutation of the set $[n]=\{0,\ldots,n-1\}$ such that $\pi(0)=0,\;\pi(n-1)=n-1$,
and $\pi$ is compatible with $\sigma$ on the maximum zones: $\pi(j)=\sigma(j)$ when $M_1\ne\{\gamma_1\}$, otherwise $\pi(j)+1=\sigma(j+1)$.
\end{definition}

Let $\check\lambda'_0=\gamma'_1$ and $\check\lambda'_{n-1}=\gamma'_2$.
If $\check\lambda_j\in M_i$, let $\check\lambda'_{\pi(j)}$ be any arc in $M'_{\sigma(i)}$, thus
$\check q_j=tord(\check\lambda_j,\check\lambda'_{\pi(j)})=\bar q_i$.
This defines a set of $n$ arcs $\check\lambda'_0,\ldots,\check\lambda'_{n-1}$ in $T'$, ordered according to the orientation of $T'$.

\begin{definition}\label{def:block}\normalfont
Let $\chi$ be a permutation of a set $[n]=\{0,\ldots,n-1\}$. A segment of $[n]$ (a non-empty set of consecutive indices $\{j,\ldots,k\}\subseteq [n]$)
is called a \emph{block} of $\chi$ (see \cite{Atkinson}) if the set $\{\chi(j),\ldots,\chi(k)\}$ is also a set of consecutive indices
(not necessarily in increasing order).
A block is \emph{trivial} if it contains a single element, or if it is equal $[n]$.
\end{definition}

\begin{lemma}\label{lem:block}\normalfont
The following properties of blocks of a permutation $\chi$ of $[n]$ hold:
\begin{itemize}
\item[{\bf (i)}] If a segment $B$ of $[n]$ is a block of $\chi$ then the set $\chi(B)$ is a block of $\chi^{-1}$, thus $B\mapsto\chi(B)$ defines
one-to-one correspondence between the blocks of $\chi$ and $\chi^{-1}$.
\item[{\bf (ii)}] If two blocks $B$ and $B'$ of $\chi$ have non-empty intersection, then $B\cap B'$ and $B\cup B'$ are also blocks of $\chi$.
\item[\bf (iii)] Each non-empty subset $J$ of $[n]$ is contained in a unique \emph{minimal block} $B_{\chi}(J)$.
\item[{\bf (iv)}] Let $S(J)$ be the minimal segment of $[n]$ containing a non-empty subset $J$ of $[n]$. Then $|S(J)|=|J|$ if, and only if, $J$ is a segment.
\item[{\bf (v)}] For a non-empty set $J\subseteq [n]$, let us define
$\mathfrak{S}(J)=\chi^{-1}(S(\chi(S(J)),\;\;\mathfrak{S}^1(J)=\mathfrak{S}(J)$, and  $\mathfrak{S}^k(J)=\mathfrak{S}(\mathfrak{S}^{k-1}(J))$ for $k>1$.
    Then $\mathfrak{S}^k(J) \subseteq \mathfrak{S}^{k+1}(J),\;\; \mathfrak{S}^k(J) \subseteq B_{\chi}(J)$, and $\mathfrak{S}^{k+1}(J)=\mathfrak{S}^k(J)$ if, and only if, $\mathfrak{S}^k(J)=B_{\chi}(J)$.
\end{itemize}
\end{lemma}

\begin{proof}
(i) From the definition of a block, $\chi(B)$ is a segment of $[n]$ and $\chi^{-1}(\chi(B))=B$.\newline
(ii) If $\chi(B)$ and $\chi(B')$ are segments of $[n]$, then $\chi(B)\cap\chi(B')=\chi(B\cap B')\ne\emptyset$ is also a segment.
Similarly, if $\chi(B\cup B')$ is not a segment, then
$\chi(B)\cap\chi(B')=\chi(B\cap B')$ is empty, thus $B\cap B'$ is empty, a contradiction.\newline
(iii) Since $J\subset[n]$ and $[n]$ is a block of $[n]$,
 a minimal block of $[n]$ containing $J$ exists and is unique by (ii).\newline
(iv) Since $J\subseteq S(J)$, we have $|S(J)|=|J|$ only when $S(J)=J$.\newline
(v) Since $J\subseteq S(J)\subseteq B_{\chi}(J)$ and $\chi(S(J))\subseteq S(\chi(S(J))\subseteq \chi(B_{\chi}(J))$
for any non-empty subset $J$ of $[n]$, we have $J\subseteq\mathfrak{S}(J)\subseteq B_{\chi}(J)$,
thus $|\mathfrak{S}(J)|\ge|J|$, and equality holds only if $J=S(J)$ is a segment of $[n]$
and $\chi(S(J))=\chi(J)$ is a segment of $[n]$, in which case $J=B_{\chi}(J)$ is a block of $\chi$.
\end{proof}

\begin{notation}\label{ij-block}\normalfont
For $i\ne j$, $0\le i,j\le n-1$, let $B_{ij}=B_{\chi}(\{i,j\})$ be the minimal block of $\chi$ containing the set $\{i,j\}$ (see Lemma \ref{lem:block}),
and let $B'_{ij}=B_{\chi^{-1}}(\{i,j\})$ be the minimal block of $\chi^{-1}$ containing the set $\{i,j\}$.
\end{notation}

\begin{theorem}\label{beta-block}
Let $(T,T')$ be a totally transverse normal pair of H\"older triangles $T=T(\gamma_1,\gamma_2)$ and $T'=T(\gamma'_1,\gamma'_2)$,
with the minimal pizzas $\Lambda$ and $\Lambda'$ on $T$ and $T'$ associated with the distance functions $f(x)=dist(x,T')$ and $g(x')=dist(x',T)$, respectively.
Let $\{\check\lambda_0,\ldots,\check\lambda_{n-1}\}$ be a set of arcs in $T$
such that  $\check\lambda_0=\gamma_1,\;\check\lambda_{n-1}=\gamma_2$, and
each maximum zone of $\Lambda$ contains exactly one of the arcs $\check\lambda_i$.
For $i=1,\ldots,m$, if $\check\lambda_j\in M_i$, let $\check\lambda'_{\pi(j)}$ be any arc in $M'_{\sigma(i)}$,
where $\pi:[n]\to[n]$ is an extended permutation of $(T,T')$. Then
\begin{equation}\label{beta:blocks}
tord(\check\lambda_i,\check\lambda_j)\le tord(\check\lambda_k,\check\lambda_l)\; \text{for}\; \{k,l\}\subset B_{ij},
\end{equation}
\begin{equation}\label{beta:blocks-prime}
tord(\check\lambda'_i,\check\lambda'_j)\le tord(\check\lambda'_k,\check\lambda'_l)\; \text{for}\; \{k,l\}\subset B'_{ij}.
\end{equation}
\end{theorem}

\begin{proof}
For each subset $I$ of $[n]$, we define
$T_I=T(\check\lambda_j,\check\lambda_k)\subset T$, where $j$ and $k$ are the minimal and maximal values of $i\in I$.
Let $T'_I=T(\check\lambda'_j,\check\lambda'_k)\subset T'$, where $j$ and $k$ are the minimal and maximal values of $i\in I$. For $I=\{i,j\}$,
Lemma \ref{beta} implies that exponents of the H\"older triangles $T'_{\pi(S(I))}$ and $T_{\mathfrak{S}^1(I)}$ are equal to the exponent $\beta$ of $T_I$.
It follows that $T_{B_\pi(I)}$ is also a $\beta$-H\"older triangle.
\end{proof}

\begin{corollary}\label{cor:beta-block}
If $B_{ij}=B_{kl}$ then $tord(\check\lambda_{i},\check\lambda_j)=tord(\check\lambda_{k},\check\lambda_l)$.
\end{corollary}

\begin{remark}\label{rem:pi}\normalfont
Since $\pi(0)=0$ and $\pi(n-1)=n-1$, the segments $\{0,\ldots,n-2\}$ and $\{1,\ldots,n-1\}$ are blocks of $\pi$.
This implies that, for $i\ne j$ and $\{i,j\}\ne\{0,n-1\}$, each block $B_{ij}$ is non-trivial.
We have $B_{ij}=\{i,j\}$ if, and only if, $j=i\pm 1$ and $\pi(j)=\pi(i)\pm 1$.
\end{remark}

\begin{example}\label{example:pi}\normalfont
The permutation $\pi=(0,3,1,4,2,5)$ for the normal pair of H\"older triangles in Figure \ref{fig:pi} has three non-trivial blocks:
$B_{01}=\{0,\ldots,4\}$, $B_{45}=\{1,\ldots,5\}$, $B_{12}=B_{23}=B_{34}=\{1,\ldots,4\}$.
Accordingly, $tord(\check\lambda_0,\check\lambda_1)=tord(\check\lambda'_0,\check\lambda'_1)\le
tord(\check\lambda_1,\check\lambda_2)=tord(\check\lambda'_1,\check\lambda'_2)=tord(\check\lambda_2,\check\lambda_3)=tord(\check\lambda'_2,\check\lambda'_3)
=tord(\check\lambda_3,\check\lambda_4)=tord(\check\lambda'_3,\check\lambda'_4)\ge tord(\check\lambda_4,\check\lambda_5)=tord(\check\lambda'_4,\check\lambda'_5)$
in Theorem \ref{beta-block}.
\end{example}

\begin{figure}
\centering
\includegraphics[width=6in]{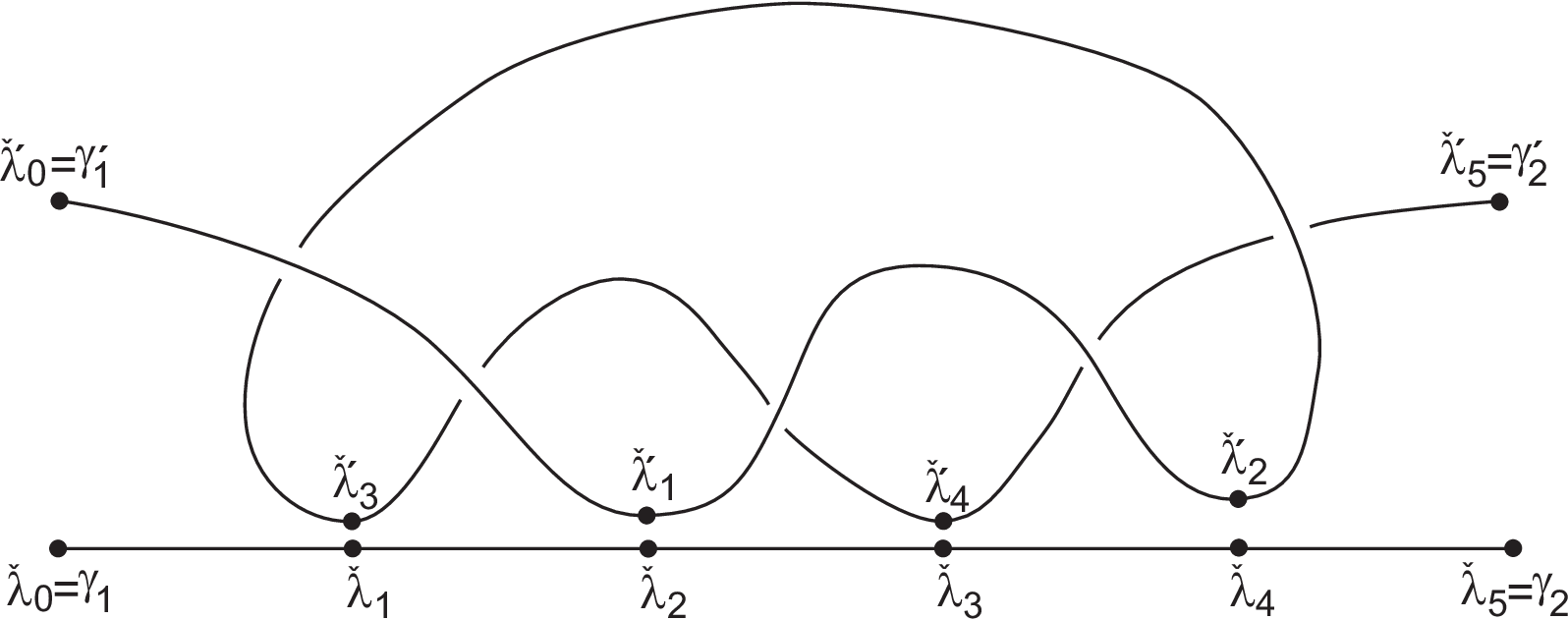}
\caption{Normally embedded triangles in Example \ref{example:pi}.}\label{fig:pi}
\end{figure}
\begin{figure}
\centering
\includegraphics[width=2in]{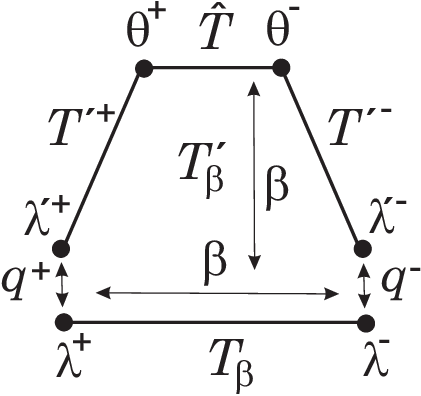}
\caption{Model pair of triangles in Definition \ref{model}.}\label{fig:model-theta}
\end{figure}
\begin{figure}
\centering
\includegraphics[width=6.3in]{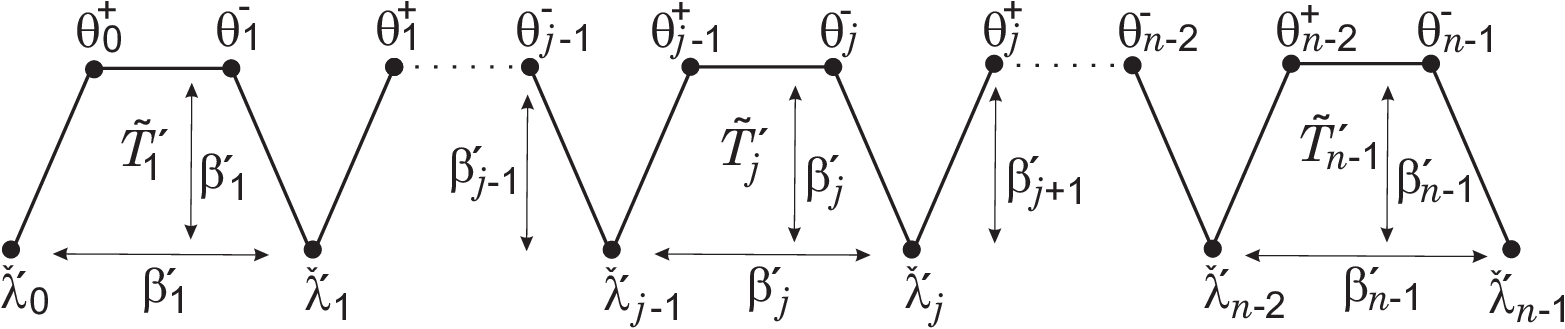}
\caption{Arcs $\check\lambda'_j$ and $\theta^\pm_j$ in the proof of Theorem \ref{transverse-blocks}.}\label{fig:order}
\end{figure}

\begin{definition}\label{supporting-family}\normalfont
Let $T=T(\gamma_1,\gamma_2)$ be a normally embedded H\"older triangle oriented from $\gamma_1$ to $\gamma_2$, and let $f(x)$ be a totally transverse function on $T$ (see Definition \ref{def:pizzaslicezone-transverse}). Let $\Lambda$ be a minimal pizza on $T$ associated with $f$.
A family $\{\check\lambda_j\}_{j=0}^{n-1}$ of $n$ arcs in $T$ ordered according to the orientation of $T$, such that $\check\lambda_0=\gamma_1$ and
$\check\lambda_{n-1}=\gamma_2$,
is called a \emph{supporting family} associated with $f$ if one of the following conditions is satisfied.

(A) Both arcs $\check\lambda_0=\gamma_1$ and $\check\lambda_{n-1}=\gamma_2$ are maximum zones, and each arc $\check\lambda_j$ belongs to a maximum zone $M_{j+1}$ of
$\Lambda$.

(B) The arc $\check\lambda_0=\gamma_1$ is a maximum zone, the arc $\check\lambda_{n-1}=\gamma_2$ is a minimum zone, and an arc $\check\lambda_j$ belongs to a maximum zone
$M_{j+1}$ of $\Lambda$ when $j<n-1$.

(C) The arc $\check\lambda_0=\gamma_1$ is a minimum zone, the arc $\check\lambda_{n-1}=\gamma_2$ is a maximum zone, and an arc $\check\lambda_j$ belongs to a maximum zone
$M_j$ of $\Lambda$ when $j>0$.

(D) Both arcs $\check\lambda_0=\gamma_1$ and $\check\lambda_{n-1}=\gamma_2$ are minimum zones, and an arc $\check\lambda_j$ belongs to a maximum zone $M_j$ of $\Lambda$ when
$0<j<n-1$.
\end{definition}

\begin{definition}\label{admissible}\normalfont
Let $\chi$ be a permutation of the set $[n]=\{0,\ldots,n-1\}$, such that $\chi(0)=0$ and $\chi(n-1)=n-1$.
For indices $i\ne j$ in $[n]$, let $B_{ij}=B_\chi(\{i,j\})$ be the minimal block of $\chi$ containing $\{i,j\}$ (see Notation \ref{ij-block}).
The permutation $\chi$ is called \emph{admissible} with respect to a totally transverse Lipschitz function $f$ on a H\"older triangle $T$ if
a supporting family $\{\check\lambda_j\}_{j=0}^{n-1}$ of arcs in $T$ associated with $f$ satisfies inequalities (\ref{beta:blocks}).
\end{definition}

\begin{remark}\label{admissible-remark}\normalfont
Let $(T,T')$ be a totally transverse normal pair of H\"older triangles, where $T=T(\gamma_1,\gamma_2)$ and $T'=T(\gamma'_1,\gamma'_2)$,
with the minimal pizzas $\Lambda$ and $\Lambda'$ associated with the distance functions $f(x)=dist(x,T')$ and $g(x')=dist(x',T)$, respectively.
Then the extended permutation $\pi$ of the pair $(T,T')$ (see Definition \ref{def:pi}) is admissible.
\end{remark}

\begin{definition}\label{model}\normalfont
Let $q^+,q^-,\beta$ be three exponents, such that $\beta \le \min(q^+,q^-)$.
Let $T_\beta\subset \R^2_{u,v}$, be the standard $\beta$-H\"older triangle (see Definition \ref{standard holder}) bounded by the arcs
$\lambda^+=\{u\ge 0,\,v\equiv 0\}$ and $\lambda^-=\{u\ge 0,\,v=u^\beta\}$.
Let $\lambda'^+=\{(u,v)\in \lambda^+,\, z=u^{q^+}\}$ and $\lambda'^-=\{(u,v)\in \lambda^-,\, z=u^{q^-}\}$
be two arcs in $\R^3_{u,v,z}$, where $\R^2_{u,v}=\{(u,v,z): z=0\}$ is a subspace of $\R^3_{u,v,z}$.
Let $\theta^+=\{(u,v,z)\in\lambda'^+,\, w=u^\beta\}$ and $\theta^-=\{(u,v,z)\in\lambda'^-,\, w=u^\beta\}$ be two arcs in $\R^4_{u,v,z,w}$,
where $\R^3_{u,v,z}=\{(u,v,z,w): w=0\}$ is a subspace of $\R^4_{u,v,z,w}$.

Let $T'^+=T(\lambda'^+,\theta^+)$ be the $\beta$-H\"older triangle in $\R^4_{u,v,z,w}$ obtained as the union
 of straight line segments with endpoints $\lambda'^+\cap\{u=u_0\}$ and $\theta^+\cap\{u=u_0\}$, parallel to the $w$-axis, over small non-negative $u_0$.
 Similarly, let $T'^-=T(\lambda'^-,\theta^-)$ be the $\beta$-H\"older triangle in $\R^4_{u,v,z,w}$ obtained as the union
 of straight line segments with endpoints $\lambda'^-\cap \{u=u_0\}$ and $\theta^-\cap \{u=u_0\}$, parallel to the $w$-axis.
Let $\hat T'=T(\theta^+,\theta^-)$ be the $\beta$-H\"older triangle in $\R^4_{u,v,z,w}$ obtained as the union
of straight line segments with endpoints $\theta^+\cap\{u=u_0\}$ and $\theta^-\cap\{u=u_0\}$.
Note that $T'^+$ and $\hat T'$ have a common boundary arc $\theta^+$, while $\hat T'$ and $T'^-$  have a common boundary arc $\theta^-$.
Thus $T'_\beta=T'^+\cup \hat T'\cup T'^-$ (see Fig.~\ref{fig:model-theta}) is a $\beta$-H\"older triangle.
The pair $(T_\beta,T'_\beta)$ of $\beta$-H\"older triangles is called a $(q^+,q^-,\beta)$-\emph{model}.
\end{definition}

\begin{remark}\label{model-remark}\normalfont
The $(q^+,q^-,\beta)$-model $(T_\beta,T'_\beta)$ is a totally transverse normal pair of $\beta$-H\"older triangles, with the given tangency orders $q^+$ and $q^-$ of their
boundary arcs, where $\beta\le\min(q^+,q^-)$.
Let $f(x)=dist(x,T')$ be the distance function on $T$.
If $\beta=q^+=q^-$, then $T_\beta$ is a single pizza slice of a minimal pizza associated with $f$,
and both boundary arcs of $T_\beta$ are minimum zones for $f$.
If $\max(q^+,q^-)>\beta=\min(q^+,q^-)$, then $T_\beta$ is a single pizza slice of a minimal pizza associated with $f$,
one of its boundary arcs is a maximum zone for $f$, another one is a minimum zone.
If $\beta<\min(q^+,q^-)$, then $T_\beta$ has two pizza slices of a minimal pizza associated with $f$,
both boundary arcs of $T_\beta$ are maximum zones for $f$, and the set $G(T_\beta)$ of its generic arcs is a minimum zone.
\end{remark}

The following theorem implies that the necessary conditions (\ref{q:blocks}) and (\ref{beta:blocks}) on the exponents $\bar q_i$ and $tord(\check\lambda_i,\check\lambda_j)$ are sufficient for the existence and uniqueness, up to outer Lipschitz equivalence,
of a totally transverse normal pair $(T,T')$ of H\"older triangles with the given
characteristic permutation $\sigma$ of a minimal pizza associated with the totally transverse distance function $f$ on $T$.
Given a normally embedded H\"older triangle $T$ and a supporting family $\{\check\lambda_i\}_{i=0}^{n-1}$ of arcs in $T$
for the minimal pizza $\Lambda$ on $T$ associated with $f$, a H\"older triangle $T'$ is constructed as follows.
Adding an extra variable $z$, we define a family of arcs $\{\check\lambda'_j\}_{j=0}^{n-1}$ that will become
a supporting family for a minimal pizza $\Lambda'$ on $T'$ associated with the distance function $g(x')=dist(x',T)$.
Each arc $\check\lambda'_j$ is the graph of a function $z=u^{\bar q_i}$ over the arc $\check\lambda_i$, where $j=\pi(i)$.
Next, any two consecutive arcs $\check\lambda'_{j-1}$ and $\check\lambda'_j$ are ``connected'' by a H\"older triangle $T'_j$
based on the model H\"older triangle $T'_{\beta'_j}$
(see Definition \ref{model} and Figure \ref{fig:model-theta} above) where $\beta'_j=tord(\check\lambda'_{j-1},\check\lambda'_j)$.
The H\"older triangle $T'$ is defined as the union of H\"older triangles $T'_j$.
To show that $T'$ is normally embedded, we use conditions (\ref{q:blocks}) and (\ref{beta:blocks}) to prove that triangles $T'_j$ are normally embedded and pairwise
transverse, thus $T'$ is combinatorially normally embedded  (see Definition \ref{combinatorialLNE} and Proposition \ref{combinatorialLNE-prop}).

\begin{theorem}\label{transverse-blocks}
Let $T=T(\gamma_1,\gamma_2)$ be a normally embedded H\"older triangle oriented from $\gamma_1$ to $\gamma_2$, and let
$\{\check\lambda_j\}_{j=0}^{n-1}$ be a supporting family of arcs in $T$ associated with a non-negative totally transverse Lipschitz function $f$ on $T$
(see Definition \ref{supporting-family}).
Then, for any admissible with respect to $f$ permutation $\pi$ of $[n]=\{0,\ldots,n-1\}$, there exists a unique, up to outer Lipschitz equivalence,
totally transverse normal pair $(T,T')$ of H\"older triangles, such that
the function $dist(x,T')$ on $T$ is contact equivalent to $f$ and $\pi$ is the extended permutation of the pair $(T,T')$ (see Definition \ref{def:pi}).
\end{theorem}

\begin{proof}
We may assume that $T=T_\beta$ is a standard $\beta$-H\"older triangle (\ref{Formula:Standard Holder triangle}) in $\R^2_{uv}$,
the arcs $\check\lambda_i$ are germs at the origin of the graphs $\{u\ge 0,\,v=\check\lambda_i(u)\}$ of Lipschitz functions $\check\lambda_i(u)\ge 0$, and each
$T_i=T(\check\lambda_{i-1},\check\lambda_i)$ is a $\beta_i$-H\"older triangle.
From the non-archimedean property of the tangency order, we have
$tord(\check\lambda_i,\check\lambda_l)=\min\left(tord(\check\lambda_i,\check\lambda_k),tord(\check\lambda_k,\check\lambda_l)\right)$ for $i<k<l$.
It will be convenient in this proof to parameterize all arcs by $u$ instead of the distance to the origin.

To define a totally transverse pair $(T,T')$, we are going to construct a normally embedded H\"older triangle $T'$ in several steps.
In Step 1, we define a set of $n$ arcs $\check\lambda'_j$ in $\R^3_{u,v,z}$ that will be a supporting family for $T'$,
assuming that $\R^2_{u,v}=\{(u,v,z):z=0\}$ is a subspace of $\R^3_{u,v,z}$.
In Step 2, the triangle $T'$ in $\R^{n+2}_{u,v,z,\mathbf w}$, where $\mathbf w=(w_1,\ldots,w_{n-1})$, is constructed as the union of $n-1$
subtriangles $\tilde T'_j=T(\check\lambda'_{j-1},\check\lambda'_j)\subset\R^4_{u,v,z,w_j}$ based on the models from Definition \ref{model}.
Each space $\R^4_{u,v,z,w_j}$ is assumed to be a subspace of the space $\R^{n+2}_{u,v,z,\mathbf w}$ defined by equations $w_k=0$ for all $k\ne j$.
In Step 3, we prove that $T'$ is normally embedded.
In step 4, we prove that the distance function $f(x)=dist(x,T')$ on $T$
is totally transverse and satisfies conditions of Theorem \ref{transverse-blocks},
thus $(T,T')$ is a totally transverse normal pair.

{\bf Step 1.} For $j=0,\ldots,n-1$, let $\check\lambda'_j\subset\R^3_{u,v,z}$ be the arc
$\{(u,v)\in\check\lambda_i,\, z=u^{\check q_i}\}$, where $j=\pi(i)$ and $\check q_i=ord_{\check\lambda_i} f$.
Let $\check q'_j=\check q_i=tord(\check\lambda'_j,T)$.
The set of arcs $\{\check\lambda'_j\}_{j=0}^{n-1}$ will be a supporting family for the H\"older triangle $T'$ associated with
the distance function $g(x')=dist(x',T)$ on $T'$.
In particular, $\gamma'_1=\check\lambda'_0$ and $\gamma'_2=\check\lambda'_{n-1}$ will be the boundary arcs of $T'$.
Let $\beta'_j=tord(\check\lambda'_{j-1},\check\lambda'_j)$ for $j=1,\ldots,n-1$.
From the non-archimedean property of the tangency order we have, for $j=\pi(i_1)$ and $k=\pi(i_2)$, isometry
\begin{equation}\label{tord-lambda}
tord(\check\lambda'_j,\check\lambda'_k)=tord(\check\lambda_{i_1},\check\lambda_{i_2}).
\end{equation}
Since $T$ is normally embedded and the arcs $\check\lambda_i$ satisfy the block inequalities (\ref{beta:blocks}),
the arcs $\check\lambda'_j$ satisfy combinatorial normal embedding inequalities (see Definition \ref{combinatorialLNE})
necessary for the H\"older triangle $T'$ with a supporting family $\{\check\lambda'_j\}_{j=0}^{n-1}$ to be normally embedded:
\begin{equation}\label{check-LNE}
tord(\check\lambda'_j,\check\lambda'_l)=\min\left(tord(\check\lambda'_j,\check\lambda'_k),tord(\check\lambda'_k,\check\lambda'_l)\right)\;\text{for}\;j<k<l.
\end{equation}
This follows from (\ref{tord-lambda}) when $j=\pi(i_1),\,k=\pi(i_2)$ and $l=\pi(i_3)$, where $i_2\in[i_1,i_3]$, as
$tord(\check\lambda_{i_1},\check\lambda_{i_3})=\min\left(tord(\check\lambda_{i_1},\check\lambda_{i_2}),tord(\check\lambda_{i_2},\check\lambda_{i_3})\right)$ in that case.
If, e.g., $i_1<i_3<i_2$ then $i_2\in B_\pi(\{i_1,i_3\})$, thus $tord(\check\lambda_{i_1},\check\lambda_{i_2})=tord(\check\lambda_{i_1},\check\lambda_{i_3})\le
tord(\check\lambda_{i_2},\check\lambda_{i_3})$ due to (\ref{beta:blocks}). The other cases are similar.

Since the arcs $\check\lambda_i$ satisfy inequalities (\ref{beta:blocks}), the arcs $\check\lambda'_j$ satisfy the block inequalities (\ref{beta:blocks-prime}).

{\bf Step 2.}
Consider the product $\R^{n+2}_{u,v,z,\mathbf w}$ of $\R^3_{u,v,z}$ and $\R^{n-1}_{\mathbf w}$,where $\mathbf w=(w_1,\dots,w_{n-1})$.
Each space $\R^4_{u,v,z,w_j}$ is a subspace of $\R^{n+2}_{u,v,z,\mathbf w}$ defined by equations $w_k=0$ for all $k\ne j$.
We define $2n-2$ arcs $\theta^\pm_j$ in $\R^{n+2}_{u,v,z,\mathbf w}$ as follows:\newline
For $0<j\le n-1$, let $\theta^-_j=\{(u,v,z)\in\check\lambda'_j,\; w_j=u^{\beta'_j}\}\subset\R^4_{u,v,z,w_j}$.\newline
For $0\le j<n-1$, let $\theta^+_j=\{(u,v,z)\in\check\lambda'_j,\; w_{j+1}=u^{\beta'_{j+1}}\}\subset\R^4_{u,v,z,w_{j+1}}$.

From the inequalities (\ref{q:blocks}) and (\ref{beta:blocks}), we get $\beta'_j\le\min(\check q'_{j-1},\check q'_j)$
for each $j=1,\ldots,n-1$.

For $1\le j\le n-1$, consider a $(\check q'_{j-1},\check q'_j,\beta'_j)$-model $(T_{\beta'_j},T'_{\beta'_j})$ (see Definition \ref{model}),
where $T_{\beta'_j}=T(\lambda^+,\lambda^-)$ is the standard $\beta'_j$-H\"older triangle and $T'_{\beta'_j}=T'^+\cup \hat T'\cup T'^-$
is the union of three $\beta'_j$-H\"older triangles $T'^+=T(\lambda'^+,\theta^+),\;\hat T'=T(\theta^+,\theta^-)$ and $T'^-=T(\lambda'^-,\theta^-)$.

Let $h_j:T'_{\beta'_j}\to \R^4_{u,v,z,w_j}$ be a map preserving the variable $u$, such that
$h_j(\lambda'^+)=\lambda'_{j-1}$, $h_j(\lambda'^-)=\lambda'_j$, $h_j(\theta^+)=\theta^+_{j-1}$, $h_j(\theta^-)=\theta^-_j$ and,
for any small positive $u_0$, the map $h_j$ is linear
on each of the three straight line segments of $T'_{\beta'_j}\cap\{u=u_0\}$.
This defines $h_j$ completely.
Each space $\R^4_{u,v,z,w_j}$ is a subspace of the space $\R^{n+2}_{u,v,z,\mathbf w}$ defined by equations $w_k=0$ for all $k\ne j$.
We'll show in Step 3 that $h_j$ defines an outer bi-Lipschitz map
$T'_{\beta'_j}\to\tilde T'_j$, where $\tilde T'_j=T'^+_{j-1}\cup \hat T'_j\cup T'^-_j,\;
T'^+_{j-1}=h_j(T'^+),\;\hat T'_j=h_j(\hat T'),\;T'^-_j=h_j(T'^-)$ (see Figure \ref{fig:order}).

The arcs $\theta^\pm_j$ will belong to minimum zones of a minimal pizza on $T'$ associated with $g$.
These arcs will be located on $T'=T(\gamma'_1,\gamma'_2)$ in the following order (see Figure \ref{fig:order}):
\begin{equation}\label{order}
\check\lambda'_0,\theta^+_0,\theta^-_1,\check\lambda'_1,\theta^+_1,\ldots,\theta^-_{j-1},\check\lambda'_{j-1},\theta^+_{j-1},\theta^-_j,
\check\lambda'_j,\theta^+_j,\ldots,\theta^-_{n-2},\check\lambda'_{n-2},\theta^+_{n-2},\theta^-_{n-1}, \check\lambda'_{n-1}.
\end{equation}
Note that each variable $w_j$ has the order $\beta'_j$ on the arcs $\theta^-_{j-1}$ and $\theta^+_j$,
and is identically zero on all arcs $\theta^-_k$ for $k\ne j-1$, on all arcs $\theta^+_k$ for $k\ne j$, and on all arcs $\check\lambda'_k$.

{\bf Step 3.}
Let us show first that each H\"older triangle $\tilde T'_j$ is normally embedded.
It is the union of three normally embedded $\beta'_j$-H\"older triangles $T(\check\lambda'_{j-1},\theta^+_{j-1})$,
$T(\theta^-_j,\check\lambda'_j)$ and $T(\theta^+_{j-1},\theta^-_j)$,
each of them being a family of straight line segments.
The first two of them are families of straight line segments
of length $u^{\beta'_j}$ parallel to the axis $w_j$, the third one connects them by a family of the straight line segments in $\{w_j=u^{\beta'_j},\; w_k\equiv
0\;\text{for}\;k\ne j\}$.
Since $tord(\check\lambda'_{j-1},\check\lambda'_j)=\beta'_j$, the tangency order between
$T(\check\lambda'_{j-1},\theta^+_{j-1})$ and $T(\theta^-_{j},\check\lambda'_j)$ is also $\beta'_j$.
This implies that $\tilde T'_j$ is normally embedded.

To show that $T'$ is combinatorially normally embedded (see Definition \ref{combinatorialLNE})
we need to prove that any two H\"older triangles $\tilde T'_j$ and $\tilde T'_k$ are transverse.
Let $\eta\subset\tilde T'_j$ and $\eta'\subset\tilde T'_k$ be any two arcs,
and let ${\rm proj}_T \eta$ and ${\rm proj}_T \eta'$ be their projections to $T$.
Note that the variable $w_j$, which is non-zero on $\tilde T'_j$, vanish on $\tilde T'_k$,
and the variable $w_k$, which is non-zero on $\tilde T'_k$, vanish on $\tilde T'_j$.
Thus $tord(\eta,\eta')\le tord(\eta,{\rm proj}_T \eta')$ and $tord(\eta,\eta')\le tord({\rm proj}_T \eta,\eta')$

If $k=j+1$, then $\tilde T'_j$ and $\tilde T'_{j+1}$ are consecutive H\"older triangles in $T'$, and
$\tilde T'_j\cap\tilde T'_{j+1}=\check\lambda'_j$.

When $tord(\eta,\check\lambda'_{j-1})=tord(\eta,\check\lambda'_j)=\beta'_j$, then
$tord(\eta,T)=ord_\eta w_j=\beta'_j$ and
$tord(\eta,\eta')\le tord(\eta,{\rm proj}_T \eta')\le\beta'_j=tord(\eta,\check\lambda'_j)$.
Similarly, when $tord(\eta',\check\lambda'_j)=tord(\eta',\check\lambda'_{j+1})=\beta'_{j+1}$, then
$tord(\eta,\eta')\le\beta'_{j+1}=tord(\eta',\check\lambda'_{j+1})$.

If $tord(\eta,\check\lambda'_j)>\beta'_j$, then
 $\eta\subset T(\check\lambda'_j,\theta^-_j)$ and $tord(\eta,\check\lambda'_j)=ord_\eta w_j\le tord(\eta,\eta')$,
as $w_j\equiv 0$ on $\tilde T'_{j+1}$.
Similarly, if $tord(\eta',\check\lambda'_j)>\beta'_{j+1}$, then
 $\eta'\subset T(\check\lambda'_j,\theta^+_j)$ and $tord(\eta',\check\lambda'_j)=ord_{\eta'} w_{j+1}\le tord(\eta,\eta')$.

If $tord(\eta,\check\lambda'_{j-1})>\beta'_j$ and
$tord(\eta',\check\lambda'_{j+1})>\beta'_{j+1}$, then
$tord(\eta,\lambda'_j)=\beta'_j\le tord(\eta,\eta')$ by the non-archimedean property of the tangency order, as
$tord(\lambda'_{j-1},\lambda'_{j+1})\ge\min(\beta'_j,\beta'_{j+1})$ by (\ref{check-LNE}).
Thus the H\"older triangles $\tilde T'_j$ and $\tilde T'_{j+1}$ are transverse, with $\tilde\gamma=\tilde\gamma'=\check\lambda'_j$ in Definition \ref{def:transverse}.

Let $j<k-1, j-1=\sigma(i_{-}), j=\sigma(i_{+}), k-1=\sigma(l_{-}), k=\sigma(l_{+})$.

If $tord(\check\lambda'_j,\check\lambda'_{k-1})<\max(\beta'_j,\beta'_k)$, then $\tilde T'_j$ and $\tilde T'_k$ are transverse by the non-archimedean property
of the tangency order, thus we may assume that
$tord(\check\lambda'_j,\check\lambda'_{k-1})\ge\max(\beta'_j,\beta'_k)$.

The same arguments as in the case $k=j+1$ above show that
$tord(\eta,\eta')\le\beta'_j=tord(\eta,\check\lambda'_j)=tord(\eta,\check\lambda'_{k-1})$
when $tord(\eta,\check\lambda'_{j-1})=tord(\eta,\check\lambda'_j)=\beta'_j$
and $tord(\eta,\eta')\le\beta'_k=tord(\eta',\check\lambda'_{k-1})=tord(\eta',\check\lambda'_j)$
when $tord(\eta',\check\lambda'_{k-1})=tord(\eta',\check\lambda'_k)=\beta'_k$.

If $\kappa=tord(\eta,\check\lambda'_j)>\beta'_j$, then
 $\eta\subset T(\check\lambda'_j,\theta^-_j)$ and $\kappa=ord_\eta w_j$.
 If $\kappa\le tord(\check\lambda'_j,\check\lambda'_{k-1})$, then $tord(\eta,\check\lambda_{k-1})=\kappa$,
  and the same arguments as in the case $k=j+1$ show that $tord(\eta,\eta')\le tord(\eta,\lambda'_{k-1})$,
as $w_j\equiv 0$ on $\tilde T'_k$.

Otherwise, if $\kappa>\alpha$, then $tord(\eta,\check\lambda_{k-1})=\alpha=tord(\eta,\eta')$
when $tord(\eta',\check\lambda'_{k-1})\ge\alpha$ and $tord(\eta,\check\lambda_{k-1})<\alpha=tord(\eta,\eta')$
when $tord(\eta',\check\lambda'_{k-1})<\alpha$.
Similarly, if $tord(\eta',\check\lambda'_{k-1})>\beta'_k$, then
 $\eta'\subset T(\check\lambda'_{k-1},\theta^+_k)$ and $tord(\eta',\check\lambda'_j)\le tord(\eta,\eta')$.

If $tord(\eta,\check\lambda'_{j-1})>\beta'_j$, then $tord(\eta,\eta')\le tord(\eta,\check\lambda'_{k-1})$
by the non-archimedean property of the tangency order, by the same argument as in the case $k=j+1$.
Similarly, if $tord(\eta',\check\lambda'_k)>\beta'_k$, then $tord(\eta,\eta')\le tord(\eta',\check\lambda'_j)$
by the non-archimedean property of the tangency order.
Thus the H\"older triangles  $\tilde T'_j$ and $\tilde T'_k$ are transverse,
 with $\tilde\gamma=\check\lambda'_j$ and $\tilde\gamma'=\check\lambda'_{k-1}$ in Definition \ref{def:transverse}.

By Proposition \ref{combinatorialLNE-prop}, the H\"older triangle $T'$ is normally embedded.

{\bf Step 4.} We show first that the distance function $g(x')=dist(x',T)$ on $T'$ is totally transverse.
This would imply that $f(x)=dist(x,T')$ is a totally transverse function on $T$,
thus its contact equivalence class is completely determined by the exponents $\beta_j$ of $T_j$ and the orders $\check q_i$ of $f|_{\check\lambda_i}$.

Let $\eta$ be any arc in $T'_j$. If $tord(\eta,\check\lambda'_{j-1})>\beta'_j$ then $\eta\subset T(\check\lambda'_{j-1},\theta^+_{j-1})$ and $tord(\eta,T)=ord_\eta w_j$. If
$tord(\eta,\check\lambda'_j)>\beta'_j$, then
 $\eta\subset T(\check\lambda'_j,\theta^-_j)$ and $tord(\eta,T)=ord_\eta w_j$.
 Otherwise, if $\eta\in G(T'_j)$, then $tord(\eta,T)=\beta'_j$.
 This implies that the distance function $g|_{T'_j}$ is totally transverse.
 Since $T'=\bigcup_j T'_j$ is normally embedded, the distance function $g(x)=dist(x,T)$ on $T'$ is totally transverse.
 Proposition \ref{pizzaslice-oriented} implies that the distance function $f(x)=dist(x,T')$ on $T$ is also totally transverse:
  a coherent slice $T_\ell$ for a minimal pizza on $T$ associated with $f$ would correspond to a coherent slice for a minimal pizza on $T'$
  associated with $g$, a contradiction. Lemma \ref{determination} implies that the function $f$ is
  determined, up to contact Lipschitz equivalence, by the exponents $\beta_j$ and $\check q_i$.
  The permutation $\pi$ is the extended permutation of the pair $(T,T')$
   by the construction of arcs $\lambda'_j$. Thus the distance function $f$ satisfies conditions of Theorem
  \ref{transverse-blocks}.
\end{proof}

\section{Realization Theorem for General Pairs}\label{sec:blocks-general}
In this section we formulate the necessary and sufficient conditions for the existence of a general normal pair $(T,T')$ of H\"older triangles with the given
$\sigma\tau$-pizza invariant.
For this purpose we define the notions of \emph{pre-pizza} and \emph{twin pre-pizza} associated with a non-negative
Lipschitz function $f$ on a normally embedded H\"older triangle $T$.\newline
Throughout this section we use the following notations:\newline
$\bullet\;(T,T')$, where $T=T(\gamma_1,\gamma_2)$ and $T'=T(\gamma'_1,\gamma'_2)$, is a normal pair of
$\beta$-H\"older triangles, oriented from $\gamma_1$ to $\gamma_2$ and from $\gamma'_1$ to $\gamma'_2$, respectively.\newline
$\bullet\;\Lambda=\{T_\ell\}_{\ell=1}^p$ and $\Lambda'=\{T_{\ell'}\}_{\ell'=1}^{p'}$ are compatible minimal pizzas on $T$ and $T'$
associated with the distance functions $f(x)=dist(x,T')$ and $g(x')=dist(x',T)$, respectively (see Definition \ref{compatible}).\newline
$\bullet\;\M=\{M_i\}_{i=1}^m$ and $\M'=\{M'_i\}_{i=1}^m$ are the sets of maximum zones of $\Lambda$ and $\Lambda'$,
ordered according to the orientations of $T$ and $T'$.
We assume that $m>0$, thus at least one maximum zone exists.\newline
$\bullet\;$ The permutation $\sigma$ of the set $[m]=\{1,\ldots,m\}$ is the characteristic permutation between the sets $\M$ and $\M'$ of maximum zones
of $\Lambda$ and $\Lambda'$ (see Definition \ref{characteristic}).\newline
$\bullet\;$ The sets $\mathcal L$ and $\mathcal L'$ of coherent pizza slices of $\Lambda$ and $\Lambda'$ have the same number of elements $L$.
The characteristic correspondence $\tau:\ell'=\tau(\ell)$ defines a permutation $\upsilon$ of the set $[L]=\{1,\ldots,L\}$
and a sign function $s:[L]\to\{+,-\}$ (see Definitions \ref{def:tau},\ \ref{upsilon}).\newline
$\bullet\;$ The disjoint unions $\K=\M\cup\mathcal L$ and $\K'=\M'\cup\mathcal L'$ of the sets
of maximum zones and coherent pizza slices of $\Lambda$ and $\Lambda'$, respectively, have the same number of elements $K$.
The permutations $\sigma$ and $\upsilon$ define the combined characteristic permutation $\omega$ of the set $[K]=\{1,\ldots,K\}$ (see Definition \ref{omega}).\newline
Note that the permutation $\sigma$ acts on maximum zones of $\Lambda$ and $\Lambda'$,
and the correspondence $\tau$ acts (although not one-to-one, see Remark \ref{rem:tau}) on the boundary arcs of
coherent pizza slices of $\Lambda$ and $\Lambda'$.
Additionally, the boundary arcs $\gamma_1$ and $\gamma_2$ of $T$ naturally correspond to the boundary arcs $\gamma'_1$ and $\gamma'_2$ of $T'$.
All these arcs are called \emph{essential} (see Definition \ref{def:primary}), and all other arcs of $\Lambda$ and $\Lambda'$ are called \emph{non-essential}. Each non-essential arc of $\Lambda$ is a common boundary arc of two transverse pizza slices.
A pre-pizza $\tilde\Lambda$ is obtained from $\Lambda$ by deleting non-essential arcs and replacing
two transverse pizza slices adjacent to each non-essential arc by a single H\"older triangle of $\tilde\Lambda$.
The pizza $\Lambda$ can be recovered from a pre-pizza $\tilde\Lambda$ by restoring non-essential arcs (see Remark \ref{pre-pizza:remark}).

If a pair $(T,T')$ is totally transverse, then essential arcs of the pizzas $\Lambda$ and $\Lambda'$ are either the boundary arcs of $T$ and $T'$
or the arcs $\lambda_\ell$ and $\lambda'_{\ell'}$ in the maximum zones of $\Lambda$ and $\Lambda'$, respectively.
In particular, $\Lambda$ and $\Lambda'$ have the same number $n$ of essential arcs. This allows us to define an extended permutation $\pi$ (see Definition \ref{def:pi}) of the set $[n]=\{0,\ldots,n-1\}$, and to formulate the necessary and sufficient conditions for the existence of
a totally transversal normal pair of H\"older triangles in Section \ref{sec:blocks}. 

That is not true for a general normal pair $(T,T')$:
since the action of $\tau$ on the boundary arcs of coherent pizza slices is not one-to-one,
the numbers of essential arcs, and even the numbers of pizza slices, of $\Lambda$ and $\Lambda'$ may be different 
(see Examples \ref{example:varpi} and \ref{example:varpi2}).
To restore equality, we define \emph{twin pre-pizzas} $\check\Lambda$ and $\check\Lambda'$, 
obtained from pre-pizzas $\tilde\Lambda$ and $\tilde\Lambda'$ by adding new arcs, called \emph{twin arcs}.
Twin arcs are defined so that $\tau$ becomes a one-to-one correspondence between the sets of 
boundary arcs of coherent pizza slices of $\check\Lambda$ and $\check\Lambda'$.
Pre-pizzas $\tilde\Lambda$ and $\tilde\Lambda'$ can be recovered from twin pre-pizzas
 $\check\Lambda$ and $\check\Lambda'$ by replacing each pair of twin arcs by a single arc.
This allows us to define a permutation $\varpi$ (see Definition \ref{def:varpi}) 
of the set $[\mathcal N]=\{0,\ldots,\mathcal N-1\}$ of indices of arcs $\check\lambda_k$ of $\check\Lambda$, 
which is the same as the set of indices of arcs $\check\lambda'_{k'}$ of $\check\Lambda'$ (see Lemma \ref{lem:mathcalN})
using the permutation $\sigma$ of the maximum zones of $\Lambda$ and $\Lambda'$,
 the permutation $\upsilon$ of coherent pizza slices of $\Lambda$ and $\Lambda'$ defined by $\tau$,
 and the sign function $s$ on coherent pizza slices of $\Lambda$ and $\Lambda'$ defined by $\tau$.
It is an analog of the permutation $\pi$ defined for the totally transverse case in Section \ref{sec:blocks}.

The twin pre-pizza defines the relation of the pizzas of the two distance function.  
It determines the number of the pizza slices of the distance function $g$. 
Notices that the number of pizza slices of $f$ may be different from the number of the pizza slices 
of the distance function $g$. It is illustrated by the example \ref{example:varpi}.

Propositions \ref{prop:varpi1} and \ref{prop:varpi2} describe the combinatorial and metric properties of $\varpi$ 
for a given normal pair $(T,T')$ of H\"older triangles.
These properties become \emph{admissibility conditions} (see Definition \ref{def:admissible}) for the permutation $\varpi$,
necessary and sufficient for the existence and uniqueness of a normal pair $(T,T')$, when only a pizza $\Lambda$ on $T$,
the permutations $\sigma$ and $\upsilon$, and the sign function $s$ are given.

Given a minimal pizza $\Lambda$ associated with a non-negative Lipschitz function $f$ on a normally embedded H\"older triangle $T$ and an admissible permutation $\varpi$,
the construction of a normal pair $(T,T')$ is similar to that in the previous section, although the admissibility conditions are more complicated.
For a given pizza $\Lambda$, permutations $\sigma$ and $\upsilon$, and a sign function $s$, an admissible permutation $\varpi$, if exists, is unique,
and the existence conditions for $\varpi$ are explicitly formulated.

\begin{definition}\label{def:primary}\normalfont
Let $\Lambda=\{T_\ell\}_{\ell=1}^p$, where $T_\ell=T(\lambda_{\ell-1},\lambda_\ell)$,
be a minimal pizza on a normally embedded H\"older triangle $T$ associated with a non-negative Lipschitz function $f$ on $T$, and
let $\{D_\ell\}_{\ell=0}^p$ be the pizza zones of $\Lambda$ (see Lemma \ref{MP}) such that $\lambda_\ell\in D_\ell$.
If $T_\ell$ is a coherent pizza slice, then $D_{\ell-1}$ and $D_\ell$ are called \emph{primary pizza zones},
the arcs $\lambda_{\ell-1}$ and $\lambda_\ell$ are called \emph{primary arcs}, and
their indices are called \emph{primary indices} of $\Lambda$.
The pair $(D_{\ell-1},D_\ell)$ is called a \emph{primary pair} of pizza zones,
the pair $(\lambda_{\ell-1},\lambda_\ell)$ is called a \emph{primary pair} of arcs,
and the pair $(\ell-1,\ell)$ is called a \emph{primary pair} of indices of $\Lambda$.\newline
A pizza zone $D_\ell$ and an arc $\lambda_\ell\in D_\ell$ are \emph{essential} if $D_\ell$ is
either a boundary arc of $T$, or a maximum zone of $\Lambda$, or a primary pizza zone.
Otherwise, $D_\ell$ and $\lambda_\ell$ are \emph{non-essential}.
\end{definition}

\begin{lemma}\label{lem:non-essential}
If $D_\ell$ is a non-essential pizza zone of $\Lambda$, then both $D_{\ell-1}$ and $D_{\ell+1}$ are essential pizza zones of $\Lambda$,
and $D_\ell$ is a transverse minimum zone of $\Lambda$ adjacent to transverse pizza slices $T_\ell$ and $T_{\ell+1}$. In particular,
\begin{equation}\label{non-essential}
\mu(D_\ell)=q_\ell=\beta_{\ell-1}=\beta_\ell<\min(q_{\ell-1},q_{\ell+1}).
\end{equation}
\end{lemma}

\begin{proof}
Since $\Lambda$ is a minimal pizza on $T$ associated with $f$, a common pizza zone $D_\ell$ of two transverse pizza slices of $\Lambda$
is either a maximum or a minimum zone.
As $D_\ell$ is non-essential, it is a minimum zone. This implies that each of the zones $D_{\ell-1}$ and $D_{\ell+1}$ is either a maximum zone or
a boundary zone of a coherent pizza slice, thus both $D_{\ell-1}$ and $D_{\ell+1}$ are essential pizza zones of $\Lambda$.
Since $D_\ell$ is adjacent to two transverse pizza slices, it is a transverse zone.
\end{proof}

\begin{definition}\label{def:pre-pizza}\normalfont
Let $\{\tilde D_j\}_{j=0}^{N-1}$ be essential pizza zones of $\Lambda$ (see Definition \ref{def:primary}) ordered according to orientation of $T$,
and let $\tilde\lambda_j\in\tilde D_j$ be essential arcs of $\Lambda$.
Let $\tilde T_j=T(\tilde\lambda_{j-1},\tilde\lambda_j)$ for $0<j<N$.
Then each H\"older triangle $\tilde T_j$ is either a coherent pizza slice of $\Lambda$, or a transverse pizza slice $T_\ell$ of $\Lambda$ such that
both zones $D_{\ell-1}$ and $D_\ell$ are essential,
or the union of two adjacent transverse pizza slices $T_\ell$ and $T_{\ell+1}$ of $\Lambda$ such that $D_\ell$
is a non-essential pizza zone of $\Lambda$.
A \emph{pre-pizza} $\tilde\Lambda=\{\tilde T_j\}_{j=1}^{N-1}$ on $T$ associated with $f$ is a decomposition of $T$ into H\"older triangles
$\tilde T_j$ with the following \emph{toppings} $\{\tilde q_j,\tilde\beta_j,\tilde Q_j,\tilde\mu_j,\tilde\nu_j\}$:\newline
{\bf 1)} If $\tilde\lambda_j=\lambda_\ell$, then $\tilde q_j=ord_{\tilde\lambda_j} f=q_\ell$ and $\tilde\nu_j=\nu_T(\tilde\lambda_j,f)=\nu_\ell$.\newline
{\bf 2)} If $\tilde T_j=T_\ell$ is a coherent pizza slice of $\Lambda$, then
 $\tilde\beta_j=\beta_\ell,\;\tilde Q_j=[\tilde q_{j-1},\tilde q_j]=Q_\ell$, and
 $\tilde\mu_j(q)=\mu_\ell(q)$ is an affine function on $\tilde Q_j$, or $\tilde\mu_j(\tilde q_j)=\tilde\beta_j$
 when $\tilde Q_j=\{\tilde q_j\}$ is a point.\newline
{\bf 3)} If $\tilde D_j=D_\ell=M_i$ is a maximum zone of $\Lambda$, then $\tilde q_j=q_\ell=\bar q_i,\;\tilde\beta_j=\beta_\ell$ when $j>0$ and
$\tilde\beta_{j+1}=\beta_{\ell+1}$ when $j<N-1$. In that case $\tilde D_j$ is called a \emph{maximum zone} of $\tilde\Lambda$.\newline
{\bf 4)} If $\tilde T_j=T_\ell\cup T_{\ell+1}$ is the union of two pizza slices of $\Lambda$, then $\tilde q_{j-1}=q_{\ell-1},\;\tilde q_j=q_{\ell+1}$ and
$\tilde\beta_j=\beta_{\ell-1}=\beta_\ell<\min(\tilde q_{j-1},\tilde q_j)$. In that case, $(j-1,j)$ is called a \emph{gap pair} of indices of $\tilde\Lambda$.\newline
A pair $(j-1,j)$ of consecutive indices of $\tilde\Lambda$ is called \emph{primary} if $\tilde T_j$ is a coherent pizza slice and \emph{secondary} otherwise.
If $\Lambda$ is a totally transverse pizza, then all pairs $(j-1,j)$ of indices of $\tilde\Lambda$ are secondary, and $\tilde\Lambda$ is called a
\emph{totally transverse} pre-pizza.
\end{definition}

\begin{lemma}\label{max-pre-pizza}
A pizza zone $\tilde D_j$ of a pre-pizza $\tilde\Lambda=\{\tilde T_j\}_{j=1}^{N-1}$ corresponding to $\Lambda$
is a maximum zone if, and only if, the following conditions are satisfied:
\begin{equation}\label{maxpreA}
\text{If\ }j>0,\;\text{then either\ }(j-1,j)\;\text{is a gap pair or\ }\tilde q_j\ge\tilde q_{j-1}.
\end{equation}
\begin{equation}\label{maxpreB}
\text{If\ }j<N-1,\;\text{then either\ }(j,j+1)\;\text{is a gap pair or\ }\tilde q_j\ge\tilde q_{j+1}.
\end{equation}
\end{lemma}

\begin{proof}
If $j>0$ and $\tilde D_j$ is a maximum zone of $\tilde\Lambda$, then
$\tilde T_j$ is either a coherent pizza slice $T_\ell$ of $\Lambda$, or a transverse pizza slice $T_\ell$ of $\Lambda$,
or the union of two transverse pizza slices
$T_{\ell-1}$ and $T_\ell$ of $\Lambda$ with a common minimum pizza zone $D_\ell$.
In the first case, $\tilde q_j=q_\ell\ge\tilde q_{j-1}=q_{\ell-1}$.
In the second case, $\tilde q_j=q_\ell>q_{\ell-1}=\tilde q_{j-1}$.
In the third case, $(j-1,j)$ is a gap pair of indices of $\tilde\Lambda$.
Thus condition (\ref{maxpreA}) is satisfied for $\tilde D_j$.\newline
Similarly, if $j<N-1$ and $\tilde D_j$ is a maximum zone, then condition (\ref{maxpreB}) is satisfied for $\tilde D_j$.\newline
If $\tilde D_j$ is not a maximum zone of $\tilde\Lambda$, then either $(j-1,j)$ or $(j,j+1)$ is not a gap pair.
If $(j,j+1)$ is a gap pair and $(j-1,j)$ is not a gap pair, then $\tilde T_j=T_\ell$ is a pizza slice of $\Lambda$
and $\tilde q_{j-1}=q_{\ell-1}>q_\ell=\tilde q_j$, thus condition (\ref{maxpreA}) is not satisfied.\newline
Similarly, if $(j-1,j)$ is a gap pair and $(j,j+1)$ is not a gap pair, then $\tilde T_{j+1}=T_\ell$ is a pizza slice of $\Lambda$
and $\tilde q_{j+1}=q_{\ell+1}>q_\ell=\tilde q_j$, thus condition (\ref{maxpreB}) is not satisfied.\newline
If both $(j-1,j)$ and $(j,j+1)$ are not gap pairs, then
either $j=0$ and $\tilde T_1=T_1$ is a pizza slice of $\Lambda$ such that $\tilde q_1=q_1>q_0=\tilde q_0$,
or $j=N-1$ and $\tilde T_{N-1}=T_p$ is a pizza slice of $\Lambda$ such that $\tilde q_{N-2}=q_{p-1}>q_p=\tilde q_{N-1}$, or
$0<j<N-1$,
$\tilde T_j=T_\ell$ and $\tilde T_{j+1}=T_{\ell+1}$ are pizza slices of $\Lambda$,
and either $\tilde q_{j-1}=q_{\ell-1}>q_\ell=\tilde q_j$ or $\tilde q_{j+1}=q_{\ell+1}>q_\ell=\tilde q_j$,
or both $\tilde q_{j-1}>\tilde q_j$ and $\tilde q_{j+1}>\tilde q_j$, thus at least one of conditions (\ref{maxpreA}) and (\ref{maxpreB}) is not satisfied.
\end{proof}

\begin{remark}\label{pre-pizza:remark} \normalfont
A minimal pizza $\Lambda$ on $T$ associated with $f$ can be recovered as a refinement of a pre-pizza $\tilde\Lambda$ by
adding a generic arc in each H\"older triangle $\tilde T_j$ such that $(j-1,j)$ is a gap pair.
If $f$ is totally transverse, then all pizza zones of $\Lambda$ are either maximum or minimum zones, with all interior minimum zones being non-essential.
Thus a totally transverse pre-pizza $\tilde\Lambda$ on $T$ is a decomposition of $T$ by the arcs $\lambda_\ell$ in
interior maximum zones of $\Lambda$.
\end{remark}

\begin{definition}\label{def:twin-pre-pizza}\normalfont
Let $\{\tilde\lambda_j\}_{j=0}^{N-1}$ be the arcs of a
pre-pizza $\tilde\Lambda=\{\tilde T_j\}_{j=1}^{N-1}$ on $T$ associated with $f$.
A \emph{twin pre-pizza} $\check\Lambda$ on $T$ associated with $f$ is a decomposition of $T$ into H\"older triangles $\check T_k$
obtained by the following operations:\newline
{\bf 1.} Each arc $\tilde\lambda_j$ common to coherent pizza slices $\tilde T_j$ and $\tilde T_{j+1}$,
such that $\tilde D_j$ is a transverse pizza zone of $\tilde\Lambda$ (i.e., $\tilde\nu_j=\mu(\tilde D_j)=\tilde q_j$,
see Definition \ref{def:pizzaslicezone-transverse}) and $\tilde D_j$ is not a maximum zone,
is replaced by \emph{twin arcs} $\tilde\lambda^-_j\in\tilde D_j$ and $\tilde\lambda^+_j\in\tilde D_j$, such that
$\tilde\lambda^-_j \prec\tilde\lambda^+_j$ and $tord(\tilde\lambda^-_j,\tilde\lambda^+_j)=\tilde q_j$.\newline
{\bf 2.} If the boundary arc $\gamma_1=\tilde\lambda_0$ of $T$ is a transverse minimum zone (see Definition \ref{def:pizzaslicezone-transverse}) and
$\tilde T_1$ is a coherent pizza slice of $\tilde\Lambda$, then an interior arc $\tilde\lambda^+_0\subset\tilde T_1$,
such that $tord(\tilde\lambda^+_0,\tilde\lambda_0)=\tilde q_0$, is added.
Since $\mu_T(\tilde\lambda_0,f)=\tilde q_0$, we have $ord_{\tilde\lambda^+_0} f=\tilde q_0$.\newline
{\bf 3.} If the boundary arc $\gamma_2=\tilde\lambda_{N-1}$ of $T$ is a transverse minimum zone and $\tilde T_{N-1}$ is a coherent pizza slice of $\tilde\Lambda$, then an interior arc $\tilde\lambda^-_{N-1}\subset\tilde T_{N-1}$,
such that $tord(\tilde\lambda^-_{N-1},\tilde\lambda_{N-1})=\tilde q_{N-1}$, is added.
Since $\mu_T(\tilde\lambda_{N-1},f)=\tilde q_{N-1}$, we have $ord_{\tilde\lambda^-_{N-1}}f=\tilde q_{N-1}$.\newline
Applying these operations to $\tilde\Lambda$ and ordering all arcs according to orientation of $T$, we obtain a \emph{twin pre-pizza}
$\check\Lambda=\{\check T_k\}_{k=1}^{\mathcal N-1}$ on $T$ associated with $f$, where $\check T_k=T(\check\lambda_{k-1},\check\lambda_k)$
is a $\check\beta_k$-H\"older triangle, with the set $\{\check\lambda_k\}_{k=0}^{\mathcal N-1}$ of $\mathcal N\ge N$ arcs.
If $\check\lambda_k\in\tilde D_j$, then $\check q_k=\tilde q_j$ and $\check\nu_k=\tilde\nu_j$.
If $\check\lambda_1=\tilde\lambda^+_0$, then $\check q_1=\check\nu_1=\tilde q_0$.
If $\check\lambda_{\mathcal N-2}=\tilde\lambda^-_{N-1}$, then $\check q_{\mathcal N-2}=\check\nu_{\mathcal N-2}=\tilde q_{N-1}$.
A pair of arcs $(\check\lambda_{k-1},\check\lambda_k)$ such that $\check T_k\subset\tilde T_j$ is a coherent pizza slice for $f$,
and the pair of indices $(k-1,k)$ of $\check\Lambda$, is \emph{primary}.
In that case, $\check\beta_k=\tilde\beta_j,\; \check q_{k-1}=\tilde q_{j-1},\; \check q_k=\tilde q_j,\; \check Q_k=\tilde Q_j$,
 and $\check\mu_k(q)=\tilde\mu_j(q)$ is an affine function on $\check Q_k$, or $\check\mu_k(\check q_k)=\check\beta_k$ when
 $\check Q_k=\{\check q_k\}$ is a point.
If a pair $(k-1,k)$ of indices of $\check\Lambda$ is not primary, then it is \emph{secondary}.\newline
If $\check\lambda_k\subset\tilde D_j$, where $\tilde D_j$ is a maximum zone of $\tilde\Lambda$, then $\check\lambda_k$ is a \emph{maximum arc}
of $\check\Lambda$ and $k$ is a \emph{maximum index} of $\check\Lambda$.
In that case, $\check q_k=\tilde q_j=\bar q_i$, $\check\beta_k=\tilde\beta_j$ and $\check\beta_{k+1}=\tilde\beta_{j+1}$.\newline
If $\check T_k=\tilde T_j$, where $(j-1,j)$ is a gap pair of indices of $\tilde\Lambda$,
then $(k-1,k)$ is a \emph{gap pair} of indices of $\check\Lambda$,
$\check q_{k-1}=\tilde q_{j-1},\; \check q_k=\tilde q_j$ and $\check\beta_k=\tilde\beta_j<\min(\check q_{k-1},\check q_k)$.\newline
\emph{Twin pairs of arcs} of $\check\Lambda$ are defined as either $(\check\lambda_{k-1},\check\lambda_k)=(\tilde\lambda^-_j,\tilde\lambda^+_j)$,
 where $\tilde\lambda^-_j$ and $\tilde\lambda^+_j$ are twin arcs in $\tilde D_j$, or $(\check\lambda_0,\check\lambda_1)$ when $\check\lambda_1=\tilde\lambda_0^+$,
or $(\check\lambda_{\mathcal N-2},\check\lambda_{\mathcal N-1})$ when $\check\lambda_{\mathcal N-2}=\tilde\lambda_{N-1}^-$, and $(k-1,k)$ is called a \emph{twin pair} of indices of $\check\Lambda$.
The number $\mathcal N$ of arcs $\check\lambda_k$ of $\check\Lambda$ is equal to the number $N$ of arcs $\tilde\lambda_j$ of $\tilde\Lambda$
plus the number os twin pairs of arcs of $\check\Lambda$.
\end{definition}

\begin{remark}\label{rem:twin-pre-pizza}\normalfont Let $\{\check\lambda_k\}_{k=0}^{\mathcal N-1}$ be the arcs of a twin pre-pizza $\check\Lambda$
on $T$ corresponding to a pre-pizza $\tilde\Lambda$ associated with $f$.
If $(\check\lambda_{k-1},\check\lambda_k)$ is a secondary pair of arcs of $\check\Lambda$, then exactly one of the following three properties holds:\newline
{\bf(A)} $\check\lambda_{k-1}$ and $\check\lambda_k$ are twin arcs and $\check q_{k-1}=\check q_k=\check\beta_k$.
If $1<k<\mathcal N-1$, then $\check\lambda_{k-1}$ and $\check\lambda_k$ belong to a transverse pizza zone $\tilde D_j$ of
$\tilde\Lambda$ adjacent to two coherent pizza slices, such that $\tilde D_j$ is not a maximum zone of $\tilde\Lambda$,
thus both $(k-2,k-1)$ and $(k,k+1)$ are primary pairs of indices of $\check\Lambda$ corresponding to primary pairs $(j-1,j)$ and $(j,j+1)$ of indices of $\tilde\Lambda$, such that
\begin{equation}\label{maxtwin}
\max(\check\mu_{k-1}(\check q_{k-1}),\check\mu_{k+1}(\check q_k))=\check\beta_k,\quad\check q_{k-1}=\check q_k<\max(\check q_{k-2},\check q_{k+1})=\max(\tilde q_{j-1},\tilde
q_{j+1}).
\end{equation}
If $k=1$, then $\check\lambda_0=\gamma_1$ is a transverse boundary arc of $T$ which is not a maximum arc of $\tilde\Lambda$,
thus $(1,2)$ is a primary pair of indices of $\check\Lambda$
corresponding to a primary pair $(0,1)$ of indices of $\tilde\Lambda$, such that
\begin{equation}\label{maxtwinB}
 \check\mu_2(\check q_1)=\check\beta_1,\quad\check q_0=\check q_1<\check q_2=\tilde q_1\quad\text{(see (\ref{maxpreB}) in Lemma \ref{max-pre-pizza})}.
\end{equation}
If $k=\mathcal N-1$, then $\gamma_2$ is a transverse boundary arc of $T$ which is not a maximum arc of $\tilde\Lambda$,
thus $(\mathcal N-3,\mathcal N-2)$ is a primary pair of indices of $\check\Lambda$, corresponding to a primary pair $(N-2,N-1)$ of indices of $\tilde\Lambda$,
such that
\begin{equation}\label{maxtwinA}
 \check\mu_{\mathcal N-2}(\check q_{\mathcal N-2})=\check\beta_{\mathcal N-1},\quad\check q_{\mathcal N-2}=\check q_{\mathcal N-1}<\check q_{\mathcal N-3}=\tilde
 q_{N-2}\quad\text{(see (\ref{maxpreA}) in Lemma \ref{max-pre-pizza})}.
\end{equation}
{\bf(B)} $T(\check\lambda_{k-1},\check\lambda_k)$
is a transverse pizza slice of $\Lambda$, thus $\check q_{k-1}=\check\nu_{k-1}\ne\check q_k=\check\nu_k$
and $\check\beta_k=\min(\check q_{k-1},\check q_k)$.\newline
{\bf(C)} $T(\check\lambda_{k-1},\check\lambda_k)$
is the union of two transverse pizza slices of $\Lambda$,
thus $\check\beta_k<\min(\check q_{k-1},\check q_k)$, and $(k-1,k)$ is a gap pair of indices of $\check\Lambda$.
\end{remark}

\begin{figure}
\centering
\includegraphics[width=6in]{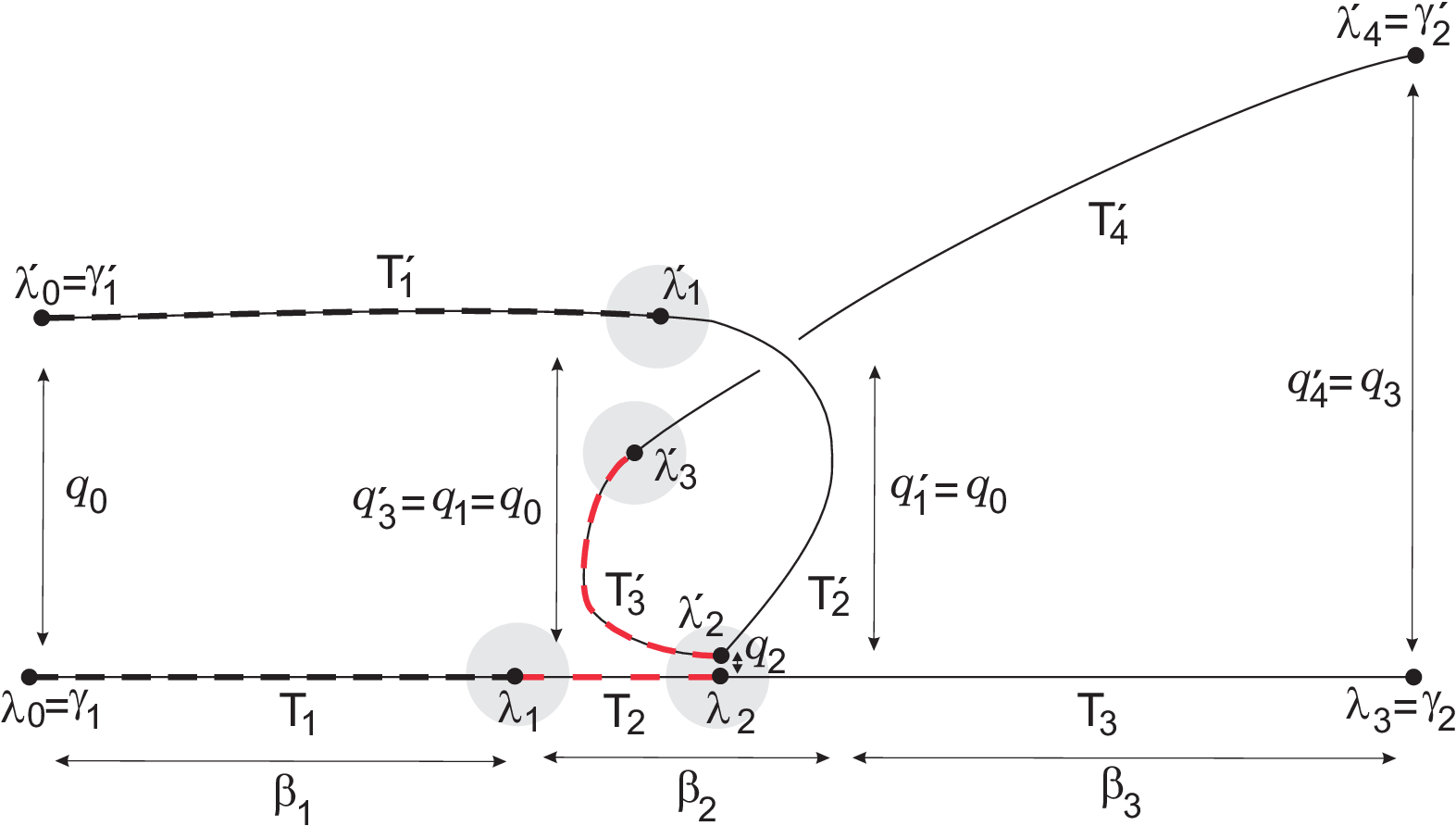}
\caption{A normal pair of H\"older triangles in Example \ref{example:varpi}. Coherent pizza slices are shown in dashed lines. }\label{fig:varpi}
\end{figure}

\begin{figure}
\centering
\includegraphics[width=5in]{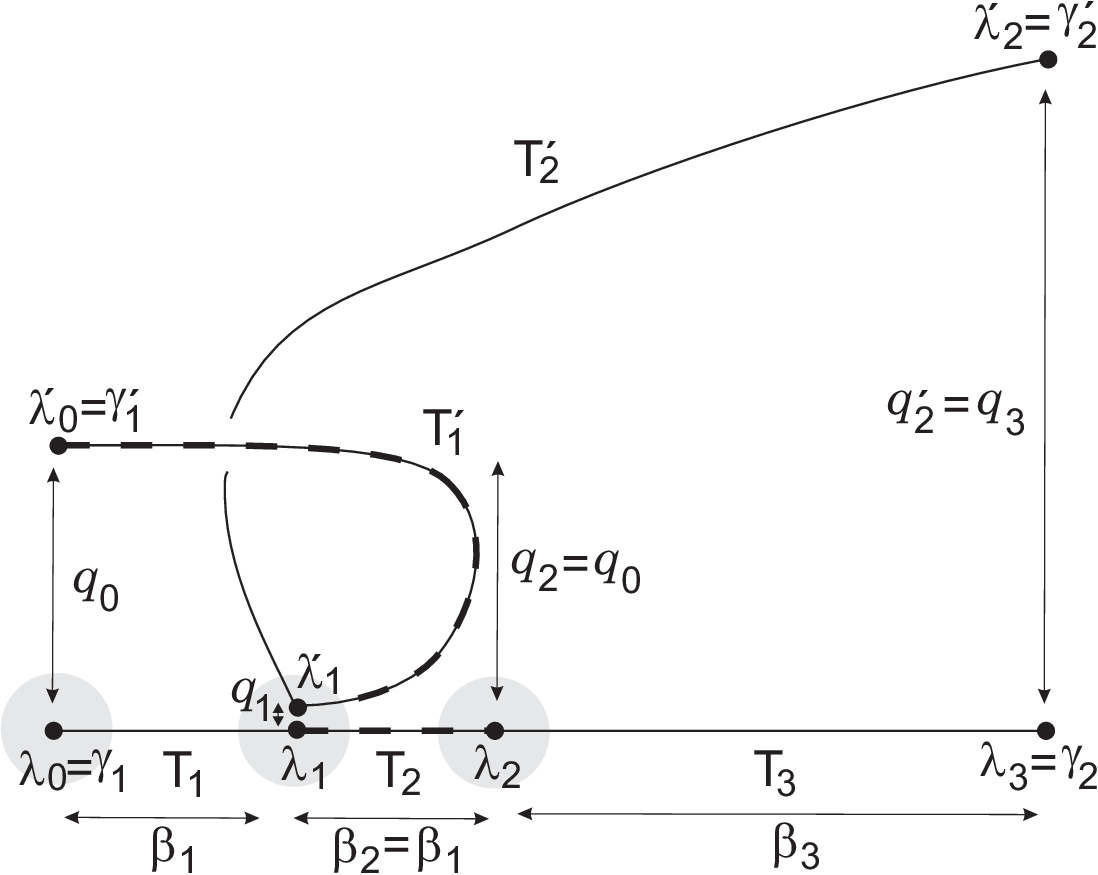}
\caption{A normal pair of H\"older triangles in Example \ref{example:varpi2}. Coherent pizza slices are shown in dashed lines.}\label{fig:varpi2}
\end{figure}

\begin{example}\label{example:varpi}\normalfont
A normal pair $(T,T')$ of H\"older triangles in Figure \ref{fig:varpi} has different numbers of pizza slices in $T$ and $T'$.
This is related to the presence of twin arcs in $T$ and the absence of them in $T'$.
The minimal pizza $\Lambda$ on $T$ associated with the distance function $f$ has three pizza slices $T_1,\,T_2$ and $T_3$,
with exponents $\beta_1,\,\beta_2$ and $\beta_3$, such that $\beta_2>\beta_1$ and $\beta_2>\beta_3$.
The minimal pizza $\Lambda'$ on $T'$ associated with the distance function $g$ has four pizza slices $T'_1,\,T'_2,\,T'_3$ and $T'_4$,
with exponents $\beta'_1=\beta_1,\,\beta'_2=\beta'_3=\beta_2$ and $\beta'_4=\beta_3$.
The pizza slices $T_1$ and $T_2$ of $\Lambda$, and the pizza slices $T'_1$ and $T'_3$ of $\Lambda'$, are coherent.
The orders $q_\ell$ of $f$ on the pizza zones $D_\ell$ of $\Lambda$ satisfy inequalities
$q_2>q_0=q_1>q_3$, and the orders $q'_{\ell'}$ of $g$ on the pizza zones $D'_{\ell'}$ of $\Lambda'$ satisfy equalities $q'_0=q'_1=q'_3=q_0$, $q'_2=q_2$ and $q'_4=q_3$.
The pizza zones $D_0=\{\gamma_1\}=M_1$ and $D_2=M_2$ are the only maximum zones of $\Lambda$, and the permutation $\sigma$ is trivial: $\sigma(1)=1$ and $\sigma(2)=2$.
The correspondence $\tau$, where $\tau(1)=1$ and $\tau(2)=3$, is positive on $T_1$ and negative on $T_2$.

There are no non-essential pizza zones in $T$, thus pre-pizza $\tilde\Lambda$ is the same as pizza:
$N=n=4$ and $\tilde\lambda_j=\lambda_j$ for $j=0,\ldots,3$.
The zones $\tilde D_1=D_1$ and $\tilde D_2=D_2$ of $\tilde\Lambda$ are transverse: $\mu(D_1)=q_0$ and $\mu(D_2)=q_2$.
The arc $\tilde\lambda_1=\lambda_1$, a common boundary arc of two coherent pizza slices, belongs to a transverse pizza zone $D_1$ of $\Lambda$ which is not a maximum zone.
By Definition \ref{def:twin-pre-pizza}, it should be replaced by twin arcs $\tilde\lambda^-_1$ and $\tilde\lambda^+_1$ in the twin pre-pizza $\check\Lambda$.
Thus $\mathcal N=5$, $\check\lambda_0=\gamma_1$, $\check\lambda_1=\tilde\lambda^-_1$,
$\check\lambda_2=\tilde\lambda^+_1$, $\check\lambda_3=\tilde\lambda_2=\lambda_2$, $\check\lambda_4=\gamma_2$.
There are no non-essential pizza zones in $T'$, thus pre-pizza $\tilde\Lambda'$ is the same as pizza: $N'=n'=5$ and $\tilde\lambda'_j=\lambda'_j$ for $j=0,\ldots,4$.
The zones $\tilde D'_1=D'_1$, $\tilde D'_2=D'_2$ and $\tilde D'_3=D'_3$ are transverse: $\mu(D'_1)=\mu(D'_3)=q_0$ and $\mu(D'_2)=q_2$.
Since each of the arcs $\lambda'_1,\;\lambda'_2$ and $\lambda'_3$ in these zones is not a common boundary arc of two
coherent pizza slices, the twin pre-pizza $\check\Lambda'$ is the same as pre-pizza $\tilde\Lambda'$:
$\mathcal N'=5=\mathcal N$ and $\check\lambda'_k=\tilde\lambda'_k=\lambda'_k$ for $k=0,\ldots,4$.
\end{example}

\begin{example}\label{example:varpi2}\normalfont
A normal pair $(T,T')$ of H\"older triangles in Figure \ref{fig:varpi2} has
different numbers of pizza slices in $T$ and $T'$. There are no twin arcs in $T$, but there is a twin arc in $T'$ corresponding to its boundary arc $\gamma'_1$.
The minimal pizza $\Lambda$ on $T$ associated with the distance function $f$ has three pizza slices $T_1,\,T_2$ and $T_3$,
with exponents $\beta_1,\,\beta_2$ and $\beta_3$, such that $\beta_1=\beta_2>\beta_3$.
The minimal pizza $\Lambda'$ on $T'$ associated with the distance function $g$ has two pizza slices $T'_1$ and $T'_2$,
with exponents $\beta'_1=\beta_1$ and $\beta'_2=\beta_3$.
The pizza slices $T_2$ of $\Lambda$ and $T'_1$ of $\Lambda'$ are coherent.
The orders $q_\ell$ of $f$ on the pizza zones $D_\ell$ of $\Lambda$ satisfy inequalities
$q_1>q_0>q_3$, and the orders $q'_{\ell'}$ of $g$ on the pizza zones $D'_{\ell'}$ of $\Lambda'$ satisfy equalities $q'_0=q_0,\;q'_1=q_1$ and $q'_2=q_3$.
The pizza zones $D_0=\{\gamma_1\}$ and $D_3$ are minimum zones of $\Lambda$, the pizza zone $D_1$ is the only maximum zone of $\Lambda$, and the permutation $\sigma$ is
trivial: $\sigma(1)=1$.
The correspondence $\tau$, where $\tau(2)=1$, is negative on $T_2$.

There are no non-essential pizza zones of $\Lambda$, thus $\tilde\Lambda=\Lambda,\;N=n=4$ and $\tilde\lambda_j=\lambda_j$ for $j=0,\ldots,3$.
The zones $\tilde D_1=D_1$ and $\tilde D_2=D_2$ are transverse: $\mu(D_1)=q_0$ and $\mu(D_2)=q_2$.
Since there is a single coherent pizza slice $T_2$ of $\Lambda$, and both boundary arcs of $T_2$ are interior arcs of $T$,
there are no twin arcs in $T$, thus $\check\Lambda=\tilde\Lambda$, $\mathcal N=N=4$ and $\check\lambda_k=\tilde\lambda_k$ for $k=0,\ldots,3$.
There are no non-essential pizza zones in $T'$, thus $\tilde\Lambda'=\Lambda',\;N'=n'=3$
and $\tilde\lambda'_j=\lambda'_j$ for $j=0,1,2$.
Since $\tilde\lambda'_0=\lambda'_0$ is a transverse boundary arc of $T'$ adjacent to a coherent pizza slice $\tilde T'_1=T'_1$ of $\Lambda'$,
and $\{\lambda'_0\}$ is a minimum zone of $\Lambda'$, the twin pre-pizza $\check\Lambda'$ contains a twin arc $\check\lambda'_1$
of the boundary arc $\check\lambda'_0=\tilde\lambda'_0=\gamma'_1$ of $T'$, such that $tord(\check\lambda'_0,\check\lambda'_1)=q'_0,\;
\mathcal N'=4,\;\check\lambda'_2=\tilde\lambda'_1$ and $\check\lambda'_3=\tilde\lambda'_2$.
\end{example}

\begin{remark}\label{rem:mathcalN}\normalfont
In Examples \ref{example:varpi} and \ref{example:varpi2} we have $\mathcal N'=\mathcal N$.
We are going to prove (see Lemma \ref{lem:mathcalN}) that this equality holds for any normal pair of H\"older triangles.
\end{remark}

\begin{lemma}\label{lem:mathcalN}
Let $(T,T')$ be a normal pair of H\"older triangles, with the distance functions $f(x)=dist(x,T')$ and $g(x')=dist(x',T)$
on $T$ and $T'$, respectively. Then the number $\mathcal N$ of arcs $\check\lambda_k$ in a twin pre-pizza $\check\Lambda$ on $T$
associated with $f$ is the same as
the number of arcs $\check\lambda'_{k'}$ in a twin pre-pizza $\check\Lambda'$ on $T'$ associated with $g$.
\end{lemma}

\begin{proof}
The characteristic permutation $\sigma$ defines a bijection between the sets of maximum zones of $\check\Lambda$
and $\check\Lambda'$. Thus the two sets have the same cardinality $m$. 
Similarly, the characteristic correspondence $\tau$ (see Definition \ref{def:tau}) defines a bijection between the sets of coherent pizza slices of $\check\Lambda$ and $\check\Lambda'$. Thus the two sets have the same cardinality $L$. 
It follows from \cite[Proposition 3.9]{BG} that there is a one-to-one correspondence between the sets of coherent interior pizza zones of
$\check\Lambda$ and $\check\Lambda'$.
Thus the two sets have the same cardinality $n_2$. 
It follows from \cite[Proposition 4.10]{BG} that there is a one-to-one correspondence between the set of transverse maximum zones of $\check\Lambda$
 adjacent to  two coherent pizza slices and the set of transverse maximum zones of $\check\Lambda'$ adjacent to two coherent pizza slices.
Thus the two sets have the same cardinality $m_2$. 
It follows from \cite[Proposition 4.10]{BG} that there is a one-to-one correspondence between the set of transverse maximum zones of $\check\Lambda$ not adjacent to any coherent pizza slices and the set of transverse maximum zones of $\check\Lambda'$ not adjacent to any coherent pizza slices.
Thus the two sets have the same cardinality $m_0$. 
The number $\delta$ of transverse boundary arcs of $T$ which are minimum zones of $\check\Lambda$ is the same as the number of
transverse boundary arcs of $T'$ which are minimum zones of $\check\Lambda'$.

We are going to prove that $\mathcal N = 2L-n_2-m_2+m_0+\delta$.
Note that the number in the right-hand side is the same for $\check\Lambda$ and $\check\Lambda'$.

The number $2L$ counts the boundary arcs of coherent pizza slices of $\Lambda$, with the common boundary arcs of two coherent pizza slices being counted twice.
If such a boundary arc belongs to a coherent pizza zone, then it must be counted once, and the number $n_2$ should be subtracted from $2L$.
If such an arc belongs to a transverse maximum zone, then it must be counted once, and the number $m_2$ should be subtracted.\newline
To each of the remaining common boundary arcs of two coherent slices correspond two twin arc of $\check\Lambda$, thus such arcs must be counted twice,
as they are counted in $2L$. 
If a boundary arc $\check\gamma$ of $T$ is adjacent to a coherent pizza slice, it is counted once in $2L$.
However, if $\check\gamma$ is transverse and $\{\check\gamma\}$ is a minimum zone, then it corresponds to twin arcs in $\check\Lambda$,
thus it should be counted twice.
Accordingly, the number $\delta$ should be added.\newline
This takes care of all boundary arcs of coherent pizza slices in $\check\Lambda$.\newline
 The remaining arcs in $\check\Lambda$ correspond to maximum zones not adjacent to any coherent pizza slices, thus the number $m_0$ should be added.
\end{proof}

\begin{definition}\label{def:varpi}\normalfont
Let $(T,T')$ be a normal pair of H\"older triangles with the distance functions $f(x)=dist(x,T')$ and $g(x')=dist(x',T)$ on $T$ and $T'$, respectively.
If $\check\Lambda=\{\check\lambda_k\}_{k=0}^{\mathcal N-1}$ is a twin pre-pizza on $T$ associated with $f$, and  $\check\Lambda'=\{\check\lambda'_{k'}\}_{k'=0}^{\mathcal N-1}$ is a twin pre-pizza on $T'$ associated with $g$, then
the combination of the characteristic permutation $\sigma$ and the characteristic correspondence $\tau$ of the pair $(T,T')$, together with assignments
$\check\lambda_0\mapsto\check\lambda'_0$ and $\check\lambda_{\mathcal N-1}\mapsto\check\lambda'_{\mathcal N-1}$
for the boundary arcs of $T$ and $T'$, defines a one-to-one correspondence $\check\tau$ between the sets of arcs $\{\check\lambda_k\}$ and  $\{\check\lambda'_{k'}\}$:
if $(\check\lambda_{k-1},\check\lambda_k)$ is a primary pair of arcs of $\check\Lambda$ corresponding to a coherent pizza slice $T_\ell$ of $\Lambda$,
then the coherent pizza slice of $\check\Lambda'$ corresponding to the coherent pizza slice $T'_{\tau(\ell)}$ of $\Lambda'$
is bounded by the arcs $\check\tau(\lambda_{\ell-1})$ and $\check\tau(\lambda_\ell)$ of $\check\Lambda'$,
in the same (resp., opposite) order when $\tau$ is positive (resp., negative) on $T_\ell$,
and if $\check\lambda_k$ is the arc of $\check\Lambda$ contained in a maximum zone $M_i$ of $\Lambda$, then the maximum zone $M'_{\sigma(i)}$
of $\Lambda'$ contains the arc $\check\tau(\check\lambda_k)$ of $\check\Lambda'$.
Remark \ref{rem:coherent-zone} implies that the definition of $\check\tau$ is consistent
when an arc of $\check\Lambda$ is a boundary arcs of a coherent pizza slice of $\Lambda$ and belongs either to a coherent pizza zone or to a maximum zone of $\Lambda$.
Since the sets of arcs of $\check\Lambda$ and $\check\Lambda'$ are ordered according to orientations of $T$ and $T'$
and have the same number of elements $\mathcal N$,
this defines a permutation $\varpi$ of the set $[\mathcal N]=\{0,\ldots,\mathcal N-1\}$.
\end{definition}

\begin{remark}\label{rem:varpi+}\normalfont
The permutation $\varpi$ in Example \ref{example:varpi} is $(0,1,3,2,4)$.
Note that the zone $\check D_1$ is ``split'' (see Remark \ref{rem:tau}): the arc $\check\lambda_1=\tilde\lambda^-_1$ is mapped to $\check\lambda'_1=\lambda'_1$, but the arc
$\check\lambda_2=\tilde\lambda^+_1$ is mapped to $\check\lambda'_3=\lambda'_3$.
\end{remark}

\begin{definition}\label{determined-varpi}\normalfont
Let $\Lambda$ be a minimal pizza on a normally embedded H\"older triangle $T$
with $m$ maximum zones and $L$ coherent pizza slices, associated with a
non-negative Lipschitz function $f$ on $T$, and let $(\sigma,\upsilon,s)$ be an allowable triple
of a permutation $\sigma$ of the set $[m]=\{1,\ldots,m\}$, a permutation $\upsilon$ of the set $[L]=\{1,\ldots,L\}$ and a sign function $s:[L]\to\{+,-\}$
(see Definition \ref{def:allowable}). Let $\check\Lambda=\{\check T_k\}_{k=1}^{\mathcal N}$ be the corresponding twin pre-pizza associated with $f$.
Then there is a unique permutation $\varpi$ of the set $[\mathcal N]=\{0,\ldots,\mathcal N-1\}$ with the following properties:\newline
$\mathbf 1)$ $\varpi(0)=0$ and $\varpi(\mathcal N-1)=\mathcal N-1$.\newline
$\mathbf 2)$ $\varpi$ is compatible with the permutation $\sigma$: if the arcs $\check\lambda_k$ and $\check\lambda_l$ of $\check\Lambda$
belong to maximum zones $M_i$ and $M_j$
of $\Lambda$, then $\varpi(k)<\varpi(l)$ if, and only if, $\sigma(i)<\sigma(j)$.\newline
$\mathbf 3)$ $\varpi$ is compatible with the permutation $\upsilon$: if the arcs $\check\lambda_k$ and $\check\lambda_l$ of $\check\Lambda$, where $k\ne l$, are boundary arcs
of the $i$-th and $j$-th coherent pizza slice of $\check\Lambda$, then $\varpi(i)<\varpi(j)$ if, and only if, $\upsilon(i)<\upsilon(j)$.\newline
$\mathbf 4)$ $\varpi$ is compatible with the sign function $s$ in the following sense: if $\check\lambda_{k-1}$ and $\check\lambda_k$ are boundary arcs of the $j$-th coherent pizza slice $\check T_k$
of $\check\Lambda$, then $\varpi(k-1)=\varpi(k)-1$ when $s(j)=+$ and $\varpi(k-1)=\varpi(k)+1$ when $s(j)=-$.\newline
$\mathbf 5)$ $\varpi$ is compatible with the permutation $\omega$ of the set $[K]=\{1,\ldots,K\}$, where $K=m+L$ (see Definition \ref{omega} and Proposition \ref{prop:allowable}) in the following sense: if an arc $\check\lambda_l$ of $\check\Lambda$ belongs to a maximum zone $M_i$, and if $\check T_k$ is a coherent pizza slice of $\check\Lambda$,
then $\varpi(l)$ is less than at least one of the indices $\varpi(k-1)$ and $\varpi(k)$ if, and only if,
the image by $\omega$ of the index in $[K]$ corresponding to $M_i$ is less than the image by $\omega$ of the index in $[K]$ corresponding to $\check T_k$.
Since $\omega$ is determined by $\Lambda,\;\sigma,\;\upsilon$ and $s$ (see Remark \ref{omega-remark}) this condition may be reformulated as follows:\newline
$\mathbf 5')$ If an arc $\check\lambda_l$ of $\check\Lambda$ belongs to a maximum zone $M_i$ of $\Lambda$, and if
$\check T_k=T(\check\lambda_{k-1},\check\lambda_k)$ is a coherent pizza slice of $\check\Lambda$
which belongs to a caravan $C$,
then $\varpi(l)<\max(\varpi(k-1),\varpi(k))$ if, and only if, $\sigma(i)\le j_-(C)$ (see Definition \ref{def:order} and Proposition \ref{prop:order}).\newline
A permutation $\varpi$ with these properties is called \emph{determined} by $\Lambda,\;\sigma,\;\upsilon$ and $s$.
\end{definition}

\begin{proposition}\label{sigma-upsilon-varpi}
Let $(T,T')$, be a normal pair of H\"older triangles with the minimal pizzas
$\Lambda$ and $\Lambda'$ on $T$ and $T'$ associated with the distance functions $f(x)=dist(x,T')$ and $g(x')=dist(x',T)$, respectively.
Let $m$ be the number of maximum zones of $\Lambda$. Let a permutation $\sigma$
 of the set $[m]=\{1,\ldots,m\}$ be the characteristic
permutation of the pair $(T,T')$.
Let $L$ be the number of coherent pizza slices of $\Lambda$,
and let the permutation $\upsilon$ of the set $[L]=\{1,\ldots,L\}$ and the sign function $s:[L]\to\{+,-\}$ be
defined by the characteristic correspondence $\tau$ of the pair $(T,T')$.
Let $\mathcal N$ be the number of arcs of the twin pre-pizza $\check\Lambda$ on $T$, and
let $\varpi$ of the permutation of the set $[\mathcal N]=\{0,\ldots,\mathcal N-1\}$.
Then the permutation $\varpi$ is determined by $\Lambda,\;\sigma,\;\upsilon$ and $s$ (see Definition \ref{determined-varpi}).
If $(T,T'')$ is another normal pair with the same minimal pizza $\Lambda$ on $T$, the same characteristic permutation $\sigma$, the same
permutation $\upsilon$ and the same sign function $s$, then the pairs $(T,T')$ and $(T,T'')$ have the same permutation $\varpi$ (see
Remark \ref{omega-remark}).
\end{proposition}

\begin{proof} According to Theorem \ref{omega-omega}, the combined characteristic permutation $\omega$ of the set $[K]=\{1.\ldots,K\}$,
where $K=m+L$, is determined by the pizza $\Lambda$, permutations $\sigma$ and $\upsilon$, and the sign function $s$.
The order of coherent pizza slices in $\check\Lambda'$ is the same as their order in $\Lambda'$,
which is determined by the permutation $\upsilon$.
The sign function $s$ defines the order of boundary arcs of each coherent pizza slice $\check T'$ of $\check\Lambda'$:
Let $\check T'$ be the $k'$-th coherent pizza slice of $\check\Lambda'$, where $k'=\upsilon(k)$,
and let $\check T$ be the $k$-th coherent pizza slice of $\check\Lambda$.
If $s(k)=+$ (resp., $s(k)=-$) then the boundary arcs of $\check T'$ have the same (resp., opposite) order as the corresponding boundary arcs of $\check T$.
If $\check T'\prec\check T''$ are two coherent pizza slices of $\check\Lambda'$, then either they do not have common boundary arcs
and the boundary arcs of $\check T'$ precede the boundary arcs of $\check T''$, or they have a common boundary arc $\lambda'$
(which belongs either to a coherent pizza zone or to a maximum zone of $\check\Lambda'$) and the other
boundary arc of $\check T'$ precedes $\lambda'$, while $\lambda'$ precedes the other boundary arc of $\check T''$.

To define $\varpi$, it is enough to know the number of maximum zones of $\check\Lambda'$ preceding each of its coherent pizza slice,
but it is the same as the number of maximum zones of $\Lambda'$ preceding each of its coherent pizza slice,
which was defined in Proposition \ref{prop:order}.
\end{proof}

\begin{remark}\label{general-block-remark}\normalfont
The properties of blocks in Lemma \ref{lem:block} hold if we replace the permutation $\chi$ of the set $[n]$ by the permutation $\varpi$
of the set $[\mathcal N]$.
\end{remark}

\begin{proposition}\label{prop:varpi1} Let $(T,T')$ be a normal pair of H\"older triangles with the minimal pizzas $\Lambda$ and $\Lambda'$
on $T$ and $T'$ associated with the distance functions $f(x)=dist(x,T')$ and $g(x')=dist(x',T)$, and let $\check\Lambda$ and $\check\Lambda'$ be the corresponding twin pre-pizzas.
If $\mathcal N$ is the number of arcs in $\check\Lambda$, same as the number of arcs in $\check\Lambda'$,
then the permutation $\varpi$ of the set $[\mathcal N]=\{0,\ldots,\mathcal N-1\}$ in Definition \ref{def:varpi} satisfies the following properties:\newline
{\bf 1.} $\varpi(0)=0$ and $\varpi({\mathcal N}-1)={\mathcal N}-1$.\newline
{\bf 2.} An arc $\check\lambda_k$ of $\check\Lambda$ belongs to a maximum zone $M_i$ if, and only if, an arc $\check\lambda'_{\varpi(k)}$ of $\check\Lambda'$ belongs to the maximum zone $M'_{\sigma(i)}$, thus $\varpi$ is compatible with $\sigma$ on the maximum zones (see Remark \ref{rem:allowable}).\newline
{\bf 3.} The permutation $\varpi$ is compatible with the permutation $\upsilon$, the sign function $s$, and the combined characteristic permutation $\omega$ of the pair $(T,T')$.\newline
{\bf 4.} If $C$ is a caravan of pizza slices $\Lambda$, then the indices in $\check\Lambda$ of the boundary arcs of coherent pizza slices of $C$
are mapped by $\varpi$ to the indices in $\check\Lambda'$ of the boundary arcs of the corresponding coherent pizza slices of the caravan $C'=\tau(C)$
of pizza slices of $\Lambda'$, in the same (resp., opposite) order when $s(C)=+$ (resp., when $s(C)=-$).
In particular, the set of indices in $\check\Lambda$ of the boundary arcs of all pizza slices of $C$ is a block of $\varpi$, on which $\varpi$ acts either preserving or reversing the order.\newline
{\bf 5.} The set of indices in $\check\Lambda$ of the maximum zones in $\mathcal A(C)$ adjacent to a caravan $C$ of pizza slices of $\Lambda$ is mapped by $\varpi$
to the set of indices in $\check\Lambda'$ of the maximum zones in $\mathcal A(C')$ adjacent to the caravan $C'=\tau(C)$ of pizza slices of $\Lambda'$.
\end{proposition}

\begin{proof} Properties {\bf 1} and {\bf 2} are part of Definition \ref{def:varpi} of the permutation $\varpi$.\newline
Property {\bf 3} follows from Definition \ref{def:varpi}, since the action of $\varpi$ on the primary pairs of indices of $\check\Lambda$
 is compatible with the action of $\tau$ on the boundary arcs of coherent pizza slices of $\Lambda$, while the permutation $\upsilon$,
 sign function $s$ and combined characteristic permutation $\omega$ are defined by that action of $\tau$.\newline
Property {\bf 4} follows from Proposition \ref{prop:order}, where this property is stated for the permutation $\upsilon$, and
from compatibility of $\varpi$ with $\tau$, $\upsilon$ and $\omega$.\newline
Property {\bf 5} follows from the normal embedding property of $T$ and $T'$.
\end{proof}

\begin{remark}\label{varpi-isometry}\normalfont
Let $(T,T')$ be a normal pair of H\"older triangles with the minimal pizzas $\Lambda$ and $\Lambda'$
on $T$ and $T'$ associated with the distance functions $f(x)=dist(x,T')$ and $g(x')=dist(x',T)$, and let $\check\Lambda$ and $\check\Lambda'$ be the corresponding twin pre-pizzas.
Then the map $\check\lambda_k\mapsto \check\lambda'_{\varpi(k)}$
from the set $\{\check\lambda_k\}$ of arcs of $\check\Lambda$ to the set $\{\check\lambda'_{k'}\}$ of arcs of $\check\Lambda'$ induced by $\varpi$
is an isometry with respect to the metric
$\xi$ (see Definition \ref{tord}) defined by the tangency order: $tord(\check\lambda'_{\varpi(i)},\check\lambda'_{\varpi(j)})=tord(\check\lambda_i,\check\lambda_j)$.
\end{remark}

\begin{proposition}\label{prop:varpi2} Let $(T,T')$ be a normal pair of H\"older triangles with the minimal pizzas $\Lambda$ and $\Lambda'$
on $T$ and $T'$ associated with the distance functions $f(x)=dist(x,T')$ and $g(x')=dist(x',T)$, and let $\check\Lambda$ and $\check\Lambda'$ be the corresponding twin pre-pizzas.\newline
If $(k'-1,k')$ is a secondary pair of indices of $\check\Lambda'$,
where $k'-1=\varpi(i)$ and $k'=\varpi(j)$, then $\check\beta'_{k'}=tord(\check\lambda_i,\check\lambda_j)$,
and one of the following properties holds:\newline
$\mathbf{(A)}$ $(k'-1,k')$ is a twin pair of indices of $\check\Lambda'$ and $\check q_i=\check q_j=tord(\check\lambda_i,\check\lambda_j)$.\newline
There are three subcases of $\mathbf{(A)}$:\newline
$\mathbf{(A_1)}$ $1<k'<\mathcal N-1,\;
(k'-2,k'-1)$ and $(k',k'+1)$ are primary pairs of indices of $\check\Lambda'$,
\begin{equation}\label{maxtwinprime}
\max(\check\mu'_{k'-1}(\check q'_{k'-1}),\check\mu'_{k'+1}(\check q'_{k'}))=\check\beta'_{k'},\quad\check q'_{k'-1}=\check q'_{k'}<\max(\check q'_{k'-2},\check q'_{k'+1}),
\end{equation}
the arcs $\check\lambda_i$ and $\check\lambda_j$ belong either to two transverse pizza zones of $\Lambda$ which are not maximum zones or
to the same transverse pizza zone of $\Lambda$ which is not a maximum zone.\newline
$\mathbf{(A_2)}$ $k'=1,\; i=0,\; (1,2)$ is a primary pair of indices of $\check\Lambda'$,
\begin{equation}\label{maxtwinprimeB}
\check\mu'_2(\check q'_1)=\check\beta'_1,\quad\check q'_0=\check q'_1<\check q'_2,
\end{equation}
the arc $\gamma_1$ is a transverse boundary arc of $T$ which is not a maximum zone,
the arc $\check\lambda_j$ belongs either to a transverse pizza zone of $\Lambda$ which is not a maximum zone
or to the $\check q_0$-order zone of $f$ containing $\gamma_1$.\newline
$\mathbf{(A_3)}$ $k'=\mathcal N-1,\; j=\mathcal N-1,\; (\mathcal N-3,\mathcal N-2)$ is a primary pair of indices of $\check\Lambda'$,
\begin{equation}\label{maxtwinprimeA}
\check\mu'_{\mathcal N-2}(\check q'_{\mathcal N-2})=\check\beta'_{\mathcal N-1},\quad\check q'_{\mathcal N-2}=\check q'_{\mathcal N-1}<\check q'_{\mathcal N-3},
\end{equation}
the arc $\gamma_2$ is a transverse boundary arc of $T$ which is not a maximum zone,
the arc $\check\lambda_i$ belongs either to a transverse pizza zone of $\Lambda$ which is not a maximum zone
or to the $\check q_{\mathcal N-1}$-order zone of $f$ containing $\gamma_2$.\newline
$\mathbf{(B)}$ $T(\check\lambda'_{k'-1},\check\lambda'_{k'})$ is a transverse pizza slice of $\Lambda'$,
$\check q_i\ne\check q_j$ and $\min(\check q_i,\check q_j)=tord(\check\lambda_i,\check\lambda_j)$.\newline
$\mathbf{(C)}$ $(k'-1,k')$ is a gap pair of indices of $\check\Lambda'$ and
$\min(\check q_i,\check q_j)>tord(\check\lambda_i,\check\lambda_j)$.
\end{proposition}

\begin{proof}
The statement follows from Remark  \ref{rem:twin-pre-pizza} applied to $\check\Lambda'$ instead of $\check\Lambda$, and from Remark \ref{varpi-isometry}.
In particular, (\ref{maxtwinprime}),  (\ref{maxtwinprimeB}) and (\ref{maxtwinprimeA}) follow from (\ref{maxtwin}), (\ref{maxtwinB}) and (\ref{maxtwinA}),
respectively.
Properties $\mathbf{(B)}$ and $\mathbf{(C)}$ follow from the corresponding properties in Remark \ref{rem:twin-pre-pizza}.
\end{proof}

The following theorem is an analog of Theorem \ref{beta-block}.

\begin{theorem}\label{thm:varpi}
Let ${\mathcal B}_{ij}=B_\varpi(\{\check\lambda_i,\check\lambda_j\})$ be the minimal block of $\varpi$ containing $\{i,j\}$,
and let ${\mathcal B}'_{ij}=B_{\varpi^{-1}}(\{\check\lambda'_i,\check\lambda'_j\})$ be the minimal block of $\varpi^{-1}$ containing $\{i,j\}$. Then
\begin{equation}\label{dart:blocks}
tord(\check\lambda_i,\check\lambda_j)\le tord(\check\lambda_k,\check\lambda_l)\; \text{for}\; \{k,l\}\subset {\mathcal B}_{ij},
\end{equation}
\begin{equation}\label{dart:blocks-prime}
tord(\check\lambda'_i,\check\lambda'_j)\le tord(\check\lambda'_k,\check\lambda'_l)\; \text{for}\; \{k,l\}\subset {\mathcal B}'_{ij}.
\end{equation}
\end{theorem}

\begin{remark}\label{rem:dart}\normalfont
If a pair $(T,T')$ is totally transverse, then ${\mathcal N}=n$, $\varpi=\pi$, the sets of arcs $\{\check\lambda_j\}$ in Definitions
\ref{def:varpi} and \ref{supporting-family} are the same,
and the inequalities (\ref{dart:blocks}) and (\ref{dart:blocks-prime})
are the same as the inequalities (\ref{beta:blocks}) and (\ref{beta:blocks-prime}).
\end{remark}

Let $\Lambda$ be a minimal pizza on a normally embedded
H\"older triangle $T$ associated with a non-negative Lipschitz function $f(x)$ on $T$, and let $\check\Lambda=\{\check T_k\}_{k=0}^{\mathcal N-1}$,
where $\check T_k=T(\check\lambda_{k-1},\check\lambda_k)$,  be the corresponding twin pre-pizza on $T$.
Let $\varpi:[\mathcal N]\to [\mathcal N]$, where $[\mathcal N]=\{0,\ldots,\mathcal N-1\}$, be a permutation, such that $\varpi(0)=0,\;\varpi(\mathcal N-1)=\mathcal N-1$,
and each primary pair $(k-1,k)$ of indices of $\check\Lambda$, corresponding to a coherent pizza slice $\check T_k$, is mapped by $\varpi$ to a pair of consecutive indices $(\varpi(k-1),\varpi(k))$, either $\varpi(k-1)=\varpi(k)-1$ or $\varpi(k-1)=\varpi(k)+1$.
We are going to formulate \emph{admissibility conditions} for the permutation $\varpi$, necessary and sufficient for the existence of
a normal pair $(T,T')$ of H\"older triangles such that $f$ is contact equivalent to the distance function $dist(x,T')$ on $T$ and
the permutation $\varpi$ is compatible with the permutations $\sigma$ and $\upsilon$, sign function $s$ and
combined characteristic permutation $\omega$ of the pair $(T,T')$.
If such a pair $(T,T')$ exists, then it is unique up to outer Lipschitz equivalence, due to Theorem \ref{complete}.

\begin{definition}\label{def:admissible}\normalfont
Let $\Lambda$ be a minimal pizza on a normally embedded H\"older triangle $T$,
associated with a non-negative Lipschitz function $f$ on $T$.
Let $m$ and $L$ be the number of maximum zones and coherent pizza slices of $\Lambda$, respectively.
Let $\sigma$ and $\upsilon$ be allowable permutations of $[m]=\{1,\ldots,m\}$ and $[L]=\{1,\ldots,L]$, respectively,
and let $s:[L]\to\{+,-\}$ be an allowable sign function (see Definition \ref{def:allowable}).
Let $\omega$ be the unique permutation of the set $[K]=\{1,\ldots,K\}$, where $K=m+L$, compatible
with $(\sigma,\upsilon,s)$ (see Proposition \ref{prop:allowable}).
Let $\check\Lambda=\{\check T_k\}_{k=1}^{\mathcal N}$, where $\check T_k=T(\check\lambda_{k-1},\check\lambda_k)$,
be the twin pre-pizza on $T$ corresponding to $\Lambda$.
A permutation $\varpi$ of the set $[\mathcal N]=\{0,\ldots,\mathcal N-1\}$, such that $\varpi(0)=0$ and $\varpi(\mathcal N-1)=\mathcal N-1$,
is called \emph{admissible} with respect to $\Lambda,\;\sigma,\;\upsilon$ and $s$,
if the following \emph{admissibility conditions} are satisfied.\newline
{\bf 1.} The permutation $\varpi$ is determined by $\Lambda\;\sigma,\;\upsilon$ and $s$ (See Definition \ref{determined-varpi}).\newline
{\bf 2.} The block inequalities (\ref{dart:blocks}) for $\varpi$ are satisfied.
\end{definition}

\begin{theorem}\label{general-blocks}
Let $\Lambda$ be a minimal pizza on a normally embedded H\"older triangle $T=T(\gamma_1,\gamma_2)$ associated
with a non-negative Lipschitz function $f(x)$ on $T$.
Let $m$ and $L$ be the numbers of maximum zones and coherent pizza slices of $\Lambda$, respectively, and let $(\sigma,\;\upsilon,\; s)$,
where $\sigma$ is a permutation of the set $[m]=\{1,\ldots,m\}$, $\upsilon$ is a permutation of the set $[L]=\{1,\ldots,L\}$ and $s:[L]\to\{+,-\}$
is a sign function on $[L]$, be an allowable triple (see Definition \ref{def:allowable}).
If $\check\Lambda=\{\check T_k\}_{k=1}^{\mathcal N}$, where $\check T_k=T(\check\lambda_{k-1},\check\lambda_k)$, is a twin pre-pizza on $T$
associated with $f$ and $\varpi$ is an admissible permutation of the set $[\mathcal N]=\{0,\ldots,\mathcal N-1\}$,
then there exists a unique, up to outer Lipschitz equivalence, normal pair $(T,T')$ of H\"older triangles,
such that the distance function $dist(x,T')$ on $T$ is Lipschitz contact equivalent to $f$, the permutation $\sigma$ is the characteristic permutation of the pair $(T,T')$, the permutation
$\upsilon$ and sign function $s$ are defined by the characteristic correspondence $\tau$ of the pair $(T,T')$ (see Definition \ref{upsilon}) and
$\varpi$ is associated with $\sigma$ and $\tau$ as in Definition \ref{def:varpi}.
\end{theorem}

\begin{proof}
The proof is similar to the proof of Theorem \ref{transverse-blocks}, except Step 2 where coherent pizza slices of $\check\Lambda'$ are defined.
We may assume that $T=T_\beta$ is a standard $\beta$-H\"older triangle (\ref{Formula:Standard Holder triangle}) in $\R^2_{uv}$.
In particular, the arcs $\check\lambda_k$ in $T$ are germs at the origin of the graphs $\{u\ge 0,\,v=\check\lambda_k(u)\}$,
where $\check\lambda_k(u)\ge 0$ are Lipschitz functions, and each H\"older triangle $\check T_k=T(\check\lambda_{k-1},\check\lambda_k)$,
for $k=1,\ldots,\mathcal N-1$, is a $\check\beta_k$-H\"older triangle.
It follows from the non-archimedean property of the tangency order that
$tord(\check\lambda_k,\check\lambda_l)=\min_{i:k<i<l}(tord(\check\lambda_k,\check\lambda_i),tord(\check\lambda_i,\check\lambda_l))$.
We are going to construct a normally embedded $\beta$-H\"older triangle $T'$ with a
twin pre-pizza $\check\Lambda'$ on $T'$ associated with the distance function $g(x')=dist(x',T)$.

{\bf Step 1.}
We define a family of arcs $\{\check\lambda'_{k'}\subset\R^3_{u,v,z}\}$ for $k'=\varpi(k)$ as the germs
$\{(u,v)\in\check\lambda_k,\, z=u^{\check q_k}\}$, where $\check q_k=ord_{\check\lambda_k} f$.
They will be the arcs of a twin pre-pizza $\check\Lambda'$ on $T'$ associated with $g$.
Let $\check q'_{\varpi(k)}=\check q_k$ and $\check\nu'_{\varpi(k)}=\check\nu_k$ for $k=0,\ldots,\mathcal N-1$.
Note that $\check\nu_k=\check q_k$ when an arc $\check\lambda_k$ belongs to a transverse pizza zone
of $\check\Lambda$.
This implies that $\check\nu'_{k'}=\check q'_{k'}$ for $k'=\varpi(k)$ in that case.
The arcs $\gamma'_1=\check\lambda'_0$ and $\gamma'_2=\check\lambda'_{\mathcal N-1}$ will be the boundary arcs of $T'$.
Let $\check\beta'_{k'}=tord(\check\lambda'_{k'-1},\check\lambda'_{k'})$ for $k'=1,\ldots,\mathcal N-1$.
Pizza slices $\check T'_{k'}$ of $\check\Lambda'$ will be $\check\beta'_{k'}$-H\"older triangles.

From the non-archimedean property of the tangency order we have, for $j=\varpi(i_1)$ and $k=\varpi(i_2)$, isometry
\begin{equation}\label{tord-lambda-prime}
tord(\check\lambda'_j,\check\lambda'_k)=tord(\check\lambda_{i_1},\check\lambda_{i_2}),
\end{equation}
same as (\ref{tord-lambda}) in Step 1 of the proof of Theorem \ref{transverse-blocks}.
Since $T$ is normally embedded and the arcs $\check\lambda_i$ satisfy the block inequalities (\ref{dart:blocks}), the arcs $\check\lambda'_j$ satisfy combinatorial normal
embedding inequalities (see Definition \ref{combinatorialLNE}) necessary for the H\"older triangle $T'$ to be normally embedded:
\begin{equation}\label{dart-LNE}
tord(\check\lambda'_j,\check\lambda'_l)=\min_k\left(tord(\check\lambda'_j,\check\lambda'_k),tord(\check\lambda'_k,\check\lambda'_l)\right)\;\text{for}\;j<k<l,
\end{equation}
same as (\ref{check-LNE}) in Step 1 of the proof of Theorem \ref{transverse-blocks}.

{\bf Step 2.} For each primary pair of indices $(k-1,k)$ of $\check\Lambda$ corresponding to a coherent pizza slice $\check T_k$ of
$\check\Lambda$, we define a $\check\beta_k$-H\"older triangle as follows.\newline
Consider the standard pizza slice
\begin{equation}\label{standard}
z=\psi_{\check\beta_k,\check q_{k-1},\check q_k,\check\mu_k}(u,v)
\end{equation}
(see Definition \ref{pizza_slice_standard}) on the standard $\check\beta_k$-H\"older triangle $T_{\check\beta_k}$,
then translate its graph in $\R^3_{u,v,z}$ to a graph $T(\check\lambda'_{\varpi(k-1)},\check\lambda'_{\varpi(k)})$
of a Lipschitz function over $\check T_k$ by a bi-Lipschitz
homeomorphism $H_k$ in $\R^2_{u,v}$ preserving the $u$-coordinate, such that $H_k(\{v\equiv 0\})=\check\lambda_{k-1}$ and $H_k(\{v=u^{\check\beta_k}\})=\check\lambda_k$.
It follows from Proposition \ref{prop:varpi1} that $\varpi(k-1)$ and $\varpi(k)$ are consecutive indices,
either $\varpi(k-1)=\varpi(k)-1$ or $\varpi(k-1)=\varpi(k)+1$.
The H\"older triangle $T(\check\lambda'_{k'-1},\check\lambda'_{k'})$,  where $k'=\max(\varpi(k-1),\varpi(k))$,
will be a coherent pizza slice $\check T'_{k'}$ of $\check\Lambda'$,
with the same width function $\mu(q)$ as the standard pizza slice (\ref{standard}),
and $(k'-1,k')$ will be a primary pair of indices of $\check\Lambda'$.

{\bf Step 3.} If $k'-1$ and $k'$ are consecutive indices in $[\mathcal N]$ such that $(k'-1,k')$ is not
a primary pair of indices of $\check\Lambda'$ defined in Step 2,
then $(k'-1,k')$ will be a secondary pair of indices of $\check\Lambda'$.
The arcs $\check\lambda'_{k'-1}$ and $\check\lambda'_{k'}$ are already defined in Step 1.
We define two arcs $\theta^+_{k'-1}=\{(u,v,z)\in\check\lambda'_{k'-1},\;w_{k'}=u^{\check\beta'_{k'}}\}$ and
$\theta^-_{k'}=\{(u,v,z)\in\check\lambda'_{k'},\;w_{k'}=u^{\check\beta'_{k'}}\}$ in $\R^4_{u,v,z,w_{k'}}$, then
consider the $(\check q'_{k'-1},\check q'_{k'},\beta'_{k'})$-model $(T_{\beta'_{k'}},T'_{\beta'_{k'}})$ (see Definition \ref{model})
and define a map $h_{k'}:T'_{\beta'_{k'}}\to \R^4_{u,v,z,w_{k'}}\subset\R^{\mathcal N+2}_{u,v,z,\mathbf w}$, where $\mathbf w=(w_1,\ldots,w_{\mathcal N-1})$,
as in Step 2 of the proof of Theorem \ref{transverse-blocks}.

Let $\check T'_{k'}=T(\check\lambda'_{k'-1},\check\lambda'_{k'})=h_{k'}(T'_{\beta'_{k'}})\subset \R^{\mathcal N+2}_{u,v,z,\mathbf w}$.
There are three possible cases (see Proposition \ref{prop:varpi2}).

{\bf (A)} If $\check q'_{k'-1}=\check q'_{k'}=\beta'_{k'}$, then the arcs $\check\lambda'_{k'-1}$ and $\check\lambda'_{k'}$ will be twin arcs of
$\check\Lambda'$.

{\bf (B)} If $\check q'_{k'-1}\ne\check q'_{k'}$ and $\min(\check q'_{k'-1},\check q'_{k'})=\beta'_{k'}$,
then $\check T'_{k'}$ will be a transverse pizza slice of $\check\Lambda'$.

{\bf (C)} If $\min(\check q'_{k'-1},\check q'_{k'})>\beta'_{k'}$, then $\check T'_{k'}$ will be
the union of two transverse pizza slices, and $(k'-1,k')$ will be a gap pair of indices of $\check\Lambda'$.

{\bf Step 4.} Let $T'=\bigcup_{k'=1}^{\mathcal N} \check T'_{k'}$ be the union of all
 H\"older triangles defined in Steps 2 and 3.
 To prove that $T'$ is normally embedded, we are going to show that it is combinatorially normally embedded (see Definition \ref{combinatorialLNE}).

As in Step 3 of the proof of Theorem \ref{transverse-blocks}, we show first that each partial H\"older triangle
$\check T'_{k'}=T(\check\lambda'_{k'-1},\check\lambda'_{k'})$ is normally embedded.
For a primary pair of indices $(k'-1,k')$ of $\check\Lambda'$, the triangle $\check T'_{k'}$
is normally embedded as the graph of a Lipschitz function.
For a secondary pair of indices $(k'-1,k')$ of $\check\Lambda'$, the triangle $\check T'_{k'}$
is constructed from the $(\check q'_{k'-1},\check q'_{k'},\beta'_{k'})$-model $(T_{\beta'_{k'}},T'_{\beta'_{k'}})$ (see Definition \ref{model}),
and the same arguments as in Step 3 of the proof of Theorem \ref{transverse-blocks} show that it is normally embedded.

Next, we have to show that any two H\"older triangles $\check T'_{k'}$ and $\check T'_{l'}$, where $k'\ne l'$, are transverse.
If both pairs of indices $(k'-1,k')$ and $(l'-1,l')$ are primary, then the two H\"older triangles are transverse as the graphs of Lipschitz functions
on two subtriangles $\check T_k$ and $\check T_l$ of $T$, where $k\ne l$.\newline
If one of the two pairs of indices, say $(k'-1,k')$, is secondary and another one is primary, the proof is similar to the argument in
Step 3 of the proof of Theorem \ref{transverse-blocks}, as the variable $w_{k'}$, which is non-zero on $\check T'_{k'}$,
vanishes on the graph of a Lipschitz function $\check T'_{l'}$.\newline
Finally, if both pairs of indices are secondary, the proof is exactly the same as in Step 3 of Theorem \ref{transverse-blocks}.

{\bf Step 5.} We are going to show that the twin pre-pizza on $T$ associated with the distance function $\tilde f(x)=dist(x,T')$
is combinatorially equivalent to the twin pre-pizza $\check\Lambda$ on $T$ associated with $f$.

If $\check T_k$ is a coherent pizza slice of $\check\Lambda$, then, by construction in Step 2 of this proof,
the corresponding sub-triangle $\check T'_{k'}$ of $\lambda'$, where $k'=\max(\varpi(k-1),\varpi(k))$, is a graph of a Lipschitz function on $\check T_k$
Lipschitz contact equivalent to $f|_{\check T_k}$.
Thus the pair $(\check T_k,\check T'_{k'})$ is outer Lipschitz equivalent to the pair $(\check T_k, graph(f|_{\check T_k})$.
Let us show that $\tilde f|_{\check T_k}$ is Lipschitz contact equivalent to $f|_{\check T_k}$.

For any arc $\eta'$ of a coherent pizza slice $\check T'_{l'}$ of $\check\Lambda'$, where $l'=\max(\varpi(l-1),\varpi(l))$ and $l\ne k$,
the order of tangency of $\eta'$ with any arc $\eta$ of $\check T_k$
cannot exceed $ord_\eta f$, as $\check T'_{k'}$ and $\check T'_{l'}$
are graphs of functions contact equivalent to restrictions of the Lipschitz function $f$ on $T$ to distinct pizza slices
$\check T_k$ and $\check T_l$ of $\check\Lambda$.
The same arguments as in Step 4 of the proof of Theorem \ref{transverse-blocks} show that
the order of tangency of an arc $\eta$ of $\check T_k$ with any arc $\eta'$ of a non-coherent triangle $T'_{\ell'}$ of $\check\Lambda'$
cannot exceed  $ord_{\eta} f$,
as the variable $w_{l'}$ involved in construction of $\check T'_{l'}$ in Step 3 of this proof vanishes on $\check T'_{k'}$.\newline
Thus $f|_{\check T_k}$ is contact equivalent to restriction of $\tilde f$ to $T_k$.

The same arguments as in Step 4 of the proof of Theorem \ref{transverse-blocks} show that any non-coherent
triangle $\check T_k$ of $\check\Lambda$ is either a transverse pizza slice for the distance function $\tilde f$
or the union of two transverse pizza slices with a common minimum zone.
In both cases, $\tilde f$ restricted to $\check T_k$ is completely determined by the exponent
$\check\beta_k$ of $\check T_k$ and the orders $\check q_{k-1}$ and $\check q_k$ of $f$ on its boundary arcs.
This implies that $\tilde f$ restricted to $\check T_k$ is contact equivalent to $f$ on all triangles $\check T_k$ of $\check\Lambda$.

This completes the proof of Theorem \ref{general-blocks}.
\end{proof}

\end{document}